\documentclass[12pt]{amsart}
\usepackage{latexsym, enumerate}
\usepackage{amssymb, tikz, mathrsfs, inputenc, verbatim, bbm, tikz, pgffor, epsfig, mathrsfs, xcolor}

\usepackage{tikz}
\usetikzlibrary{shapes.geometric,arrows}

\tikzstyle{box} = [rectangle,text centered, draw=black]
\tikzstyle{arrow} = [thick, ->, >=stealth]

\usepackage{color}

\usepackage{amsthm}
\theoremstyle{plain}
\newtheorem*{theorem*}{Theorem}

\renewcommand{\epsilon}{\varepsilon}

\newcommand{\N}{{\mathbb N}}
\newcommand{\R}{{\mathbb R}}

\renewcommand{\phi}{\varphi}

\newcommand{\la}{\lambda}

\numberwithin{equation}{section}

\newtheorem{theorem}{Theorem}[section]

\newtheorem{proposition}[theorem]{Proposition}
\newtheorem{corollary}[theorem]{Corollary}
\newtheorem{lemma}[theorem]{Lemma}

\theoremstyle{definition}

\theoremstyle{remark}

\theoremstyle{assumption}

\title[Restriction of Laplace-Beltrami eigenfunctions]{Restriction of Laplace-Beltrami eigenfunctions to arbitrary sets on manifolds}
\author{Suresh Eswarathasan}
\address{Department of Mathematics and Statistics, Dalhousie University, 6316 Coburg Road, Halifax, Nova Scotia, Canada B3H 4R2}
\email{sr766936@dal.ca}
\author{Malabika Pramanik}
\address{University of British Columbia, Department of Mathematics, 1984 Mathematics Road, Vancouver, BC V6T 1Z2}
\email{malabika@math.ubc.ca}

\date{\today}

\begin{document}
{\allowdisplaybreaks

\begin{abstract}
Given a compact Riemannian manifold $(M, g)$ without boundary, we estimate the Lebesgue norm of Laplace-Beltrami eigenfunctions when restricted to a wide variety of subsets $\Gamma$ of $M$.  The sets $\Gamma$ that we consider are Borel measurable, Lebesgue-null but otherwise arbitrary with positive Hausdorff dimension. 
\vskip0.1in 
\noindent Our estimates are based on Frostman-type ball growth conditions for measures supported on $\Gamma$. For large Lebesgue exponents $p$, these estimates provide a natural generalization of $L^p$ bounds for eigenfunctions restricted to submanifolds, previously obtained in \cite{Ho68, Ho71, Sog88, BGT07}.  
Under an additional measure-theoretic assumption on $\Gamma$, the estimates are shown to be sharp in this range. As evidence of the genericity of the sharp estimates, we provide a large family of random, Cantor-type sets that are not submanifolds, where the above-mentioned sharp bounds hold almost surely. 
\end{abstract}

\maketitle

\tableofcontents

\section{Introduction}  
\noindent The study of eigenfunctions of Laplacians lies at the interface of several areas of mathematics, including analysis, geometry, mathematical physics and number theory.  These special functions arise in physics and in partial differential equations as modes of periodic vibration of drums and membranes. In quantum mechanics, they represent the stationary energy states of a free quantum particle on a Riemannian manifold. 
\vskip0.1in
\noindent Let $(M,g)$ denote a compact, connected, $n$-dimensional Riemannian manifold without boundary. The ubiquitous (positive) Laplace-Beltrami operator on $M$, denoted $-\Delta_g$, is the primary focus of this article.  It is well-known \cite[Chapter 3]{So} that the spectrum of this operator is non-negative and discrete. Let us denote its eigenvalues by $\{ \lambda_j^2: j \geq 0\}$, and the corresponding eigenspaces by $\mathbb E_j$. Without loss of generality, the positive square roots of the distinct eigenvalues can be arranged in increasing order, with \[0 = \lambda_0 < \lambda_1 < \lambda_2 < \cdots \lambda_j < \cdots  \rightarrow \infty.\]  It is a standard fact \cite[Chapter 3]{So} that each $\mathbb E_j$ is finite-dimensional. Further, the space $L^2(M, dV_g)$ (of functions on $M$ that are square-integrable with respect to the canonical volume measure $dV_g$) admits an orthogonal decomposition in terms of $\mathbb E_j$:\[ L^2(M, dV_g) = \bigoplus_{j=0}^{\infty} \mathbb E_j. \] 
\vskip0.1in
\noindent One of the fundamental questions surrounding Laplace-Beltrami eigenfunctions targets their concentration phenomena, via high-energy asymptotics or high-frequency behaviour. There are many avenues for this study, 
as exemplified in the seminal work of Shnirelman \cite{Shnirelman}, Zelditch \cite{Zelditch}, Colin de V\`erdiere \cite{CdV}, G\'erard and Leichtnam \cite{GL}, Zelditch and Zworski \cite{ZZ}, Helffer, Martinez and Robert \cite{HMR}, Rudnick and Sarnak \cite{{Rudnick-Sarnak94},{Sarnak}}, Lindenstrauss \cite{Lindenstrauss}, and Anantharaman \cite{Anantharaman}. One such approach involves studying the growth of the $L^p$ norms of these eigenfunctions. The contribution of this article lies in this category. Specifically, we describe the $L^2(M) \rightarrow L^p(\Gamma)$ mapping property of a certain spectral projector (according to the spectral decomposition above), where $\Gamma$ is a Lebesgue-null subset of $M$. In particular, $\Gamma$ does not enjoy any smooth structure, a point of departure from prior work where this feature was heavily exploited.
We begin by reviewing the current research landscape that will help place the main results Theorems \ref{mainthm-restriction-arbitrary}-\ref{mainthm-example} in context.  

\subsection{Literature review}
The Weyl law, itself a major topic in spectral theory, provides an $L^{\infty}$ bound on eigenfunctions on $M$ \cite{Ho68}. The first results that establish  $L^p$ eigenfunction bounds for $p < \infty$ are due to Sogge \cite{Sog88}. 
\begin{theorem}\cite{Sog88} \label{thm:Sogge}
Given any manifold $M$ as above and $p \in [2, \infty]$, there exists a constant $C = C(M,p) > 0$ such that the following inequality holds for all $\lambda \geq 1$: 
\begin{align} &\| \varphi_{\lambda} \|_{L^p(M)} \leq C (1 + \lambda)^{\delta(n, p)} \| \varphi_{\lambda} \|_{L^2(M)},  \text{ with } \label{Sogge-estimate}  \\   
\delta(n, p) &=  \left\{ 
\begin{aligned}
& \frac{n-1}{4} - \frac{n-1}{2p}, &\text{ if }  &2 \leq p \leq \frac{2(n+1)}{n-1}, \\
&\frac{n-1}{2} - \frac{n}{p}, &\text{ if }  &\frac{2(n+1)}{n-1} \leq p \leq \infty.
\end{aligned} \right\} \label{Sogge-exponent}
\end{align}
Here $\varphi_{\lambda}$ is any eigenfunction of $- \Delta_g$ corresponding to the eigenvalue $\lambda^2$. The bound is sharp for the $n$-dimensional unit sphere $M = \mathbb S^n$, equipped with the surface measure.  
\end{theorem}
\noindent Historically, an important motivation and source of inspiration for this line of investigation has been the Fourier restriction problem, which explores the behaviour of the Fourier transform when restricted to curved surfaces in Euclidean spaces. In fact the Stein-Tomas $L^2$ restriction theorem \cite{Tomas}, originating in Euclidean harmonic analysis, was a key ingredient in an early proof of Theorem \ref{thm:Sogge} for the sphere. Indeed, Theorem \ref{thm:Sogge} may be viewed as a form of discrete restriction on $M$ where the frequencies are given by the spectrum of the manifold, see for example \cite{Sogge86}.  Conversely, it is possible to recover the $L^2$ restriction theorem for the sphere from a spectral projection theorem such as Theorem \ref{thm:Sogge} applied to the $n$-dimensional flat torus. The lecture notes of Yung \cite[Section 2]{Yung-notes} contain a discussion of these implications.  
\vskip0.1in
\noindent Theorem  \ref{thm:Sogge} permits a number of independent proofs. For an argument that involves well-known oscillatory integral estimates of H\"ormander applied to the smooth spectral projector (denoted $\rho(\sqrt{-\Delta_g} - \lambda)$), we refer the reader to the treatise \cite{So}. The semi-classical approach of Koch, Tataru and Zworski \cite{KTZ07} has also yielded many powerful applications.
\vskip0.1in
\noindent Finer information on eigenfunction growth may be obtained through $L^p$ bounds on $\varphi_{\lambda}$ when restricted to smooth submanifolds of $M$.  One expects $\varphi_{\lambda}$ to assume large values on small sets. Thus its $L^p$-norm on a Lebesgue-null set such as a submanifold, if meaningful, is typically expected to be larger in comparison with the $L^p$ norm taken over the entire manifold $M$, as given by Theorem \ref{thm:Sogge}. The first step in this direction is due to Reznikov \cite{Rez10}, who studied eigenfunction restriction phenomena on hyperbolic surfaces via representation theoretic tools. The most general results to date on restricted norms of Laplace eigenfunctions are by Burq, G\'erard and Tzvetkov \cite{BGT07}, and independently by Hu \cite{Hu09}. The work of Tacy \cite{Tacy} has extended these results to the setting of a semi-classical pseudo-differential operator (not merely the Laplacian) on a Riemannian manifold, while removing logarithmic losses at a critical threshold.  Another particular endpoint result is due to Chen and Sogge \cite{CS14}. We have summarized below the currently known best eigenfunction restriction estimates for a general manifold, combined from this body of work and for easy referencing later.
\begin{theorem}{\cite{{BGT07},{Hu09}, {Tacy}}}\label{thm:smooth_restriction}
Let $\Sigma \subset M$ be a smooth $d$-dimensional submanifold of $M$, equipped with the canonical measure $d \sigma$ that is naturally obtained from the metric $g$.  Then for each $p \in [2, \infty]$, there exists a constant $C = C(M, \Sigma, p) >0$ such that for any $\lambda \geq 1$ and any Laplace eigenfunction $\varphi_{\lambda}$ associated with the eigenvalue $\lambda^2$, the following estimate holds:  
\begin{equation}  \| \varphi_{\lambda} \|_{L^p(\Sigma, d\sigma)} \leq C \, (1 + \lambda)^{\delta(n,d,p)} \| \varphi_{\lambda} \|_{L^2(M, dV_g)}. \label{eigenfunction-restriction-submanifold} 
\end{equation} 
The exponent $\delta(n, d, p)$ admits a multi-part description. Specifically, 
\begin{equation}
\delta(n, n-1, p) = \left\{ 
\begin{aligned}  
&\frac{n-1}{4} - \frac{n-2}{2p}, & \text{ for } & 2 \leq p \leq \frac{2n}{n-1},\\
& \frac{n-1}{2} - \frac{n-1}{p}, & \text{ for }  & \frac{2n}{n-1} \leq p \leq \infty.
\end{aligned} \right\} \label{BGT-exponent-1} 
\end{equation} 
For $d \ne n-1$, 
\begin{equation} 
\delta(n,d,p) = \frac{n-1}{2} - \frac{d}{p}, \quad \text{ for }  2 \leq p \leq \infty \text{ and } (d, p) \ne (n-2, 2). \label{dlessthann-3}
\end{equation}
For $(d, p) = (n-2, 2)$, the exponent $\delta(n,d,p)$ is still given by \eqref{dlessthann-3}; however, there is an additional logarithmic factor $\log^{1/2}(\lambda)$ appearing in the right hand side of inequality \eqref{eigenfunction-restriction-submanifold}.
\end{theorem}
\noindent The proofs in \cite{BGT07} and \cite{CS14} use a delicate analysis of oscillatory representations of the smoothed spectral projector $\chi(\sqrt{-\Delta_g}-\lambda)$ restricted to submanifolds $\Sigma$, combined with refined estimates influenced by the considered geometry.  Alternatively, \cite{Hu09} uses general mapping properties for Fourier integral operators with prescribed degenerate canonical relations to obtain bounds for the oscillatory integral operators in question. There are several recurrent features in these proofs; namely, stationary phase methods, arguments involving integration by parts and operator-theoretic convolution inequalities. This methodology heavily relies on the fact that the underlying measures are induced by Lebesgue, which in turn is a consequence of $M$ and $\Sigma$ being smooth manifolds. The present article explores the accessibility of this machinery in the absence of smoothness, and aims to find working substitutes when such methods are unavailable.  This leads to a discussion of our main results.  

\subsection{Main results}
There is a common theme in Theorems \ref{thm:Sogge} and \ref{thm:smooth_restriction} above; namely, the left hand side of both inequalities \eqref{Sogge-estimate} and \eqref{eigenfunction-restriction-submanifold} involves the $L^p$ norm of an eigenfunction $\varphi_{\lambda}$ but over different submanifolds of $M$ (including $M$ itself), and with respect to natural measures on these submanifolds. An interesting feature of the exponents $\delta(n,p)$ and $\delta(n, d,p)$ 
is that for large $p$, they are both of the form $(n-1)/2 - \alpha/p$, where 
\begin{align} \alpha &= \text{dimension of the space on which the $L^p$ norm of $\varphi_{\lambda}$ is measured} \nonumber \\ &= \left\{ \begin{aligned} &\text{dim}(M) = n &\text{ in Theorem }\ref{thm:Sogge}, \\ 
&\text{dim}(\Sigma) = d &\text{ in Theorem } \ref{thm:smooth_restriction}.  
 \end{aligned} \right\} \label{form-alpha}  \end{align} 
In view of this commonality in \eqref{Sogge-exponent}, \eqref{BGT-exponent-1} and \eqref{dlessthann-3}, we pose the following question: 
\begin{itemize} 
\item 
{\em{Given an arbitrary Borel set $\Gamma \subseteq \Sigma$, does there exist a measure $\mu$ supported on $\Sigma$ with respect to which we can estimate the growth of the  eigenfunctions $\varphi_{\lambda}$? }}
\end{itemize} 
The nontrivial situation arises when $\Gamma$ is Lebesgue-null, i.e., $\mu$ is singular with respect to the canonical measure on $\Sigma$. The optimal scenario would be to obtain bounds that reflect the dimensionality of the set $\Gamma$ in the same way that Theorems \ref{thm:Sogge} and \ref{thm:smooth_restriction} do.  We answer this by presenting the main results of our article, namely Theorems \ref{mainthm-restriction-arbitrary}, \ref{mainthm-restriction-ball}, \ref{mainthm-sharpness},  \ref{mainthm-example} and Corollary \ref{maincor-example}. 
\subsubsection{Restriction to arbitrary Borel sets}   \label{arbitrary-subsection} 
\noindent The general set-up is as follows. Given a compact $n$-dimensional Riemannian manifold $(M, g)$, let $\Sigma \subseteq M$ be a smooth embedded (sub)manifold of dimension $1 \leq d \leq n$, equipped with the restricted Riemannian metric naturally endowed by $g$. Let $(U, \mathfrak{u})$ be a local coordinate chart on $\Sigma$, where $U \subseteq \mathbb R^d$ is an open set containing $[0,1]^d$ and $\mathfrak{u}: U \rightarrow \mathfrak{u}(U) \hookrightarrow \Sigma$ is a smooth embedding.
\vskip0.1in
\noindent Given any $\epsilon \in [0, 1)$, let $E \subseteq [0, 1]^d$ be an arbitrary Borel set of Hausdorff dimension $\text{dim}_{\mathbb H}(E) = d(1 - \epsilon)$. We refer the reader to the classical textbook of Mattila \cite[Chapter 4]{M95} for the definitions and properties of Hausdorff dimension of sets in Euclidean spaces. The Borel set $E \subseteq [0,1]^d$ generates a corresponding Borel set $\Gamma = \Gamma[E]$ in $\Sigma$ by setting $\Gamma := \mathfrak{u}(E)$. Conversely, every Borel subset $\Gamma$ in $\mathfrak{u}([0,1]^d) \subseteq \Sigma$ can be identified with a set $E = \mathfrak{u}^{-1}(\Gamma)\subseteq [0,1]^d$. Similarly, any measure $\nu$ supported on $\Gamma$ corresponds with a measure $\mu = \mathfrak{u}^{\ast}\nu$ on $E$ via the pull-back $\mathfrak{u}^{\ast}$; i.e., 
\begin{equation}  \label{mu-nu-relation} 
\nu(A) := \mu(\mathfrak{u}^{-1}(A)) \quad \text{for all Borel sets $A \subseteq \Sigma$}. 
\end{equation} 
The converse is also true; any Borel measure $\mu$ generates another measure $\nu$ on $\Gamma$ through its push-forward, given by the same relation \eqref{mu-nu-relation}. 
Since $\mathfrak{u}$ is a diffeomorphism, and thus bi-Lipschitz, it preserves Hausdorff dimension \cite[Corollary 2.4]{F90}; hence $\text{dim}_{\mathbb H}(\Gamma) = \text{dim}_{\mathbb H}(E) = d(1- \epsilon)$. Let us define our critical exponent  
\begin{equation} \label{def-p0}
p_0 = p_0(n, d, \epsilon) := \frac{4d(1- \epsilon)}{n-1}. 
\end{equation} 
Our first result specifies a family of restricted eigenfunction estimates for every such set $\Gamma$. 
\begin{theorem} \label{mainthm-restriction-arbitrary}
Let $M$, $\Sigma$, $\Gamma$ and $p_0$ be as above. Then for every $\kappa > 0$ sufficiently small, there exists a probability measure $\nu^{(\kappa)}$ supported on $\Gamma$ such that for all $\lambda \geq 1$ and all $p \in [2, \infty]$, we have the eigenfunction estimate
\begin{equation}
\|\varphi_{\lambda}\|_{L^p(\Gamma, \nu^{(\kappa)})} \leq C_{\kappa, p} (1 + \lambda)^{\theta_p} \Theta(\lambda; \kappa, p) \, \|\varphi_{\lambda}\|_{L^2(M, dV_g)}.  \label{eigenfunction-restriction-estimate} 
\end{equation}
Here $\varphi_{\lambda}$ denotes any $L^2$-eigenfunction associated with the eigenvalue $\lambda^2$ for the Laplace-Beltrami operator $-\Delta_g$ on $M$. For $p_0 > 2$, the exponent $\theta_p$ is given by  
\begin{equation}
\theta_p = \theta_p(n,d, \epsilon) :=
\left\{ \begin{aligned}
&\frac{n-1}{4} & \text{ if } &2 \leq p \leq p_0 - \frac{4 \kappa}{n-1}, \\
&\frac{n-1}{2} - \frac{d(1-\epsilon)}{p} & \text{ if }  & p \geq p_0 - \frac{4 \kappa}{n-1}. 
\end{aligned} \right\} \label{our-exponent} 
\end{equation}
The function $\Theta$ represents a power loss of $\kappa/p$ beyond the critical threshold: 
\[ \Theta(\lambda; \kappa, p) := \left\{ \begin{aligned} &1 &\text{ if } &2 \leq p \leq p_0 - \frac{4 \kappa}{n-1}, \\ &(1 + \lambda)^{\frac{\kappa}{p}} \bigl( \log \lambda \bigr)^{\frac{1}{p}}  & \text{ if }  &p = p_0 - \frac{4 \kappa}{n-1}, \\ &(1 + \lambda)^{\frac{\kappa}{p}}  & \text{ if }  &p > p_0 - \frac{4\kappa}{n-1}.  \end{aligned} \right\} \] 
For $p_0 \leq 2$ and  $2 \leq p \leq \infty$, we set \[ \theta_p  = \theta_p(n,d, \epsilon) := \frac{n-1}{2} - \frac{d(1- \epsilon)}{p} \quad \text{ and } \quad \Theta(\lambda; \kappa, p) = (1 + \lambda)^{\frac{\kappa}{p}}. \] 
The positive constant $C_{\kappa, p}$ in \eqref{eigenfunction-restriction-estimate} may depend on $n, d, \epsilon, \kappa$ and $p$ but is independent of $\lambda$.  The probability measure $\nu^{(\kappa)}$ in \eqref{eigenfunction-restriction-estimate}, given by Frostman's lemma, obeys the following volume growth condition: there exists $C_{\kappa} > 0$ such that for all $x \in \Sigma$ and $r > 0$
\begin{equation}  \nu^{(\kappa)}\bigl(B_g(x;r) \bigr) \leq C_{\kappa} r^{d(1 - \epsilon)- \kappa}, \label{nu-vol-growth} \end{equation}  where $B_g(x;r) \subseteq M$ denotes the Riemannian ball centred at $x$ of radius $r$.
\end{theorem}

\noindent {\bf{Remarks: }} 
\begin{enumerate}[1.]
\item The proof of Theorem \ref{mainthm-restriction-arbitrary} appears in Section \ref{proofs-mainthms-section}. 
\vskip0.1in
\item The estimate \eqref{eigenfunction-restriction-estimate} is sharp for $p \geq \max(2, p_0)$, except possibly for the infinitesimal blow-up factor of $(1+\lambda)^{\kappa/p}$. More precisely, for every $\epsilon \in [0,1)$, $d \leq n$ and $p \geq \max(2, p_0)$, the bound in \eqref{eigenfunction-restriction-estimate} is realized, ignoring sub-polynomial losses, for certain sets of dimension $d(1 - \epsilon)$ in $M = \mathbb S^n$. The exact sharpness statements appear in Corollary \ref{maincor-sharpness} and \ref{maincor-example} below.  The estimate is not sharp for $2 \leq p < p_0$, when $p_0 > 2$. This is an artifact of our proof strategy. See remark \ref{remark-nonoptimal} following Theorem \ref{mainthm-restriction-ball} for a discussion of the non-optimality for small $p$.  
\vskip0.1in
\item For an arbitrary Borel set $\Gamma$, the information available about measures supported on it is limited. As a result, the measure $\nu^{(\kappa)}$ that realizes \eqref{eigenfunction-restriction-estimate} varies with $\kappa$ in general. Thus we are able to prove \eqref{eigenfunction-restriction-estimate} only for all $\kappa > 0$ and not for $\kappa = 0$. On the other hand, if $\Gamma$ is a (sub)-manifold of $M$, it follows from \cite{{So}, {BGT07}, {Hu09}, {Tacy}} that there is a natural Lebesgue-induced measure on $\Gamma$ for which \eqref{eigenfunction-restriction-estimate} does hold with $\kappa = 0$. Our claim in this paper is that such an improvement holds in a generic sense. In Theorem \ref{mainthm-example} and Corollary \ref{maincor-example} below, we provide a large class of sparse subsets $\Gamma$ that are not submanifolds, each supporting {\em{a single probability measure $\nu$ that obeys \eqref{nu-vol-growth} for all $\kappa > 0$}}, even though $C_{\kappa} \nearrow \infty$ as $\kappa \searrow 0$. For this measure $\nu$, we show that a stronger version of \eqref{eigenfunction-restriction-estimate} holds, with $\kappa = 0$. However, $\Theta$ is then replaced by a function of slow growth in the range $p \geq \max(2, p_0)$. A precise functional form for $\Theta$ that quantifies the infinitesimal blowup is provided; see \eqref{generic-sharp} and \eqref{our-Phi}. \label{sharpness-remark}
\vskip0.1in
\item For $p \geq \max(2, p_0)$, the exponent $\theta_p$ in Theorem \ref{mainthm-restriction-arbitrary} is of the same form alluded to in \eqref{form-alpha}, namely $\theta_p = (n-1)/2 - \alpha/p$ with $\alpha = d(1-\epsilon) = \text{dim}_{\mathbb H}(\Gamma)$.  Thus our result may be viewed as a natural interpolation between the global estimates in \cite{Sog88} and the smooth restriction estimates in \cite{BGT07}, bridging the estimates across a family of arbitrary Borel sets with continuously varying Hausdorff dimensions.   
\vskip0.1in 
\item To the best of our knowledge, Theorem \ref{mainthm-restriction-arbitrary} is the first result of its kind in several distinct categories. First, it offers eigenfunction bounds restricted to {\em{any Borel subset of positive Hausdorff dimension}}, for every manifold $M$ and every smooth submanifold $\Sigma$ therein. Second, even for integers $m$, our result produces new sets of dimension $m$, for example with $(n, d, \epsilon) = (2, 2, 1/2)$, that are not necessarily contained in any $m$-dimensional submanifold, and yet capture the same eigenfunction growth bounds as smooth submanifolds of the same dimension, up to sub-polynomial losses. Third, when $\epsilon = 0$, our result provides examples of singular measures supported on submanifolds with respect to which the eigenfunctions obey the same $L^p$ growth bounds, up to any prescribed $\kappa$-loss, as with the induced Lebesgue measure on the same submanifold.  
\vskip0.1in
\item As in \cite{{BGT07}, {Hu09}}, the proof of Theorem \ref{mainthm-restriction-arbitrary} yields estimates not merely for the eigenfunctions $\varphi_{\lambda}$, but also for the smoothed spectral projector on $M$.  The sharpness statement for $p \geq p_0$ continues to hold for such operators. 
\end{enumerate} 
\vskip0.1in
\subsubsection{Ball growth conditions imply restricted eigenfunction estimates} Theorem \ref{mainthm-restriction-arbitrary} will be derived as a consequence of a more general phenomenon, whereby {\em{any}} probability measure $\nu$ on  $\Gamma \subseteq \mathfrak{u}([0,1]^d)$ obeying a ball growth condition gives rise to an eigenfunction growth estimate restricted to $\Gamma$. Theorem \ref{mainthm-restriction-ball} in this subsection specifies this correspondence. As preparation for it, we introduce a class of functions that will be used throughout this article:
\begin{equation} \label{functions-of-slow-growth} 
\mathfrak F := \left\{ \Psi: [1, \infty) \rightarrow [1, \infty)  \Biggl| \;  \begin{aligned} &\Psi \text{ is continuous and nondecreasing for large $t$,} \\ & \Psi(t) t^{- \kappa} \searrow 0  \text{ as } t \rightarrow \infty, \text{ for all } \kappa > 0. \end{aligned} \right\} 
\end{equation}    
In other words, functions in $\mathfrak F$ grow slower than any positive power of $t$; we will refer to members of $\mathfrak F$ as {\em{functions of slow growth}}.  
\vskip0.1in
\noindent As in Section \ref{arbitrary-subsection}, let $(M,g)$ denote a compact Riemannian manifold. Suppose that $\Sigma$ is a smooth submanifold of dimension $d$, equipped with a coordinate chart $(U, \mathfrak{u})$. 
\begin{theorem} \label{mainthm-restriction-ball} 
Let $\alpha \in (0, d]$ and $\Psi \in \mathfrak F$. 
Let $\Gamma \subseteq \mathfrak{u}([0,1]^d) \subseteq \Sigma$ be a Borel subset of $M$ endowed with a probability measure $\nu$ that obeys the estimate 
\begin{equation}  \sup \left\{  \frac{\nu(B_g(x;r))}{r^{\alpha} \Psi(1/r)} : x \in \Sigma, \; r > 0  \right\} := C_0 < \infty. \label{mu-ball-condition}  \end{equation} 
Then for every $p \in [2, \infty]$, there exists a constant $C = C(p, C_0)>0$ such that for all $\lambda \geq 1$, 
\begin{equation} \label{eigenfunction-growth-nu} 
\|\varphi_{\lambda}\|_{L^p(\Gamma, \nu)} \leq C (1 + \lambda)^{\vartheta_p} \overline{\Psi}_p(\lambda) \, \|\varphi_{\lambda}\|_{L^2(M, dV_g)}.
\end{equation}   
The exponent $\vartheta_p$ admits a description analogous to \eqref{our-exponent};  
\begin{equation}
\vartheta_p = \vartheta_p(n, \alpha) :=
\left\{ \begin{aligned}
&\frac{n-1}{4}, & \text{ if } &2 \leq p \leq \max (2, p_{\ast}), \\
&\frac{n-1}{2} - \frac{\alpha}{p}, & \text{ if }  & p \geq \max(2, p_{\ast}). 
\end{aligned} \right\}  \text{ where } p_{\ast} := \frac{4 \alpha}{n-1}. \label{our-exponent-2}
 \end{equation}
The function $\overline{\Psi}_p$ lies in the class of functions $\mathfrak F$ in \eqref{functions-of-slow-growth}, and is related to $\Psi$ by 
\[ \bigl[ \overline{\Psi}_p (\lambda) \bigr]^p := \Psi(\lambda) \times \begin{cases} 1 &\text{ if } p \ne p_{\ast}, \\ \log \lambda &\text{ if } p = p_{\ast}. \end{cases} \] 
\end{theorem} 

\noindent {\bf{Remarks: }} 
\begin{enumerate}[1.]
\item We prove Theorem \ref{mainthm-restriction-ball} in Section \ref{proofs-mainthms-section}. 
\vskip0.1in
\item Let us note that the exponent $p_{\ast}$ in \eqref{our-exponent-2} coincides with $p_0$ in \eqref{def-p0} when $\alpha = d(1 - \epsilon)$. 
\vskip0.1in
\item \label{interchange-ball} At the beginning of Section \ref{arbitrary-subsection}, we have observed that any set-measure pair $(\Gamma, \nu)$ on $\Sigma \subseteq M$ corresponds uniquely to a similar pair $(E, \mu)$ on $[0,1]^d$, via the coordinate chart $\mathfrak u$. Using this correspondence, the volume growth condition \eqref{mu-ball-condition} for $\nu$ is equivalent to an analogous one for $\mu$, possibly with a different constant $C_0$:
\begin{equation} 
 \sup \left\{  \frac{\mu(B(u;r))}{r^{\alpha} \Psi(1/r)} : u \in \mathbb R^n, \; r > 0  \right\} := C_0 < \infty, \label{mu-ball-condition-2} 
\end{equation}    
where $B(u;r)$ represents the Euclidean ball centred at $u$ of radius $r$. Henceforth, we will use these conditions interchangeably. The microlocal representation of the spectral projection operator involves Euclidean coordinates via the coordinate chart $\mathfrak u$ (see Section \ref{microlocal-preliminaries-section}), so for practical purposes we rely on \eqref{mu-ball-condition-2}. 
\vskip0.1in
\item  As in Theorem \ref{mainthm-restriction-arbitrary}, the estimate  \eqref{eigenfunction-growth-nu} is sharp for $p \geq \max(2, p_{\ast})$ up to a function of slow growth, but not otherwise. In Theorem \ref{mainthm-sharpness} below and under an additional restriction on $\nu$, we prove a lower bound on $||\varphi_{\lambda}||_{L^p(\Gamma, \nu)}$ that is comparable to the upper bound in \eqref{eigenfunction-growth-nu}, for $M = \mathbb S^n$ and an infinite sequence of eigenvalues $\lambda$. In Theorem \ref{mainthm-example}, we show that these measure-theoretic restrictions are generically satisfied, in the sense mentioned in Remark \ref{sharpness-remark} following Theorem \ref{mainthm-restriction-arbitrary}.  
\vskip0.1in 
\item \label{remark-nonoptimal} Our proof technique may be viewed as an adaptation of \cite[Theorem 1]{BGT07}, which itself is based on Young's convolution inequality and ignores oscillations inherent in the underlying spectral projection operator. While this yields sharp results in the full range $2 \leq p \leq \infty$ if $p_{\ast} \leq 2$, it fails to do so when $p_{\ast} > 2$ even for submanifolds $\Sigma$, as observed in \cite{BGT07}. We refer the interested reader to the discussion in \cite[p.\ 479]{BGT07}, where the authors explain why a strategy such as Young's inequality, that does not take advantage of the oscillations in the spectral projector, can only yield sharp eigenfunction estimates for large $p$ in a non-optimal range. Indeed the critical threshold of $p$ for this restricted range, as specified in \cite[inequality (6.4)]{BGT07}, is $4 \text{ dim}(\Sigma)/(\text{dim}(M) - 1))$, which corresponds exactly with our $p_{\ast} = 4 \alpha/(n-1)$. See also the remark following Proposition \ref{gen-Young-prop} on page \pageref{Young-prop-remark}.  
\vskip0.1in
\noindent In order to effectively harness the oscillation in the spectral projector, one would need to generalize \cite[Theorem 3]{BGT07} to the context of general Borel sets. At the moment we do not know whether Hausdorff dimensionality alone, as manifested in the form of the ball growth condition \eqref{mu-ball-condition}, is sufficient to obtain these more refined estimates.   
\end{enumerate} 
\subsubsection{Conditional sharpness of Theorems \ref{mainthm-restriction-arbitrary} and \ref{mainthm-restriction-ball}} 
The exponent $\vartheta_p$ in the estimate \eqref{eigenfunction-growth-nu} is sharp for $p \geq p_{\ast}$, in the following sense.  
\begin{theorem} \label{mainthm-sharpness} 
Suppose that  $M = \mathbb S^n$ is the $n$-dimensional unit sphere, and $\Sigma$ is any $d$-dimensional submanifold of $M$, $d \leq n$. Let $\Psi, \Xi \in \mathfrak F$, where $\mathfrak F$ is as in \eqref{functions-of-slow-growth}. 
\vskip0.1in 
\noindent Fix any $\alpha \in (0, d]$, and let $p_{\ast} = 4 \alpha/(n-1)$, as in \eqref{our-exponent-2}.  Assume that $\nu$ is a probability measure supported on $\Gamma \subseteq \mathfrak{u}([0,1]^d) \subseteq \Sigma$ obeying \eqref{mu-ball-condition}, and that 
\begin{equation} \label{AD-local}
\limsup_{r \rightarrow 0} \bigl[ r^{-\alpha} \; \Xi(1/r) \nu(B_g(y_0;r)) \bigr] =: \kappa_0 > 0 \quad \text{ for some } y_0 \in \Gamma. 
\end{equation} 
Then for every $p \geq \max(1, p_{\ast}/2)$, there exists a constant $c = c(\kappa_0, p) > 0$ and a sequence of $L^2$-normalized spherical harmonics $\{ \varphi_{\lambda_j} : j \geq 1 \}$ with $\lambda_j \nearrow \infty$ such that for all $\lambda_j$ sufficiently large,
\begin{equation}  \| \varphi_{\lambda_j} \|_{L^p(\Gamma, \nu)} \geq c (1 + \lambda_j)^{\varkappa_p} \bigl[ \Lambda_p(\lambda_j) \bigr]^{-1}, \text{ with } \varkappa_p := \alpha\Bigl( \frac{2}{p_{\ast}} - \frac{1}{p}\Bigr). \label{sphere-lower-bound}  \end{equation}  
The blowup factor $\Lambda_p \in \mathfrak F$ is related to $\Psi$ and $\Xi$ by 
\begin{equation}  \Lambda_p(\lambda) := \bigl[ \Psi(C\lambda) \bigr]^{\frac{1}{2p'}} \; \bigl[ \Xi(C \lambda) \bigr]^{\frac{1}{2} + \frac{1}{2p}} \quad \text{ for some large constant $C$.} \label{Lambda-in-Psi-Xi} \end{equation}    

\end{theorem} 
\noindent The conditional sharpness statement for Theorem \ref{mainthm-restriction-ball} leads naturally to a corresponding one for Theorem \ref{mainthm-restriction-arbitrary}. We record this below.  
\begin{corollary} \label{maincor-sharpness} 
Suppose that $M = \mathbb S^n$, and that $\Sigma$, $\Gamma$ and $p_0$ are as in Theorem \ref{mainthm-restriction-arbitrary}. For $\kappa > 0$, let $\nu^{(\kappa)}$ be a probability measure obeying \eqref{eigenfunction-restriction-estimate}. Suppose there exists $\kappa > 0$, a choice of $\Psi_{\kappa}, \Xi_{\kappa} \in \mathfrak F$ and $y_{\kappa} = y_0(\kappa) \in \Gamma$ for which $\nu^{(\kappa)}$
obeys both \eqref{mu-ball-condition} and \eqref{AD-local} with $\alpha = d(1 - \epsilon) - \kappa$. 
\vskip0.1in
\noindent Then  for every $p \geq \max(p_{\ast}, 2)$, there exists a constant $c_{\kappa} = c(\kappa, p) > 0$ and a sequence of $L^2$-normalized spherical harmonics $\{ \varphi_{\lambda_j} : j \geq 1 \}$ with $\lambda_j \nearrow \infty$ such that  
\begin{equation}  \| \varphi_{\lambda_j} \|_{L^p(\Gamma, \nu^{(\kappa)})} \geq  c_{\kappa}(1 + \lambda_j)^{\theta_p + \kappa/p} \bigl[\Lambda_p(\lambda_j) \bigr]^{-1}\text{ for all sufficiently large $j$.}\label{sphere-lower-bound-2}  \end{equation}  
Here $\theta_p$ is as in Theorem \ref{mainthm-restriction-arbitrary};  $\Lambda_p$ is defined as in \eqref{Lambda-in-Psi-Xi}, with $(\Psi, \Xi)$ replaced by $(\Psi_{\kappa}, \Xi_{\kappa})$.
\end{corollary} 

\noindent {\bf{Remarks: }} 
\begin{enumerate}[1.]
\item We observe that the growth exponent $\varkappa_p$ in Theorem \ref{mainthm-sharpness} agrees with the growth exponent $\vartheta_p$ in Theorem \ref{mainthm-restriction-ball} for $p$ in the range $p \geq \max(2, p_{\ast})$. Thus, while the lower bound on the eigenfunction growth provided by Theorem \ref{mainthm-sharpness} is valid for all $p \geq \max(1, p_{\ast}/2)$, it matches the upper bound in Theorem \ref{mainthm-restriction-ball} (upto factors of $\mathfrak F$) only when $p \geq \max(2, p_{\ast})$. This establishes the sharpness of the estimate \eqref{eigenfunction-growth-nu} for $p \in [\max(2, p_{\ast}), \infty]$.   
\vskip0.1in 
\item As in Remark \ref{interchange-ball} following Theorem \ref{mainthm-restriction-ball}, we observe that the condition \eqref{AD-local} is equivalent to its Euclidean counterpart:
\begin{equation} 
\limsup_{r \rightarrow 0} \bigl[ r^{-\alpha} \; \Xi(1/r) \mu(B(u_0;r)) \bigr] =: \kappa_0 > 0 \quad \text{ for some } u_0 \in E. 
\label{AD-local-3} \end{equation} 
We will use this version in many of our proofs.   
\vskip0.1in
\item Theorem \ref{mainthm-sharpness} and Corollary \ref{maincor-sharpness} show that our estimates \eqref{eigenfunction-restriction-estimate} and \eqref{eigenfunction-growth-nu} are sharp for $M = \mathbb S^n$ and any $\Sigma$ therein, for large $p$.  More precisely, if $\mu$ is a probability measure supported on $E \subseteq [0,1]^d$ obeying \eqref{mu-ball-condition-2} and \eqref{AD-local-3}, then there exists some manifold $M$ (namely $\mathbb S^n$), such that for any submanifold $\Sigma \subseteq M$ and any coordinate chart $(U, \mathfrak u)$, the estimate \eqref{eigenfunction-growth-nu} enjoyed by $\nu = \mu \circ \mathfrak{u}^{-1}$ cannot be improved for $p \geq \max(2, p_{\ast})$. This however does not preclude the possibility of an improved estimate (i.e., a smaller growth exponent than $\vartheta_p$) on other manifolds or on special submanifolds, even with the same range of $p$. Indeed, the growth estimate \eqref{eigenfunction-growth-nu} can be improved in certain other situations, for instance, when $M = \mathbb T^n$, the $n$-dimensional flat torus (which admits a stronger Weyl law), or if the submanifold $\Sigma$ in a general manifold $M$ has additional geometric properties, for example if $\Sigma$ is a curve of nonvanishing geodesic curvature. This is consistent with similar results of this type for smooth submanifolds, see for example \cite[Theorem 2]{BGT07}, \cite{{Bourgain-geodesic2009}, {BoRu12}, {Hassell-Tacy},  {Blair-Sogge}} and the bibliography therein. We expect random Cantor subsets of such $M$ and $\Sigma$ to present better eigenfunction restriction estimates than \eqref{eigenfunction-growth-nu}, and pursue this direction in greater detail in upcoming work. 
\end{enumerate} 
\subsubsection{Genericity of the sharp eigenfunction estimates} 
Measure-theoretic assumptions of the form \eqref{mu-ball-condition} and \eqref{AD-local} are instrumental in obtaining the sharp eigenfunction estimates in Theorems \ref{mainthm-restriction-ball} and \ref{mainthm-sharpness}. One is thus naturally led to the following questions: {\em{do there exist measures satisfying these assumptions, and if so, are they ubiquitous in any quantifiable sense?}} Our next theorem answers both these questions in the affirmative. 
\vskip0.1in
\noindent In preparation for the statement, let us define a family of continuous nondecreasing functions $\Phi_R: [1, \infty) \rightarrow [1, \infty)$ indexed by a large real parameter $R$, as follows, 
 \begin{equation} \Phi_R(t) := \exp(R \sqrt{\log t}) \quad \text{ for $t \geq 1$.}  \label{our-Phi} \end{equation} 
In view of the definition \eqref{functions-of-slow-growth}, we observe that $\Phi_R \in \mathfrak F$ for every $R \geq 0$, i.e., every $\Phi_R$ is a function of slow growth. 
\begin{theorem} \label{mainthm-example}
Fix $\epsilon \in [0,1)$, and positive integers $n \geq 2$ and $1 \leq d \leq n$. 
Then there is a probability space $(\Omega, \mathcal{B}, \mathbb{P}^*)$ depending on these parameters such that for every $\omega \in \Omega$, there exists a Cantor-like set $E = E(\omega) \subseteq [0,1]^d$ equipped with a probability measure $\mu = \mu (\omega)$ that obeys the following properties $\mathbb P^{\ast}$-almost surely.
\begin{enumerate}[(a)]
\vskip0.1in 
\item The random set $E$ has Hausdorff dimension $d(1-\epsilon)$. For $\epsilon = 0$, the random measure $\mu$ is singular with respect to the $d$-dimensional Lebesgue measure on $[0,1]^d$.   \label{parta} 
\vskip0.1in
\item There exist a positive deterministic constant $C_1 = C_1(n, d, \epsilon)$ and a positive $\omega$-dependent constant $C_2=C_2(\omega,n,d,\epsilon)$ such that 
\begin{equation} 
\sup \left\{ \mu  \bigl(B(x; r)\bigr) r^{-d(1 - \epsilon)} \bigl[ \Phi_{C_1}(1/r) \bigr]^{-1} : r > 0, x \in [0,1]^d  \right\} \leq C_2 < \infty.  \label{ball-condition-upper-bound} 
\end{equation} 
As a result, for any compact manifold $(M, g)$, any sub-manifold $\Sigma$ and any coordinate chart $(U, \mathfrak{u})$, the measure $\nu = \mu \circ \mathfrak{u}$ given by \eqref{mu-nu-relation} obeys the ball condition \eqref{mu-ball-condition} with $\alpha = d(1 - \epsilon)$, and $\Psi \equiv \Phi_{C_1}$.  
\label{mainthm-mainest}
\vskip0.1in
\item For every $v_0 \in E$, there exists a constant $C_0 = C(v_0 ; \omega) > 0$ such that  
\begin{equation} 
\inf \left\{ \mu \bigl(B(v_0; r)\bigr) r^{-d(1 - \epsilon)} \bigl[ \Phi_{C_1}(1/r) \bigr] : 0 < r < 1 \right\} \geq C_0^{-1}.
\label{ball-condition-lower-bound} \end{equation}   \label{sharpness-part}
\noindent Thus $\nu = \mu \circ \mathfrak{u} $ obeys \eqref{AD-local} at every point $y_0 = \mathfrak{u}(v_0) \in \Gamma$, with $\alpha = d(1 - \epsilon)$, $\kappa_0 \geq C_0^{-1}$ and $\Xi \equiv \Phi_{C_1}$.  \label{mainthm-example-partc} 
\end{enumerate}
\end{theorem}
\vskip0.1in
\noindent Our construction will show that the support of $\mathbb P^{\ast}$ contains uncountably many points $\omega \in \Omega$, each of which generates a distinct set $E = E(\omega) \subseteq [0,1]^d$; see the discussion following Lemma \ref{nonemptyE-lemma} on page \pageref{uncountable-sets}. In view of Theorems \ref{mainthm-restriction-ball} and \ref{mainthm-example}, almost every such set $E$ gives rise to an eigenfunction restriction estimate of the form \eqref{eigenfunction-growth-nu} on every $n$-dimensional manifold $M$. As indicated in remark \ref{sharpness-remark} following Theorem \ref{mainthm-restriction-arbitrary}, the estimate that we obtain for these random sets $E$ is stronger than the statement provided by Theorem \ref{mainthm-restriction-arbitrary} for a general set $E$, in the sense that one finds a single measure $\nu$ that obeys \eqref{eigenfunction-restriction-estimate} with $\kappa = 0$ (and hence for all $\kappa > 0$), at the expense of a slowly growing $\Phi$. Theorem \ref{mainthm-sharpness} shows that this generic estimate is sharp when $M = \mathbb S^n$.  We summarize these findings in the corollary below. 
\begin{corollary} \label{maincor-example}
Let $E = E(\omega)$, equipped with $\mu = \mu(\omega)$, denote respectively the random Cantor set and Cantor measure given by Theorem \ref{mainthm-example}. Then for $\mathbb P^{\ast}$-almost every $\omega \in \Omega$, the set $E = E(\omega)$ satisfies the conclusions below. 
\begin{enumerate}[(a)]
\item Choose any manifold $(M, g)$ and any submanifold $\Sigma$ therein, with $\Sigma$ parametrized by the local coordinate chart $(U, \mathfrak{u})$. Set $\Gamma = \mathfrak{u}(E) \subseteq \Sigma \subseteq M$.  
\vskip0.1in
\noindent Then $\Gamma$,  equipped with the probability measure $\nu = \mu \circ \mathfrak{u}^{-1}$, enjoys the following eigenfunction restriction property. For every $p \in [2, \infty]$, there exist positive finite constants $R = R(n, d, \epsilon)$ and $C = C(\omega, p, \Sigma, M)$ such that for all $\lambda \geq 1$, 
\begin{equation}  ||\varphi_{\lambda}||_{L^p(\Gamma, \nu)} \leq C (1 + \lambda)^{\theta_p} \; \Phi_{R/p}(\lambda) \; ||\varphi_{\lambda}||_{L^2(M, dV_g)}. \label{generic-sharp} \end{equation}  
Here $\theta_p$ is as in Theorem \ref{mainthm-restriction-arbitrary} but with $\kappa = 0$, and $\Phi_R$ is as in \eqref{our-Phi}. 
\vskip0.1in 
\item If $M = \mathbb S^n$, then for every $p \in [2, \infty]$ we can find a constant $C = C(\omega, p) > 0$ and a sequence of $L^2$-normalized eigenfunctions $\{\varphi_{\lambda_j} : j \geq 1\}$ with $\lambda_j \nearrow \infty$ such that
\[ \| \varphi_{\lambda_j} \|_{L^p(\Gamma, \nu)} \geq C^{-1} (1 + \lambda_j)^{\theta_p} \bigl[\Phi_{R}(\lambda_j) \bigr]^{-1}\text{ for all sufficiently large $j$.} \]
\end{enumerate}  
\end{corollary} 
\noindent {\bf{Remarks: }}  
\begin{enumerate}[1.]
\item The specific form \eqref{our-Phi} of the blow-up factor $\Phi_R(\lambda)$, which is super-logarithmic but sub-polynomial, is an artifact of the parameter choices needed for the random Cantor construction, see Section \ref{parameters-choice-section}. Many alternative choices are possible within the framework of this construction, some of which yield slower blow-up than $\Phi_R(\lambda)$, at the cost of additional technical challenges. We have opted not to pursue these improvements here. However, all the ``almost sure'' probabilistic statements of this type that employ our strategy will be accompanied by some blow-up. Using a different random construction mechanism, one can find a {\em{single}} Cantor-type set $E$ that avoids such losses; see remark \ref{remark-losses} in Section \ref{parameters-choice-section}. However, we do not know whether such sets are generic in a quantifiable sense.  
\vskip0.1in
\item The random measures $\mu = \mu(\omega)$ that we construct and their supporting sets $E = E(\omega)$ have many analytic and geometric properties that are not directly exploited in the proof. In particular, these measures have optimal Fourier decay subject to the Hausdorff dimension of their support. More precisely, for almost every $\omega$, our measures obey \[ \bigl| \widehat{\mu} (\xi) \bigr| \leq C \Psi(\xi) \bigl(1 + |\xi| \bigr)^{-d(1 - \epsilon)/2},  \quad |\xi| \geq 1, \; \xi \in \mathbb R^d, \]
for some $\Psi \in \mathfrak F$. Here $\widehat{\mu}$ denotes the Fourier transform of $\mu$, given by $\widehat{\mu}(\xi) = \int e^{- i x \cdot \xi} d\mu(x)$. In other words, the sets $E$ in Theorem \ref{mainthm-example} have the same Fourier dimension as their Hausdorff dimension, i.e. they are almost surely Salem. 
\vskip0.1in 
\noindent Measures that enjoy Fourier decay have long been known to play an important role in eigenfunction restriction problems. For instance, it appears in the work of Bourgain and Rudnick \cite{BoRu12}, where the authors obtain significant improvements on the general estimates of \cite{BGT07} in the special case of $M = \mathbb T^n$, $n = 2,3$.  More generally, the study of harmonic-analytic principles (such as Fourier decay, fractal analogues of the uncertainty principle, study of oscillatory integrals and operators) in settings where standard techniques (such as integration by parts or stationary phase) are not viable have led to major developments in spectral theory, for instance in the work surrounding resonance gaps in infinite-area hyperbolic surfaces \cite{{Na05}, {BoDy17}, {BoDy18}}. 
However, our methods are not Fourier-analytic in nature. This is reminiscent of \cite{LP11}, where a similar random Cantor set was constructed, but whose Fourier-analytic properties were not directly relevant to the proof.     
\vskip0.1in 
\item Restriction of eigenfunctions to fractals has appeared in a related but distinct line of inquiry that addresses spatial equidistribution of eigenfunctions restricted to subsets of manifolds, instead of norm growth. Most recently Hezari and Rivi\`ere  \cite{HR18} have established, in the specific setting of the flat torus $\mathbb T^n$, spatial equidistribution of a density one subsequence of eigenfunctions on sets of possibly fractional Hausdorff dimension. In fact, Corollary 2.7 of \cite{HR18} proves $L^2$-restriction estimates of this density one subsequence for a wider class of measures than the one used in this paper. This addresses a special case of a conjecture in \cite{BoRu12} on $L^2$-restriction to hypersurfaces. It would be of interest to explore the issue of equidistribution for the entire sequence of toral eigenfunctions for the class of random Cantor sets considered in this paper. 
\end{enumerate} 
\subsection{Overview of the proof}
The broad strokes of our approach follow that of \cite[Theorem 1]{BGT07}, so we briefly review the main ideas involved there. 
\begin{enumerate}[1.]
\item One starts with a microlocal approximation $\mathscr{T}_{\lambda}$ of the smoothed spectral projector $\chi(\sqrt{-\Delta_g} - \lambda)$. The approximation $\mathscr{T}_{\lambda}$ is an oscillatory integral operator, whose phase function is essentially the distance function in the ambient Riemannian metric.   
\vskip0.1in
\item The $TT^{\ast}$ method applied to $\mathscr{T}_{\lambda} \mathscr{T}_{\lambda}^{\ast}$ reduces the problem to estimating the $L^p$ norm of the latter operator on the restricted set $\gamma$, which for \cite[Theorem 1]{BGT07} was a smooth curve on $M$. 
\vskip0.1in 
\item The integration kernel of $\mathscr{T}_{\lambda} \mathscr{T}_{\lambda}^{\ast}$ is itself an oscillatory integral, with a nondegenerate phase function. The method of stationary phase, applied to this oscillatory integral, yields a pointwise upper bound on the kernel, leading to a pointwise bound on the operator $\mathscr{T}_{\lambda} \mathscr{T}_{\lambda}^{\ast}$. The dominating operator is a convolution, with an explicit convolving factor.
\vskip0.1in
\item The proof is then completed by invoking Young's convolution inequality for the Lebesgue measure on $\mathbb R$. The admissible exponents of the inequality are precisely those for which the convolving factor is integrable.  
\end{enumerate} 
\noindent A careful analysis of \cite[Theorem 1]{BGT07}, which we carry out in Section \ref{microlocal-preliminaries-section}, shows that steps 1, 2 and 3 above extend with minor revisions to the setting of an arbitrary measure space, with the smooth curve $\gamma$ replaced by $\Gamma = \mathfrak{u}(E)$. A noteworthy point of departure is the following. Whereas the natural measure on the curve $\gamma$ used in \cite{BGT07} is absolutely continuous with respect to the translation-invariant Lebesgue measure on $\mathbb R$, a general measure $\mu$ accompanying an arbitrary Borel set $E$ need not be translation invariant any more. 
The proof thus fails critically at the last step, since Young's convolution inequality is unavailable, indeed known to be false, in general measure spaces. However, two key facts remain in place: 
\vskip0.1in
\begin{enumerate}[1.]
\item {\em{A generalization of Young's inequality}}, proved in Section \ref{young-section} (a version of which appears in the literature as Schur's lemma), {\em{continues to hold for measures that obey a ball growth condition}} of the type \eqref{mu-ball-condition}. Thus the fourth step in the strategy of \cite{BGT07} can be completed in the presence of this measure-theoretic assumption. This is the main ingredient in the proof of Theorem \ref{mainthm-restriction-ball}.  
\vskip0.1in
\item {\em{Ball growth conditions like \eqref{mu-ball-condition} are guaranteed for any Borel set of a given Hausdorff dimension}}. In fact, a standard result in geometric measure theory known as Frostman's lemma characterizes the Hausdorff dimension of a Borel set in terms of such measures. This, combined with the preceding observation, leads to the eigenfunction restriction estimates for arbitrary Borel sets $E$ of prescribed dimension, as stated in Theorem \ref{mainthm-restriction-arbitrary}. 
\end{enumerate} 
\vskip0.1in 
\subsection{Layout of the paper} The organization of our article and the logical dependencies among various sections are explained in Figure 1. Section \ref{microlocal-preliminaries-section} sets up the microlocal analysis background regarding the smooth spectral projector $\mathscr{T}_{\lambda}$. In particular, we recall an explicit asymptotic expansion for it, whose leading term becomes the main object of interest. Our adaptation of Young's inequality for general measures appears in Proposition \ref{gen-Young-prop} of this section. Its proof, based on standard techniques, has been relegated to Section \ref{young-section}. Application of this generalized Young's inequality in our context requires an estimate on the integration kernel of $\mathscr{T}_{\lambda} \mathscr{T}_{\lambda}^{\ast}$. This estimate is obtained in Section \ref{integration-kernel-section}, using the ball growth condition \eqref{mu-ball-condition}. This is a key step in the proofs of Theorems \ref{mainthm-restriction-arbitrary} and \ref{mainthm-restriction-ball}, which appear in Section \ref{proofs-mainthms-section}. In Section \ref{lower-bound-section}, we obtain a conditional lower bound on the spectral projectors, which pave the way for the sharpness results Theorems \ref{mainthm-sharpness} and \ref{maincor-sharpness}. The theorems themselves have been proved in  Section \ref{sharpness-section}.  The remainder of the paper is devoted to the proofs of Theorem \ref{mainthm-example} and Corollary \ref{maincor-example}. Section \ref{s:randomcantor} is dedicated to a discussion of a general Cantor-type construction. Section \ref{random-cantor-section} sets up the framework for Theorem \ref{mainthm-example}, namely the measure space $(\Omega, \mathcal B, \mathbb P^{\ast})$ and the associated random Cantor sets $E$ and measures $\mu$. The properties derived in this section are then used to prove Theorem \ref{mainthm-example} and Corollary \ref{maincor-example} in Section \ref{ball-measures}. 
The key probabilistic tools and associated results that have been used repeatedly throughout Section \ref{random-cantor-section} have been collected in  Section \ref{Appendix-section}.  
\begin{figure} 
\begin{tikzpicture}[node distance=1.5cm]
\node (thm) [box] {Section 2};
\node (Yineq) [box, xshift=4cm] {Section 10}; 
\node (PartA) [box, below of=thm, xshift=-4cm] {Section 5};
\node (PartB) [box, below of=thm] {Section 3};
\node (PartC) [box, below of=thm, xshift = 4cm] {Section 7};
\node (29) [box, below of=PartA, yshift=-1.5cm] {\S 6: Theorem \ref{mainthm-sharpness} and Corollary \ref{maincor-sharpness}};
\node (35) [box, below of=PartB] {\S 4: Theorems \ref{mainthm-restriction-arbitrary} and \ref{mainthm-restriction-ball}};
\node (10) [box, below of=PartC] {Section 8};
\node (11) [box, below of= PartC, xshift=3cm] {Section 11}; 
\node (29b) [box, below of=10] {\S 9: Theorem \ref{mainthm-example} and Corollary \ref{maincor-example}};
\draw [arrow] (PartA)--(29);
\draw [arrow] (PartB)--(35);
\draw [arrow] (PartC)--(10);
\draw [arrow] (35)--(29b);
\draw [arrow] (29) -- (29b);
\draw [arrow]  (10)--(29b);
\draw [arrow] (11) -- (10);
%
\draw [arrow] (Yineq) -- (thm);
\draw [arrow] (thm) -- (PartB);
\draw [arrow] (thm) -- (PartA);
\end{tikzpicture}
\caption{Layout}
\end{figure} 

\section{Microlocal tools} \label{microlocal-preliminaries-section}
\subsection{The smoothed spectral projector} \label{ss:specproj}


Our study of eigenfunction restriction estimates is based on an explicit integral representation of an underlying operator $\mathscr T_{\lambda}$, called the smoothed spectral projector, in local coordinates.  The formulation of the spectral projector $\mathscr{T}_{\lambda}$ as a Fourier integral operator is well-known and ubiquitous in the literature (see for example, \cite[Theorem 4]{BGT07}, \cite{So} and the references therein). We briefly recall the set-up here for completeness, and state the explicit representation of $\mathscr{T}_{\lambda}$ in Theorem \ref{t:smoothproj} below.  
\vskip0.1in
\noindent Let us consider a Riemannian uniformly normal neighbourhood of a base point $x_0 \in M$. Let $\rho_0 > 0$ denote the radius of a small geodesic ball centred at $x_0$ contained in this neighbourhood, whose exact value will be specified later.  Without loss of generality and due to our normal coordinate system, we identify $x_0$ with the origin in $\R^n$.  We set $V = \{ y \in \mathbb{R}^n : |y| \leq \rho_0 \}$ and fix a neighbourhood $W$ of $x_0$ contained in $V$. Without loss of generality, we will assume that $\mathfrak{u}(U) \subseteq W$, where $\mathfrak u: U \rightarrow \Sigma$ is a local coordinate chart on $\Sigma$. 
\vskip0.1in
\noindent Let us consider the first-order pseudo-differential operator $\sqrt{-\Delta}$ given by the spectral theorem. Then
\[ \sqrt{-\Delta} = \sum_{j=0}^{\infty} \lambda_j \mathscr{P}_j, \qquad I = \sum_{j=0}^{\infty} \mathscr{P}_j, \]
where $\mathscr{P}_j$ is the projection operator onto the finite-dimensional eigenspace corresponding to the eigenvalue $\lambda_j$. Furthermore \[ e^{it \sqrt{-\Delta}} = \sum_{j=0}^{\infty} e^{it \lambda_j} \mathscr{P}_j. \]
Let us now fix a function $\chi \in \mathcal S(\mathbb R^n)$ with $\chi(0) = 1$ and  supp$(\widehat{\chi}) \subseteq (\rho/2, \rho)$, with $\rho < \rho_0$. This leads to the smooth projection operators
\begin{equation}  
\mathscr{T}_{\lambda} := \chi(\sqrt{-\Delta} - \lambda) = \sum_{j} \chi(\lambda_j - \lambda) \mathscr{P}_j. \label{def-chi-lambda} \end{equation} 
We observe that \begin{equation} \chi(\sqrt{-\Delta} -\lambda) \varphi_{\lambda} = \varphi_{\lambda} \text{ for all } \lambda \in \{ \lambda_j : j \geq 1\}. \label{projector-and-eigenfunction}  \end{equation}  
Furthermore, a formal operator calculus made rigorous by the spectral theorem shows that 
\begin{align}
\mathscr{T}_{\lambda} := \chi(\sqrt{-\Delta} - \lambda) &= \sum_{j} \chi(\lambda_j - \lambda) \mathscr{P}_j  \label{Tlambda-def}\\ 
&= \sum_j \left[ \frac{1}{2 \pi} \int e^{it(\lambda_j - \lambda)} \hat{\chi}(t) \, dt \right] \circ \mathscr{P}_j \nonumber \\ 
&= \frac{1}{2 \pi} \int e^{-it \lambda} e^{it \sqrt{-\Delta}} \hat{\chi}(t) \, dt. \nonumber
\end{align}   
For $\rho$ small enough, it is a classical result that $e^{it \sqrt{-\Delta}}$ can be represented as a Fourier integral operator in local coordinates, see for instance \cite{{Ho68}}, \cite[Theorem 4.1.2, Lemma 5.1.3]{So}. A stationary phase argument then leads to the following well-known theorem:   
\begin{theorem}{\cite[Lemma 2.3]{BGT05}, \cite[Theorem 4]{BGT07}} \label{t:smoothproj}
Given the setup described above, one can find a small radius $\rho_0 > 0$ such that for every $\rho \in (0, \rho_0)$ and every positive integer $N \geq 1$, there exist distinct positive constants $0 < c_0' < c_0 < c_1 < c_3 < c_4 < c_2 < 1$ depending on these parameters with the properties listed below.  
\vskip0.1in
\noindent For every $\lambda \geq 1$, there exists a smooth function $a_{\lambda}: W \times V  \rightarrow \mathbb{C}$ such that:
\begin{equation} \label{def-S} 
\text{supp}(a_{\lambda}) = S := \bigl\{ (x,y) \in W \times V : |x| \leq c_0 \rho < c_1 \rho \leq |y| \leq c_2 \rho < \rho \bigr\},
\end{equation}
and the following conclusions hold.  
\begin{enumerate}[(a)]
\item The function $a_{\lambda}(x, y)$  is a polynomial in $\lambda^{-1}$, and does not vanish for $(x,y) \in S$ with  $|x| \leq c_0' \rho$ and $d_g(x,y) \in [c_3 \rho, c_4 \rho]$. \label{t:smoothproj-parta}
\vskip0.1in
\item The function $a_{\lambda}$ and its spatial derivatives are uniformly bounded, i.e., for every multi-index $\alpha \in \mathbb Z_{\geq 0}^{2n}$, there exists a constant $C_{\alpha} > 0$ such that $|\partial_{x,y}^{\alpha} a_{\lambda}(x,y)| \leq C_{\alpha}$.  
\vskip0.1in 
\item The function $a_{\lambda}$ appears in the representation of the integral kernel for the smoothed spectral projector $\mathscr{T}_{\lambda}$ defined as in \eqref{Tlambda-def}. Specifically, 
for all $x \in W$ and all $f \in L^2(V)$ 
\begin{equation} \label{e:smoothproj}
\mathscr{T}_{\lambda}(f)(x) = \lambda^{\frac{n-1}{2}} \int_{y \in V} e^{- i \la d_g(x,y)} a_{\lambda}(x,y) f(y) \, dy + \mathscr{R}_{\lambda}(f), \qquad |x| < c_0' \rho. 
\end{equation}
The remainder term $\mathscr{R}_{\lambda}$ is a uniformly bounded smoothing operator in the sense that for every $0 \leq r \leq N$, there exists a constant $C_{r,N}$ independent of $\lambda$ such that 
\begin{equation} \label{R-smoothing} 
||\mathscr{R}_{\lambda}||_{\mathtt H^r(W) \rightarrow L^2(M)} \leq C_{r, N} \lambda^{r-N}  \text{ for all $\lambda \geq 1$.} 
\end{equation} 
 Here $\mathtt H^r(W)$ denotes the space of square-integrable functions on $W$, whose weak derivatives of order $\leq r$ are also square-integrable.
\end{enumerate} 
\end{theorem}
\noindent If we choose $r=N$ with sufficiently large $N$ in \eqref{R-smoothing}, the Sobolev embedding theorem ensures that $\mathscr{R}_{\lambda}f$ is continuous on $W$ for every $f \in L^2(M)$. In particular, this means that \[ ||\mathscr{R}_{\lambda} f||_{L^{\infty}(W)} \leq C ||\mathscr{R}_{\lambda}||_{\mathtt H^r(W)} \leq C ||f||_{L^2(M)}.  \] Combining this with the continuity of $\mathscr{R}_{\lambda}f$, we have the following estimate: for any probability measure $\nu$ supported on a set $\Gamma \subseteq M$, 
\[ ||\mathscr{R}_{\lambda} f||_{L^2(\nu)} \leq ||\mathscr{R}_{\lambda} f||_{L^{\infty}(\nu)} \leq  ||\mathscr{R}_{\lambda} f||_{L^{\infty}(W)} \leq C ||f||_{L^2(M)} \]   for all functions $f \in L^{2}(V)$. Thus $\mathscr{R}_{\lambda}$ does not contribute any significant power of $\lambda$ in eigenfunction restriction estimates, and we ignore it in the sequel. By a slight abuse of notation, we will rename as $\mathscr{T}_{\lambda}$ the leading term in \eqref{e:smoothproj}.

\subsection{Restriction of $\mathscr{T}_{\lambda}$ to $\Gamma$ and reduction to a normal operator}
 Equipped with the representation \eqref{e:smoothproj}, let us now consider the smooth spectral projection operator $\mathscr T_{\lambda}$ {\em{restricted}} to $\Gamma$:
\begin{equation} \label{restriction-operator-def}
\mathcal T_{\lambda} := \mathscr{T}_{\lambda} \Bigr|_{\Gamma}, 
\end{equation} 
which is the main operator of interest. The reason for this is that $\varphi_{\lambda} = \mathscr{T}_{\lambda} \varphi_{\lambda}$ in view of \eqref{def-chi-lambda} and \eqref{projector-and-eigenfunction}, which in turn equals $\mathcal T_{\lambda} \varphi_{\lambda}$ on $\Gamma$. Thus eigenfunction estimates of the desired form \eqref{eigenfunction-growth-nu} follows if we show that $\mathcal T_{\lambda}$ maps $L^2(M, dV)$ boundedly to $L^p(\Gamma, \nu)$, with the operator norm bounded by a constant multiple of $(1 + \lambda)^{\vartheta_p} \overline{\Psi}_p(\lambda)$. By a standard duality principle known as the $TT^{\ast}$-method (used in a similar context in \cite[Theorem 1]{BGT07}), we have for any linear operator $\mathcal T_{\lambda}$: 
\begin{equation} \label{T-TTstar}  \bigl| \bigl|\mathcal T_{\lambda} \bigr| \bigr|_{L^2(M) \rightarrow L^p(\nu)}^2 = \bigl| \bigl| \mathcal  T_{\lambda} \mathcal T_{\lambda}^{\ast} \bigr|\bigr|_{L^{p'}(\nu) \rightarrow L^p(\nu)} \quad \text{ where } \frac{1}{p} + \frac{1}{p'} = 1. \end{equation}  Therefore, the desired mapping property of $\mathcal T_{\lambda}$  is true if and only if $\mathcal  T_{\lambda} \mathcal T_{\lambda}^{\ast}$ maps $L^{p'}(\Gamma, \nu)$ to $L^p(\Gamma, \nu)$, with the appropriate norm bound. Let us summarize this discussion by writing: 
\begin{equation} \label{hope-to-show}  
\text{ if } \bigl| \big| \mathcal  T_{\lambda} \bigr| \bigr|_{L^{2}(M) \rightarrow L^p(\nu)} = \bigl| \big| \mathcal  T_{\lambda} \mathcal T_{\lambda}^{\ast}\bigr| \bigr|_{L^{p'}(\nu) \rightarrow L^p(\nu)} ^{\frac{1}{2}}\leq C (1 + \lambda)^{\vartheta_p} \overline{\Psi}_p(\lambda), \; \text{ then } \;  \eqref{eigenfunction-growth-nu} \text{ holds}. 
\end{equation} 
We are thus led to examine the normal operator $\mathcal  T_{\lambda} \mathcal T_{\lambda}^{\ast}$, with the aim of proving the operator norm estimate in \eqref{hope-to-show}. 
 
\vskip0.1in
\noindent Let  $\mathscr{K}_{\lambda}(x, x')$ denote the Schwarz kernel of the operator $\mathscr{T}_{\lambda} \mathscr{T}_{\lambda}^*$, where $\mathscr{T}_{\lambda}$ is the operator given by the leading term in \eqref{e:smoothproj}. If $(U, \mathfrak{u})$ is a local coordinate chart on $\Sigma$ such that $\mathfrak{u}(U) \subseteq W$, it follows that 
\begin{equation} \mathcal T_{\lambda} \mathcal T_{\lambda}^{\ast} f \bigl(\mathfrak{u}(u) \bigr)= \int \mathscr{K}_{\lambda}(\mathfrak{u}(u), \mathfrak{u}(v)) f \circ \mathfrak{u}(v)  \; d\mu(v), \qquad u \in U, \; \mathfrak{u}(u) \in \Sigma. \label{normal-op} \end{equation}    
In other words, the Schwarz kernel of $\mathcal T_{\lambda} \mathcal T_{\lambda}^{\ast}$ is the restriction of $\mathscr{K}_{\lambda}$ to $\Gamma \times \Gamma$. The first step in establishing Lebesgue boundedness of the normal operator $\mathcal T_{\lambda} \mathcal T_{\lambda}^{\ast}$ is therefore studying the Schwarz kernel $\mathscr{K}_{\lambda}$. Our next result gives an asymptotic expansion of this kernel.
\begin{theorem}{\cite[Lemma 6.1]{BGT07}}\label{thm:kernel_expansion}
Let $W \subseteq \mathbb R^n$ be as in Theorem \ref{t:smoothproj}. Then for $(x, x') \in W \times W$ and  $\mathscr{K}_{\lambda}(x, x')$ as in the preceding paragraph, the following conclusions hold:
\vskip0.1in 
\begin{enumerate}[(a)]
\item The function $(x, x') \mapsto \lambda^{1-n} \mathscr{K}_{\lambda}(x, x')$ is uniformly bounded on $W \times W$ by a constant independent of $\lambda$. \label{K-lambda-uniform-bound} 
\vskip0.1in 
\item There exist constants $\rho < 1 < C$ and a sequence of real-valued symbols $(\mathfrak a_m^{\pm}, \mathfrak b_m) \in C^{\infty}(W \times W \times \R)$ such that for $|x - x'| \geq C \lambda^{-1}$ and any $N \in \N$, the following expansion holds: 
\begin{equation} \label{asymptotic-expansion-kernel} 
\lambda^{1-n} \mathscr{K}_{\lambda}(x, x') = \sum_{\pm} \sum_{m=0}^{N-1} \frac{e^{\pm i \lambda d_g(x,x')}}{(\lambda | x - x' |)^{\frac{n-1}{2} + m}} \mathfrak a_m^{\pm}(x, x', \lambda) + \mathfrak b_N(x,x', \lambda).
\end{equation}
Each of the symbols $a_{m}^{\pm}$ has support in $W \times W$ of size $\mathcal{O}(\rho)$ (independent of $\lambda$) and is uniformly bounded in $\lambda$.  The remainder $\mathfrak b_N$ obeys the estimate \[| \mathfrak b_N(x,x', \lambda)| \leq C_N \left(\lambda | x - x' | \right)^{-\frac{n-1}{2} - N}.\] \label{K-lambda-asymptotic} 
\item Setting $N=1$ in \eqref{asymptotic-expansion-kernel} and combining it with part \eqref{K-lambda-uniform-bound} leads to the following statement: if $\{x = \mathfrak{u}(u): u \in U\}$ is a local parameterization of a $d$-dimensional smooth embedded (sub)manifold $\Sigma \subseteq M$, then there exists a constant $C > 1$ such that 
\begin{equation} \label{kernel-domination} \bigl| \mathscr{K}_{\lambda}(x,x') \bigr| \leq C \lambda^{n-1} \langle \lambda (u-v) \rangle^{-\frac{n-1}{2}}, 
\end{equation} for all $x = \mathfrak{u}(u)$, $x' = \mathfrak{u}(v) \in \Sigma$, $u,v \in U$.  Here $\langle \cdot \rangle$ denotes the Japanese bracket given by $\langle u \rangle := (1 + |u|^2)^{1/2}$.    
\end{enumerate} 
\end{theorem}

\subsection{Reduction to a generalized Young-type inequality} The kernel estimate \eqref{kernel-domination} allows us to bound $\mathcal T_{\lambda} \mathcal T_{\lambda}^{\ast}$ pointwise by a convolution operator. 
\begin{proposition} \label{gen-Young-prop}
Given a set $\Gamma \subseteq \Sigma$, let $\mathcal T_{\lambda}$ denote the restriction to $\Gamma$ of the smooth spectral projection operator, as defined by \eqref{restriction-operator-def}. Suppose that $\Gamma$ is parameterized by $E \subseteq [0,1]^d$ via a coordinate chart $\mathfrak{u}$, i.e., $\Gamma = \mathfrak{u}(E)$.  Let $\mu$ be a non-negative measure supported on $E$. 
\begin{enumerate}[(a)] 
\item Then for all functions $f$ and all $u \in E$,  we have the pointwise inequality 
\begin{align} 
\bigl|\mathcal T_{\lambda} \mathcal T_{\lambda}^{\ast} f \bigl(\mathfrak{u}(u) \bigr)| &\leq C \lambda^{n-1} \bigl[\mathcal L_{\lambda} |f| \bigr](\mathfrak{u}(u)), \quad \text{ where }  \label{pointwise-domination} \\  \mathcal L_{\lambda}f \bigl(\mathfrak{u}(u) \bigr) &:= \int_E \mathcal K_{\lambda}(u-v) (f \circ \mathfrak{u})(v) \, d\mu(v), \text{ and } \nonumber \\  \mathcal K_{\lambda}(u) &:= \langle \lambda u \rangle^{-\frac{n-1}{2}}. 
\nonumber
\end{align}  
\item \label{gen-Young-b} For $p \geq 2$, the operator $\mathcal T_{\lambda}$ is bounded as a linear operator from $L^2(M, dV)$ to $L^p(\Gamma, \nu)$ provided 
\begin{equation}  \mathtt a_p^{\frac{p}{2}} := \sup_{u \in E} \int_E \big|\mathcal K_{\lambda}(u - v) \bigr|^{\frac{p}{2}} \, d\mu(v) < \infty, \label{def-a}  \end{equation} 
with the operator norm of $\mathcal T_{\lambda}$ is bounded above by a constant multiple of $\sqrt{ \lambda^{n-1} \mathtt a_p}$.  \label{T-op-norm} 
\end{enumerate} 
\end{proposition} 
\begin{proof} 
The inequality \eqref{pointwise-domination} follows from \eqref{normal-op} combined with \eqref{kernel-domination}.
Part (\ref{T-op-norm}) is a consequence of a generalized Young's inequality, stated and proved in Proposition \ref{l:youngsineq} in an appendix below (Section \ref{young-section}). We have used this proposition with $T$ replaced by $\mathcal L_{\lambda}$, $\tau$ replaced by $\mu$, $r = p$, $q=p'$ and $s = p/2$. In view of the symmetry and the translation-invariance of the kernel $\mathcal K_{\lambda}$, both the quantities $A_s$ and $B_s$ in \eqref{AsBs} equal $\mathtt a_p$ in this context. If $\mathtt a_p$ is finite, the conclusion of Proposition \ref{l:youngsineq} asserts that $\lambda^{n-1} \mathcal L_{\lambda}$ is bounded as a linear operator from $L^{p'} \rightarrow L^p$ with norm at most $\lambda^{n-1} \mathtt a_p$. In view of \eqref{pointwise-domination}, $\mathcal T_{\lambda} \mathcal T_{\lambda}^{\ast}$ has the same property. By the duality principle \eqref{T-TTstar}, $\mathcal T_{\lambda}$ maps $L^2(M, dV_g)$ to $L^p(\Gamma, \nu)$ with norm bounded by the square root of $\lambda^{n-1} \mathtt a_p$.   
\end{proof} 
\noindent \label{Young-prop-remark} {\bf{Remark: }} Following up on remark \ref{remark-nonoptimal} on page \pageref{remark-nonoptimal}, we pause for a moment to point out to the reader that the inequality \eqref{pointwise-domination}, which replaces $\mathcal T_{\lambda} \mathcal T_{\lambda}^{\ast}$ by a multiple of a convolution operator $\mathcal L_{\lambda}$, is the step where oscillations are ignored. The phase function $\lambda d_g(x,x')$ present in $\mathscr{K}_{\lambda}(x,x')$ disappears upon the application of triangle inequality, leaving an integral operator with Schwartz kernel $|\mathscr{K}_{\lambda}|$, which is pointwise bounded by $\mathcal K_{\lambda}$. This loss of information results in an eventual loss in the range of $p$ where our estimates are sharp.   

\section{An integration kernel estimate} \label{integration-kernel-section} 
\noindent In view of Proposition \ref{gen-Young-prop}, the problem of Lebesgue boundedness of $\mathcal T_{\lambda}$ reduces to an estimation of the quantity $\mathtt a_p$ in \eqref{def-a}. Given a probability measure $\mu$ supported on a Borel set $E \subseteq [0,1]^d$ and any $p>0$, let us define the function 
\begin{equation}  \mathfrak A(u, \lambda; p)  := \int \langle \lambda (u-v) \rangle^{-\frac{p(n-1)}{4}} \, d\mu(v).  \label{random-function} 
\end{equation} 
The relevance of $\mathfrak A$ is that its supremum over $u \in E$ yields $\mathtt a_p^{p/2}$. It is our claim that a ball growth condition on $\mu$ of the form \eqref{mu-ball-condition} is sufficient to generate an quantitative estimate for $\mathfrak A$, and hence $\mathtt a_p$. The following proposition, which makes this precise, is the main step towards Theorems \ref{mainthm-restriction-ball} and \ref{mainthm-restriction-arbitrary}. 

\begin{proposition} \label{reduction-mainprop}
Fix $\alpha \in [0, d)$ and $\Psi \in \mathfrak F$. Suppose that $E \subseteq [0,1]^d$ is a Borel set supporting a probability measure $\mu$ that obeys \eqref{mu-ball-condition}. Let $p \in (0, \infty)$. 
Then there exists a constant $C > 0$ depending only on $p, n, \alpha$ such that for all $\lambda \geq 1$, the quantity $\mathfrak A$ in \eqref{random-function} admits the following estimate:   
\begin{equation} \sup_{u \in [0,1]^d} |\mathfrak A(u, \lambda;p)| \leq  C \Psi(\lambda) \lambda^{-\beta_p} \times \begin{cases} 1 &\text{ if } p \ne p_{\ast}, \\ \log \lambda &\text{ if } p = p_{\ast}. \end{cases}  
\label{main-estimate} \end{equation} 
Here $p_{\ast}$ is the exponent defined in \eqref{our-exponent-2}. The exponent $\beta_p$ is given by 
\begin{equation} 
\beta_p = \beta(p, n, \alpha) := \left\{ \begin{aligned} &\frac{p(n-1)}{4} &\text{ if } 0 < p \leq p_{\ast}, \\ &\alpha &\text{ if } p \geq p_{\ast}. \end{aligned} \right\} \label{def-alphap} 
\end{equation}
\end{proposition} 
\begin{proof} 
Let us fix $u \in [0,1]^d$, and define the sets
\begin{align*} 
U_j &= \{v \in [0,1]^d : 2^{-j-1} \leq |u-v| \leq 2^{-j} \} \text{ for } - C_d = \log_2(1/\sqrt{d}) \leq j < j_0 = \log_2 \lambda, \\ 
U_{j_0} &= \{v \in [0,1]^d : |u-v| \leq \lambda^{-1} \}.   
\end{align*} 
The relevance of these sets is that they cover $[0,1]^d$ and are essentially disjoint, allowing the following decomposition of the integral representing  $\mathfrak A(u, \lambda; p)$: 
\begin{align*}
\mathfrak A(u, \lambda; p) &=  \int \frac{d \mu(v)}{(1 + \lambda| u-v|)^{p(n-1)/4}} \\
&  \leq \sum_{j=-C_d}^{j_0} 
\int_{U_j} \frac{d \mu(v)}{(1 + \lambda| u-v|)^{p(n-1)/4}}  \\
& \leq \sum_{j=-C_d}^{j_0-1} (\lambda 2^{-j-1})^{-p(n-1)/4} \mu(B(u; 2^{-j})) + \mu(B(u; \lambda^{-1})).
\end{align*} 
Invoking the crucial ball growth condition \eqref{mu-ball-condition}, we can estimate $\mu(B(u; r))$, obtaining  
\begin{align*} 
\mathfrak A(u, \lambda; p) & \leq C_{0} 2^{p(n-1)/4}  \lambda^{-p(n-1)/4} \sum_{j=-C_d}^{j_0-1} 2^{j(\frac{p(n-1)}{4} - \alpha)} \Psi(2^{j})  + C_{0} \lambda^{- \alpha} \Psi(\lambda)  \\
& \leq C  \lambda^{-p(n-1)/4} \, \Psi(\lambda) \times \left\{ \begin{aligned} 
&\lambda^{\frac{p(n-1)}{4} - \alpha}, & \mbox{ if } p > p_\ast, \\
&\log \lambda, & \mbox{ if } p = p_\ast, \\
&1, & \mbox{ if } p < p_\ast
\end{aligned} \right\} + C \lambda^{-\alpha} \Psi(\lambda)   \\
& \leq C \lambda^{-\beta_p} \Psi(\lambda) \times \begin{cases} 1 &\text{ if } p \ne p_{\ast}, \\ \log \lambda &\text{ if } p = p_{\ast}. \end{cases}
\end{align*}
This concludes the proof.
\end{proof} 
\section{Proof of Theorems \ref{mainthm-restriction-ball} and \ref{mainthm-restriction-arbitrary}} \label{proofs-mainthms-section} 
\begin{proof}[{\bf{Proof of Theorem \ref{mainthm-restriction-ball}}}]
In view of \eqref{T-TTstar}, it suffices to the prove the $L^2(M) \rightarrow L^p(\nu)$ estimate on $\mathcal T_{\lambda}$ mentioned in \eqref{hope-to-show}. Combining Propositions \ref{gen-Young-prop}\eqref{gen-Young-b}  and \ref{reduction-mainprop}, we obtain for $p \geq 2$, 
\begin{align*} 
\bigl| \bigl| \mathcal T_{\lambda} \bigr|\bigr|_{L^2(M) \rightarrow L^p(\nu)} &\leq \lambda^{\frac{n-1}{2}} \mathtt a_p^{\frac{1}{2}} \text{ (with $\mathtt a_p$ as in \eqref{def-a})} \\   
&\leq \lambda^{\frac{n-1}{2}} \bigl[ \sup_{u \in [0,1]^d} |\mathfrak A(u, \lambda;p)| \bigr]^{\frac{1}{p}} \\ 
&\leq C (1 + \lambda)^{\frac{n-1}{2}} (1 + \lambda)^{-\beta_p/p} \overline{\Psi}_p(\lambda) \\ 
&\leq C (1 + \lambda)^{\vartheta_p} \overline{\Psi}_p(\lambda).
\end{align*} 
It is easy to verify that $\vartheta_p = (n-1)/2 - \beta_p/p$ and $\overline{\Psi}_p$ are defined as in the statement of Theorem \ref{mainthm-restriction-ball}. 
\end{proof}   
\begin{proof}[{\bf{Proof of Theorem \ref{mainthm-restriction-arbitrary}}}]
In addition to Theorem \ref{mainthm-restriction-ball}, our main tool here is Frostman's lemma (see for example Theorem 8.8 in \cite{M95} or Section 4.1 of \cite{F90}), which characterizes the Hausdorff dimension of a Borel set in Euclidean space in terms of ball growth conditions of the form \eqref{mu-ball-condition}. Specifically, it says that for every Borel set $E \subseteq \mathbb R^d$, 
\begin{equation}  \text{dim}_{\mathbb H}(E) = \sup \left\{s : \begin{aligned} &\exists \text{ a probability measure } \tau \text{ supported on $E$}  \\ & \text{ and a finite constant $C$ such that } \\ &\tau(B(x;r)) \leq C r^s \text{ for all } x \in \mathbb R^d \text{ and }  r > 0. \end{aligned} \right\}. \label{Frostman} \end{equation}  
This means that given any Borel set $E \subseteq [0,1]^d$ of Hausdorff dimension $d(1- \epsilon)$ and any small $\kappa > 0$, there exists a probability measure $\mu^{(\kappa)}$ supported on $E$ that obeys \eqref{mu-ball-condition} with $\alpha = d(1 - \epsilon) - \kappa$, $\Psi \equiv 1$ and a positive finite constant $C_0 = C_0(\kappa)$. Theorem \ref{mainthm-restriction-ball} then ensures that the estimate \eqref{eigenfunction-growth-nu} holds with $\nu = \nu^{(\kappa)} = \mu^{(\kappa)} \circ \mathfrak{u}^{-1}$. We observe that for our choice of $\alpha$, the critical threshold $p_{\ast}$ in Theorem \ref{mainthm-restriction-ball} is given by $p_{\ast}^{(\kappa)} = p_0 - 4\kappa/(n-1)$, and $\vartheta_p$ of Theorem  \ref{mainthm-restriction-ball} takes on the value
\[ \vartheta_p = \vartheta_p^{(\kappa)} = \left\{ \begin{aligned} &\frac{n-1}{4} &\text{ for } 2 \leq p \leq p_0 - \frac{4 \kappa}{n-1}, \\ &\frac{n-1}{2} - \frac{d(1 - \epsilon)}{p} + \frac{\kappa}{p} &\text{ for } p \geq p_0 - \frac{4 \kappa}{n-1}. \end{aligned} \right\}  \]
Recalling the definitions of $\theta_p$ and $\Theta$ from Theorem \ref{mainthm-restriction-arbitrary}, we  observe that \[ (1 + \lambda)^{\vartheta_p} \overline{\Psi}_p(\lambda) \leq C_{\kappa, p} (1 + \lambda)^{\theta_p} \Theta(\lambda; \kappa, p)  \]
for a large enough choice of constant $C_{\kappa, p}$. Thus restating \eqref{eigenfunction-growth-nu} in terms of $\theta_p$ and $\Theta$ leads us to the claimed estimate \eqref{eigenfunction-restriction-estimate}, completing the proof.  
\end{proof} 

\section{Lower bounds for spectral projectors} \label{lower-bound-section}
\noindent We now turn our attention to estimating the norms of spectral projectors from below. The main result of this section, Proposition \ref{smooth-spectral-proj-prop}, achieves this under certain measure-theoretic conditions. This lower bound will play an important role in the proof of Theorem \ref{mainthm-sharpness}, which appears in the next section. 
\subsection{Conditional lower bound for smooth spectral projectors} 
\begin{proposition}\label{smooth-spectral-proj-prop} 
Let us fix $\alpha \in (0, d]$, and a probability measure $\mu$ that obeys the hypotheses of Theorem \ref{mainthm-sharpness}. Then for any Riemannian manifold $(M, g)$, any submanifold $\Sigma \subseteq M$ and any $p \geq \max(1, p_{\ast}/2)$, the following is true:
\begin{equation}  \limsup_{\lambda \rightarrow \infty} \bigl[ \lambda^{-\varkappa_p} \; \Lambda_p(\lambda) \bigr]\;  || \mathcal{T}_{\lambda}||_{L^2(M) \rightarrow L^p(\nu)} > 0.  \label{limsup-condition-2} \end{equation}  
Here $\mathcal{T}_{\lambda}$ is the restricted spectral projector defined in \eqref{restriction-operator-def},  and $\nu$ is the push-forward measure of $\mu$ supported on $\Gamma = \mathfrak{u}(E)$, as defined in \eqref{mu-nu-relation}. The exponent $\varkappa_p$ and the blowup factor $\Lambda_p \in \mathfrak F$ are is as in \eqref{our-exponent-2} and \eqref{Lambda-in-Psi-Xi} respectively. \end{proposition} 
\begin{proof}
The overall structure of the proof of Proposition \ref{smooth-spectral-proj-prop} is very similar to \cite[Lemma 5.1]{BGT07}. However, for the sake of completeness, we include it in its entirety, since at critical junctures of the argument, well-known properties of the Lebesgue measure have to be replaced by their analogues for $\mu$. 
\vskip0.1in
\noindent The conclusion of Proposition \ref{smooth-spectral-proj-prop} will follow if we prove the existence of a constant $c > 0$ and an infinite sequence $\lambda_j \nearrow \infty$ such that 
\begin{equation}  || \mathcal{T}_{\lambda_j}||_{L^2(M) \rightarrow L^p(\nu)}^2 =   ||\mathcal T_{\lambda_j} \mathcal T_{\lambda_j}^{\ast}||_{L^{p'}(\nu) \rightarrow L^p(\nu)} \geq c \bigl[\Lambda_p(\lambda_j)\bigr]^{-2} \lambda_j^{2 \vartheta_p}  \quad \text{ for all large $\lambda_j$.} \label{normal-conclusion} 
\end{equation} 
We will construct a family of test functions $f_{\lambda}$ such that for a suitable choice of constants $0 < c < 1 < C$, 
\begin{align}
||f_{\lambda}||_{L^{p'}(\nu)} &\leq C \bigl[\Psi(C \lambda) \bigr]^{\frac{1}{p'}} \text{ for all } \lambda \geq 1,   \label{f-norm} \\ 
\bigl| \bigl| \mathcal T_{\lambda} \mathcal T_{\lambda}^{\ast} f_{\lambda}\bigr| \bigr|^p_{L^p(\nu)} &\geq c \bigl[ \Xi(C \lambda) \bigr]^{-p-1} \lambda^{2p \vartheta_p}  \text{ for infinitely many large } \lambda \geq 1. \label{op-f-norm}
\end{align} 
Clearly \eqref{f-norm} and \eqref{op-f-norm} together imply  \eqref{normal-conclusion}, $\lambda_j \nearrow \infty$ being the sequence that verifies \eqref{op-f-norm}.  
\vskip0.1in
\noindent Let us fix a point $v_0 \in E$ for which \eqref{AD-local} holds, and a small absolute constant $\sigma > 0$ soon to be specified. For $\lambda \geq 1$, we choose a test function $f_{\lambda}$ of the form 
\begin{equation}  f_{\lambda} (\mathfrak{u}(v)) = \lambda^{\frac{\alpha}{p'}} \psi (\lambda (v-v_0)). \label{test-function-f-lambda} \end{equation} 
Here $\psi: \mathbb R^d \rightarrow [0,1]$ is a smooth function, supp$(\psi) \subseteq B(0; \sigma)$, $\psi \equiv 1$ on $B(0, \sigma/2)$. 
 It follows from \eqref{mu-ball-condition} and its equivalent formulation \eqref{mu-ball-condition-2} that 
\[ ||f_{\lambda}||_{L^{p'}(\nu)}^{p'} = \int |f_{\lambda} (\mathfrak{u}(v))|^{p'} \; d\mu(v) \leq  \lambda^{\alpha}  \mu \bigl( B(v_0; \sigma \lambda^{-1} ) \bigr) \leq C \Psi(C\lambda). \] This proves \eqref{f-norm}. 
\vskip0.1in
\noindent We turn now to the proof of \eqref{op-f-norm}. Let us recall from \eqref{restriction-operator-def} that $\mathcal T_{\lambda}$ is the restriction to $\Gamma$ of the operator $\mathscr{T}_{\lambda}$ defined in \eqref{def-chi-lambda}. By Theorem \ref{t:smoothproj}, $\mathscr{T}_{\lambda}$ admits the representation given by \eqref{e:smoothproj}. As before, the operator $\mathscr{R}_{\lambda}$ will be ignored, since its operator norm is uniformly bounded in $\lambda$ and hence does not contribute to the growth we seek in \eqref{limsup-condition-2}:  \[||\mathscr{R}_{\lambda}||_{L^2(M) \rightarrow L^p(\nu)} \leq [||\mathscr{R}_{\lambda}||_{L^2(M) \rightarrow L^{\infty}(\nu)} \leq C \text{ for any $p \geq 2$.} \]
Let us continue to denote by $\mathscr{T}_{\lambda}$ the leading term in \eqref{e:smoothproj}. 
From this, we find that $\mathcal T_{\lambda} \mathcal T_{\lambda}^{\ast}$ is an integral operator of the form \eqref{normal-op}, whose integration kernel $\mathscr{K}_{\lambda}$ is given by    
\begin{equation} \begin{aligned} \mathscr{K}_{\lambda}(\mathfrak{u}(u), \mathfrak{u}(v)) = \lambda^{n-1} \int \exp \bigl[-i \lambda \bigl( d_g(\mathfrak{u}(u), y) - &d_g(\mathfrak{u}(v), y) \bigr) \bigr] \\ &  \times a_{\lambda} \bigl(\mathfrak{u}(u), y \bigr) 
\overline{a_{\lambda} \bigl(\mathfrak{u}(v), y \bigr)} \, dy,  \end{aligned} \label{normal-kernel}  \end{equation} 
for $u, v \in E$. We will show momentarily, in Lemma \ref{normal-kernel-lower-bound-lemma} below,  that there is a small constant $c_0 > 0$ such that $\mathscr{K}_{\lambda}$ obeys the pointwise bound 
\begin{equation} 
\text{Re} \bigl[ \mathscr{K}_{\lambda}(\mathfrak{u}(u), \mathfrak{u}(v)) \bigr] \geq c_0 \lambda^{n-1} \text{ for } u, v \in \text{supp}(f_{\lambda} \circ \mathfrak{u}), \label{real-K-lower} 
\end{equation} 
provided $\sigma$ is sufficiently small. Assuming this for the moment, the rest of the calculation proceeds as follows. For $\mathfrak{u}(u) \in \text{supp}(f_{\lambda})$,  
\begin{align} 
\bigl|\mathcal T_{\lambda} \mathcal T_{\lambda}^{\ast} f_{\lambda} (\mathfrak{u}(u)) \bigr| &\geq  \int \text{Re} \bigl[\mathscr{K}_{\lambda} (\mathfrak{u}(u), \mathfrak{u}(v)) \bigr] f_{\lambda}(\mathfrak{u}(v)) \, d\mu(v) \nonumber
 \\ &\geq c_0 \lambda^{n-1} \int f_{\lambda}(\mathfrak{u}(v))  \, d\mu(v) \nonumber 
 \\ &\geq c_0 \lambda^{n-1 + \frac{\alpha}{p'}} \mu \bigl[B(v_0; \sigma/(2\lambda)) \bigr], \label{last-step} 
\end{align}
where we have substituted \eqref{real-K-lower} at the second step and \eqref{test-function-f-lambda} at the third step. In view of \eqref{AD-local}, there exists $\kappa_0 > 0$ and a sequence $\lambda_j \nearrow \infty$ such that 
\begin{equation}  \mu \bigl[B(v_0; \sigma/(2\lambda_j)) \bigr] \geq \frac{\kappa_0}{2}  \Bigl( \frac{\sigma}{2\lambda_j} \Bigr)^{\alpha} \; \bigl[ \Xi(2\lambda_j/\sigma) \bigr]^{-1} \text{ for all large $j$.} \label{AD-local-2} \end{equation}   
Inserting \eqref{AD-local-2} into \eqref{last-step}, we conclude that there exist constants $c > 0$ small and $C > 0$ large, both depending only on $c_0$, $\kappa_0$ and $\sigma$ such that for all large $\lambda = \lambda_j$ and all $\mathfrak{u}(u) \in \text{supp}(f_{\lambda})$,  
 \[  \bigl|\mathcal T_{\lambda} \mathcal T_{\lambda}^{\ast} f_{\lambda} (\mathfrak{u}(u)) \bigr| \geq c \lambda^{n-1 + \frac{\alpha}{p'}} \lambda^{-\alpha} \; \bigl[\Xi(C\lambda) \bigr]^{-1} = c \lambda^{n-1 - \frac{\alpha}{p}} \; \bigl[ \Xi(C\lambda) \bigr]^{-1}.    \]
This pointwise bound on the operator allows us to estimate its $L^p(\nu)$ norm from below: 
\begin{align} 
|| \mathcal T_{\lambda} \mathcal T_{\lambda}^{\ast} f_{\lambda}||_{L^{p}(\nu)}^p &\geq c^p \lambda^{(n-1)p - \alpha} \; \bigl[ \Xi(C\lambda) \bigr]^{-p} \; \nu \bigl( \text{supp}(f_{\lambda}) \bigr) \nonumber  \\ 
&\geq c^p \lambda^{(n-1)p - \alpha}  \; \bigl[ \Xi(C\lambda) \bigr]^{-p} \; \mu \bigl[ B(v_0; \sigma \lambda^{-1}/2) \bigr] \nonumber \\ &\geq c^p \lambda^{(n-1)p - 2\alpha}   \; \bigl[ \Xi(C\lambda) \bigr]^{-p-1} \; = c^p  \; \bigl[ \Xi(C\lambda) \bigr]^{-p-1} \;  \lambda^{2p \varkappa_p}, \label{operator-f-norm}  
\end{align} 
where we have used \eqref{AD-local} again in the last inequality, in the form of \eqref{AD-local-2}. This establishes our claim \eqref{op-f-norm}, completing the proof. 
\end{proof}
\begin{lemma} \label{normal-kernel-lower-bound-lemma} 
In the notation of Proposition \ref{smooth-spectral-proj-prop}, there exist small constants $\mathfrak c_0, \sigma>0$ depending on $v_0$ but independent of $\lambda$ for which \eqref{real-K-lower} holds. 
\end{lemma} 
\begin{proof} 
Let us recall the representation of $\mathscr{K}_{\lambda}$ from \eqref{normal-kernel}. According to Theorem \ref{t:smoothproj}(\ref{t:smoothproj-parta}), the amplitude $a_{\lambda}$ is a smooth bounded function, with bounds uniform in $\lambda$, that does not vanish for $(x,y) \in S$ with $|x| < c_0' \rho$ and $d_g(x,y) \in [c_3\rho, c_4 \rho]$. Here $\rho > 0$ is a fixed small constant, and $S$ has been defined in \eqref{def-S}. Without loss of generality, we may assume that $\mathfrak{u}(v_0) \in W \cap \{x : |x| < c_0' \rho \}$ and that the intersection $S \cap \bigl[ \mathfrak{u}(v_0) \times V \bigr]$ has positive $n$-dimensional Lebesgue measure. We claim that there is an absolute constant $\mathfrak c_0 > 0$ depending on $\rho$ but independent of $\lambda$, such that
\begin{align} &\int \bigl| a_{\lambda} \bigl(\mathfrak{u}(v_0), y \bigr) \bigr|^2 \, dy =   \int a_{\lambda} \bigl(\mathfrak{u}(v_0), y \bigr)  \overline{a_{\lambda} \bigl(\mathfrak{u}(v_0), y \bigr)} \, dy \geq 4 \mathfrak c_0,  \text{ and hence } \nonumber \\  &\int \bigl| a_{\lambda} \bigl(\mathfrak{u}(v), y \bigr) \bigr|^2 \, dy \leq 8 \mathfrak c_0 \text{ for all } v \in \text{supp}(f_{\lambda} \circ \mathfrak{u}),  \text{ and } \label{local-integral-1} \\ & \text{Re}\Bigl[ \int a_{\lambda} \bigl(\mathfrak{u}(u), y \bigr)  \overline{a_{\lambda} \bigl(\mathfrak{u}(v), y \bigr)} \, dy \Bigr] \geq 2 \mathfrak c_0  \text{ for all } u, v \in \text{supp}(f_{\lambda} \circ \mathfrak{u}), \; \lambda \gg \rho^{-1}. \label{local-integral-2}
\end{align}   
We pause for a moment to justify these claims. The positivity of $\mathfrak c_0$ is assured since the map $y \mapsto a_{\lambda}(\mathfrak{u}(v_0), y)$ is not identically zero by Theorem \ref{t:smoothproj}(\ref{t:smoothproj-parta}). The remaining two relations follow from the smoothness of the integrand. 
\vskip0.1in
\noindent On the other hand, 
\begin{align} 
\Bigl| \mathscr{K}_{\lambda} \bigl(\mathfrak{u}(u), &\mathfrak{u}(v) \bigr) - \lambda^{n-1} \int a_{\lambda} \bigl(\mathfrak{u}(u), y \bigr) 
\; \overline{a_{\lambda} \bigl(\mathfrak{u}(v), y \bigr)} \, dy \Bigr| \nonumber \\ &\leq  \lambda^{n-1} \Bigl| \int \Bigl[ \exp \bigl[ -i \lambda \bigl( d_g(\mathfrak{u}(u), y) - d_g(\mathfrak{u}(v), y) \bigr) \bigr] - 1\Bigr] a_{\lambda} \bigl(\mathfrak{u}(u), y \bigr) 
\; \overline{a_{\lambda} \bigl(\mathfrak{u}(v), y \bigr)} \, dy  \Bigr| \nonumber \\ 
&\leq \lambda^{n} \int \bigl| d_g(\mathfrak{u}(u), y) - d_g(\mathfrak{u}(v), y)  \bigr| \times \bigl| a_{\lambda} \bigl(\mathfrak{u}(u), y \bigr) 
\; \overline{a_{\lambda} \bigl(\mathfrak{u}(v), y \bigr)}  \bigr| \, dy  \nonumber \\ 
&\leq \lambda^{n} d_g \bigl(\mathfrak{u}(u), \mathfrak{u}(v)  \bigr) \int  \bigl| a_{\lambda} \bigl(\mathfrak{u}(u), y \bigr) 
\; \overline{a_{\lambda} \bigl(\mathfrak{u}(v), y \bigr)}  \bigr| \, dy \nonumber \\ 
&\leq 64 \mathfrak c_0^2 \lambda^{n}  \Bigl( \frac{C\sigma}{\lambda} \Bigr) \leq \mathfrak c_0 \lambda^{n-1} \text{ for all } u, v \in \text{supp}(f_{\lambda} \circ \mathfrak{u}). \label{local-integral-3}
 \end{align} 
Here we have used the triangle inequality in $d_g$ at the third step. The bound on $d_{g}(\mathfrak{u}(u), \mathfrak{u}(v))$ in the last step follows from the fact that the diameter of supp$(f_{\lambda})$ is bounded by a constant multiple of $\sigma/\lambda$. This last step also uses \eqref{local-integral-1}, along with H\"older's inequality. The constant $\sigma$ is chosen so as to satisfy $64C \mathfrak c_0 \sigma < 1$. The desired estimate \eqref{real-K-lower} now follows from the reverse triangle inequality:  \begin{align*} \text{Re}(\mathscr{K}_{\lambda}) &=  \text{Re}\Bigl[ \lambda^{n-1} \int a_{\lambda} \bigl(\mathfrak{u}(u), y \bigr)  \overline{a_{\lambda} \bigl(\mathfrak{u}(v), y \bigr)} \, dy \Bigr] \\ &\hskip1in +  \text{Re}(\mathscr{K}_{\lambda}) - \text{Re}\Bigl[ \lambda^{n-1}\int a_{\lambda} \bigl(\mathfrak{u}(u), y \bigr)  \overline{a_{\lambda} \bigl(\mathfrak{u}(v), y \bigr)} \, dy \Bigr] \\ &\geq  \text{Re}\Bigl[ \lambda^{n-1} \int a_{\lambda} \bigl(\mathfrak{u}(u), y \bigr)  \overline{a_{\lambda} \bigl(\mathfrak{u}(v), y \bigr)} \, dy \Bigr] \\ &\hskip1in- \Bigl| \mathscr{K}_{\lambda} \bigl(\mathfrak{u}(u), \mathfrak{u}(v) \bigr) - \lambda^{n-1} \int a_{\lambda} \bigl(\mathfrak{u}(u), y \bigr) 
\; \overline{a_{\lambda} \bigl(\mathfrak{u}(v), y \bigr)} \, dy \Bigr|,  \end{align*} and then substiuting \eqref{local-integral-2} and \eqref{local-integral-3} into the last expression.    


\end{proof}

\subsection{Conditional lower bound for rough spectral projectors} 
The lower bound \eqref{limsup-condition-2} on the smooth spectral projector implies the same for its rough counterpart. 
\begin{corollary}\label{rough-spectral-proj-prop} 
In the notation and under the hypotheses of Proposition \ref{smooth-spectral-proj-prop}, we have 
\begin{equation}  \limsup_{\begin{subarray}{c} m \rightarrow \infty \\ m \in \mathbb N \end{subarray}} \bigl[ m^{-\varkappa_p} \Lambda_p(m) \bigr] \;  ||  \mathbf P_{m/2}||_{L^2(M) \rightarrow L^p(\nu)} > 0,  \label{limsup-condition} \end{equation}  
where $\mathbf P_{\lambda}$ denotes the rough spectral projector $\mathbf P_{\lambda} = \mathbf 1_{\sqrt{-\Delta} \in [\lambda, \lambda + \frac{1}{2})}$.  
\vskip0.1in 
\noindent In other words, there exists a countably infinite increasing sequence of spectral parameters $\{ \ell_k : k \geq 1\}$ (not necessarily eigenvalues) with the following properties.
\begin{enumerate}[(a)] 
\item Each $\ell_k$ is a non-negative half-integer such that $[\ell_k, \ell_k + 1/2) \cap \text{Spec}(-\Delta_g) \neq \emptyset$. 
\vskip0.1in
\item For every exponent $p \geq \max(1, p_{\ast}/2)$, one can find a constant $c = c_p > 0$ satisfying 
\begin{equation}  || \mathbf P_{\ell_k}||_{L^2(M) \rightarrow L^p(\nu)} \geq c \ell_k^{\varkappa_p} \bigl[ \Lambda_p(\ell_k) \bigr]^{-1}  \quad \text{ for all $k \geq 1$}. \label{rough-spec-proj-bound} \end{equation}  
\end{enumerate} 
\end{corollary} 
\begin{proof} 
We prove this by contradiction. If the limit superior in \eqref{limsup-condition} is zero, then 
we can find $\kappa_m \rightarrow 0$ such that 
\begin{equation} 
||\mathbf P_{\frac{m}{2}}||_{L^2(M) \rightarrow L^p(\nu)} \leq \kappa_m \bigl[ \Lambda_p(m) \bigr]^{-1} m^{\varkappa_p} \quad \text{ for all $m$}.  \label{kappam}
\end{equation} 
We will show that this condition results in the conclusion \eqref{limsup-condition-2} of Proposition \ref{smooth-spectral-proj-prop} being violated. 
\vskip0.1in
\noindent Without loss of generality and after choosing a slower decaying function if necessary, we may assume that $\lambda \mapsto \kappa_{\lambda}$ is a continuous function on $\mathbb R_{\geq 0}$ decreasing to zero at infinity, such that $f(\lambda) = \kappa_\lambda \lambda^{\varkappa_p} /\Lambda_p(\lambda)$ obeys the following properties:
\begin{itemize}
\item $f$ is increasing in $\lambda$ for sufficiently large $\lambda$,  \label{kappam-assumptions} 
\item $f$ satisfies a doubling condition, i.e., there exists a constant $C$ such that
\begin{equation} f(2 \lambda) \leq C f(\lambda) \text{ for all large $\lambda$.} \label{kappam-doubling} \end{equation} 
\end{itemize} 
\vskip0.1in 
\noindent The hypotheses above will be applied towards estimating the operator norm of $\mathcal{T}_{\lambda}$, in the following way. We observe that 
\begin{align} \text{Id} &= \sum_{m \in \mathbb Z} \mathbf P_{m/2} = \sum_{m \in \mathbb Z} \mathbf P_{m/2} \circ \mathbf P_{m/2}, \quad \text{ which implies }  \nonumber \\
\mathscr{T}_{\lambda} &= \sum_{m \in \mathbb Z} \mathbf P_{m/2} \circ \mathbf P_{m/2} \circ \mathscr{T}_{\lambda} \quad \text{ for all $\lambda \geq 1$}. \label{chi-lambda-sum}
\end{align}  
Taking the operator norm of both sides of \eqref{chi-lambda-sum} and invoking \eqref{kappam}, we arrive at the estimate 
\begin{align}
||\mathcal{T}_{\lambda}||_{L^2(M) \rightarrow L^p(\nu)}  &\leq \sum_{m \in \mathbb Z} || \mathbf P_{m/2} \circ \mathbf P_{m/2} \circ \mathcal{T}_{\lambda} ||_{L^2(M) \rightarrow L^p(\nu)} \nonumber \\ 
&\leq \sum_{m \in \mathbb Z} ||\mathbf P_{m/2}||_{L^2(M) \rightarrow L^p(\nu)} \times ||\mathbf P_{m/2} \circ \mathscr{T}_{\lambda} ||_{L^2(M) \rightarrow L^2(M)} \nonumber \\
&\leq \sum_{m \in \mathbb Z} \kappa_m \bigl[ \Lambda_p(m) \bigr]^{-1}  m^{\varkappa_p} \sup_{j: \lambda_j \in J_m} |\chi(\lambda_j - \lambda)| \nonumber \\
&\leq C_N \sum_{m \in \mathbb Z} \kappa_m \bigl[ \Lambda_p(m) \bigr]^{-1} m^{\varkappa_p} \bigl( 1 + \text{dist}(\lambda, J_m) \bigr)^{-N}. \label{Tlambda-bound} 
\end{align}   
Here $J_m$ denotes the interval $[m/2, (m+1)/2)$. The third inequality in the sequence above follows from \eqref{kappam} and the spectral theorem. The fourth inequality, which holds for any positive integer $N$, uses the fact that $\chi$ is Schwartz. We now proceed to estimate the sum in \eqref{Tlambda-bound} in three parts, depending on the relative position of the running index $m$ with respect to $m^{\ast}$, where $m^{\ast}$ denotes the unique integer such that $\lambda \in J_{m^{\ast}}$. Thus $m^{\ast} \leq 2 \lambda$. 
\vskip0.1in
\noindent  For $m \leq m^{\ast}-2$ we set $r = m^{\ast} - m$, so that $\text{dist}(\lambda, J_m) \geq \text{dist}(J_m, J_{m^{\ast}}) = (r-1)/2$. Using the fact that $f(\lambda) = \kappa_{\lambda} \lambda^{\varkappa_p} \bigl[\Lambda_p(\lambda) \bigr]^{-1}$ increases with $\lambda$, we obtain 
\begin{align} 
\sum_{m \leq m^{\ast}-2} \kappa_m \bigl[\Lambda_p(m) \bigr]^{-1} m^{\varkappa_p} \bigl( 1 + \text{dist}(\lambda, J_m) \bigr)^{-N} &\leq  f(m^{\ast}) \sum_{r \geq 2} ((r-1)/2)^{-N} \nonumber \\  &\leq C f(2 \lambda) \leq C f(\lambda) = C \kappa_{\lambda} \bigl[\Lambda_p(\lambda)\bigr]^{-1} \lambda^{\varkappa_p},  \label{small-m}
\end{align} 
where the last inequality follows from the doubling property \eqref{kappam-doubling}. For $m$ in the range $|m - m^{\ast}| \leq 1$, the estimate 
\begin{equation} 
 \sum_{m = m^{\ast}-1}^{m^{\ast} + 1} \kappa_m \bigl[ \Lambda_p(m) \bigr]^{-1} m^{\varkappa_p} \bigl( 1 + \text{dist}(\lambda, J_m) \bigr)^{-N} \leq C \kappa_{\lambda} \bigl[ \Lambda_p(\lambda) \bigr]^{-1}  \lambda^{\varkappa_p} \label{medium-m} 
 \end{equation} 
 is easy to verify from the properties of $f$; in fact, each one of the three summands is comparable to the right hand side, by the doubling property \eqref{kappam-doubling}. For $m \geq m^{\ast} + 2$, we set $r = m - m^{\ast}$, so that once again we have dist$(\lambda, J_m) \geq (r-1)/2$. Since $\kappa_m$  and $1/\Lambda_p(m)$ both decrease with $m$, this leads to 
 \begin{align} 
 \sum_{m = m^{\ast}+2}^{\infty} \kappa_m \bigl[ \Lambda_p(m) \bigr]^{-1} m^{\varkappa_p} \bigl( 1 &+ \text{dist}(\lambda, J_m) \bigr)^{-N} \nonumber \\ &\leq \kappa_{m^{\ast}} \sum_{r=2}^{\infty} \bigl[\Lambda_p(m^{\ast} + r)\bigr]^{-1} (m^{\ast} + r)^{\varkappa_p} ((r-1)/2)^{-N}  \nonumber \\   &\leq \kappa_{m^{\ast}} \bigl[\Lambda_p(m^{\ast}) \bigr]^{-1}  \sum_{r=2}^{\infty} (m^{\ast} + r)^{\varkappa_p} ((r-1)/2)^{-N} \nonumber  \\ &\leq C_N \kappa_{m^{\ast}}  \bigl[\Lambda_p(m^{\ast})\bigr]^{-1} \sum_{r=2}^{\infty} (m^{\ast})^{\varkappa_p} r^{\varkappa_p-N} \nonumber \\  &\leq C f(m^{\ast}) \leq C f(\lambda) = C \kappa_{\lambda} \bigl[ \Lambda_p(\lambda) \bigr]^{-1} \lambda^{\varkappa_p}. \label{large-m} 
\end{align} 
The third step is a consequence of the inequality $m^{\ast} + r \leq m^{\ast} r$ and the fact that $\varkappa_p \geq 0$ for $p \geq \max[1, p_{\ast}/2]$. The fourth inequality holds provided one chooses $N > \varkappa_p+1$. The subsequent inequality is justified exactly as in \eqref{small-m}. 
\vskip0.1in
\noindent Combining the estimates obtained in \eqref{Tlambda-bound}-\eqref{large-m}, we find that 
\[ ||\mathcal{T}_{\lambda}||_{L^2(M) \rightarrow L^p(\nu)} \leq C \kappa_{\lambda} \lambda^{\varkappa_{p}} \text{ for all sufficiently large $\lambda$.} \] 
Since $\kappa_{\lambda} \rightarrow 0$, this implies that \[ \limsup_{\lambda \rightarrow \infty} \bigl[ \lambda^{-\varkappa_p} \Lambda_p(\lambda) \bigr] \; ||\mathcal{T}_{\lambda}||_{L^2(M) \rightarrow L^p(\nu)}  = 0, \] contradicting the conclusion of Proposition \ref{smooth-spectral-proj-prop}. 
\end{proof}

\section{Conditional sharpness: Proof of Theorem \ref{mainthm-sharpness} and Corollary \ref{maincor-sharpness}} \label{sharpness-section} 

\begin{proof}[{\bf{Proof of Theorem \ref{mainthm-sharpness} and Corollary \ref{maincor-sharpness} }}] 
For $M = \mathbb S^n$, we have explicit information about the location of the eigenvalues of the Laplace-Beltrami operator: 
\[ \text{Spec}(-\Delta_{\mathbb S^n}) = \bigl\{ \lambda_{\ell}^2 = \ell(\ell + n-1) : \ell \in \mathbb N  \bigr\}. \] 
The gap between $\lambda_{\ell+1}$ and $\lambda_{\ell}$ approaches 1 as $\ell \rightarrow \infty$. Thus for any large integer $m \in \mathbb N$, the interval $[m/2, m/2+1/2)$ contains at most one value of $\lambda_{\ell}$. Thus in this context, Proposition \ref{rough-spectral-proj-prop} provides an increasing sequence of positive half-integers $\ell_k$ such that each interval $[\ell_k , \ell_k + 1/2)$ contains the square root of exactly one eigenvalue, say $\lambda_{j_k}$. Since
$\mathbf P_{\ell_k} \varphi_{j_k} = \varphi_{j_k}$ for any eigenfunction $\varphi_{j_k}$ associated with the eigenvalue $\lambda_{j_k}$, the desired conclusion \eqref{sphere-lower-bound} follows from \eqref{rough-spec-proj-bound} for the subsequence $\varphi_{j_k}$ of $L^2$-normalized spherical harmonics.   
\end{proof} 

\begin{proof}[{\bf{Proof of Corollary \ref{maincor-sharpness}}}] 
The corollary is an easy consequence of Theorem \ref{mainthm-sharpness}. Indeed, the hypotheses of the corollary coincide with the hypotheses of the theorem, with $\alpha = d(1 - \epsilon) - \kappa$, $\Psi = \Psi_{\kappa}$ and $\Xi = \Xi_{\kappa}$. Hence by Theorem \ref{mainthm-sharpness} there exists a $\kappa$-dependent choice of $L^2$-normalized spherical harmonics $\{ \varphi_{\lambda_j} : j \geq 1\}$ for which the lower bound \eqref{sphere-lower-bound} holds for all $p \geq \max(1, p_{\ast}/2)$, with $c = c(\kappa, p)$. In the range $p \geq \max(2, p_{\ast})$, 
\[ \varkappa_p = \vartheta_p = \frac{n-1}{2} - \frac{\alpha}{p} = \theta_p + \frac{\kappa}{p}, \]
completing the proof. 
\end{proof}

\section{Cantor-type Sets and Constructions} \label{s:randomcantor}
\noindent Our next task is to prove Theorem \ref{mainthm-example} and Corollary \ref{maincor-example}. These results aim to show that a stronger version of eigenfunction growth of the form \eqref{eigenfunction-restriction-estimate} holds generically, with a single measure $\nu$ and $\kappa = 0$, for a large class of sets $E$. This section and the next are given over to the construction of this family of sets. 
\subsection{A general Cantor-type construction} \label{section-basic}
\noindent All the subsets $E \subseteq [0,1]^d$ considered in this paper are obtained using a Cantor-type iteration, whose basic features appear in \cite[Section 2]{LP11}.  We recall the important points here, referring the reader to \cite{LP11} for a more detailed discussion. There are two main ingredients in the construction; namely, a choice of successive scales and a selection mechanism at each scale. 
\vskip0.1in
\noindent Fix a sequence of positive integers $\{ N_k : k \geq 1 \}$ with $N_k \geq 2$ for all $k \geq 1$. Set 
\begin{equation} \delta_k^{-1} = M_k = N_1 N_2 \dots N_k. \label{def-deltak} \end{equation}  Using the notation $\mathbb Z_m := \{1, \cdots, m \}$, we define a class of multi-indices with integer entries:
\begin{align} \mathbb I(k, d) &:= \bigl\{\mathbf i_k = (\overline{i}_1,\dots, \overline{i}_k); \; \overline{i}_j \in \mathbb Z_{N_j}^d, \; 1 \leq j \leq k \bigr\} \subseteq \mathbb Z^{dk}, \text{ and } \label{def-Ikd} \\
 \mathbb I^{\ast} &:= \bigcup \bigl\{ \mathbb I(k, d): k \geq 1 \bigr\}. \label{def-Istar} 
 \end{align} 
The interpretation of the integers $N_k$ and the multi-indices $\mathbf i_k$ is the following. At step $k$, the unit cube $[0,1]^d$ is partitioned into subcubes of sidelength $\delta_k$ with sides parallel to the coordinate axes. These subcubes, which we term {\em{cubes of the $k$-th generation}}, are indexed by $\mathbf i_k$. Each such cube is of the form 
\begin{align} 
Q(\mathbf i_k) &= \alpha(\mathbf i_k) + [0, \delta_k]^d, \quad \text{ with } \nonumber \\  
 \alpha(\textbf{i}_k) &= \frac{\overline{i}_1 - \bar{1}}{N_1} + \frac{\overline{i}_2 - \bar{1}}{N_1 N_2} + \dots + \frac{\overline{i}_k -  \bar{1}}{N_1 \dots N_k}. \label{digit-expansion}  
\end{align}   
Here $\bar{1} = (1, \ldots, 1) \in \mathbb R^d$. The expression \eqref{digit-expansion} above should be thought of as a finite ``digit expansion'' of $\alpha(\mathbf i_k)$ with respect to the base string $(N_1, N_2, \cdots)$. Every point in the unit cube has a possibly infinite digit expansion with respect  to this base sequence. Further, such a digit expansion is unique, except for countably many points in the unit cube. We note that the cubes of any given generation have disjoint interiors. Further, each $k$-th generation cube gives rise to exactly $N_{k+1}^d$ children, as follows: \[ Q(\mathbf i_k) = \bigcup \bigl\{ Q(\mathbf i_k, \overline{i}_{k+1}) : \; \overline{i}_{k+1} \in \mathbb Z_{N_{k+1}}^d \bigr\}. \] 
Thus any two distinct cubes $Q(\mathbf i)$ and $Q(\mathbf j)$ with $\mathbf i, \mathbf j \in \mathbb I^{\ast}$ must satisfy exactly one of the relations
\[ Q(\mathbf i) \subsetneq Q(\mathbf j), \; \text{ or } \; Q(\mathbf j) \subsetneq Q(\mathbf i), \;\text{ or } \; \text{int}(Q(\mathbf i)) \cap \text{int}(Q(\mathbf j)) = \emptyset. \] 

\vskip0.1in
\noindent To specify a selection algorithm, we fix for each $k \geq 1$ a sequence $\mathbf Y_k := \{Y_k(\mathbf i_k); \mathbf i_k \in \mathbb I(k,d)\}$ whose elements are either 0 or 1.  In other words, $\mathbf Y_k$ is a finite binary string indexed by $\mathbb I(k, d)$. Set $X_1(\overline{i}_1) := Y_1(\overline{i}_1)$, 
\begin{align} 
X_{k}({\textbf{i}}_{k}) &:= X_{k-1}(\textbf{i}_{k-1}) Y_{k}({\textbf{i}}_{k})  \text{ where } {\textbf{i}}_{k}=(\mathbf i_{k-1}, \overline{i}_{k}), \label{X-and-Y} \\ 
\textbf{X}_k &:= \{ X_k(\textbf{i}_k) : \mathbf i_k \in \mathbb I(k,d) \}, \quad 
 P_k := \#\{ \textbf{i}_k : X_k(\textbf{i}_k) = 1\}, \nonumber \\
 \mathcal Q_k &:= \{ Q(\mathbf i_k) : \mathbf i_k \in \mathbb I(k,d), X_k(\mathbf i_k) = 1 \}, \quad \mathcal Q^{\ast} := \bigcup_k \mathcal Q_k.  \label{basic-cubes-collection} 
\end{align}
The relevance of these definitions is the following. A total of $P_k$ cubes of the $k$-th generation are chosen at step $k$, the marker of selection being $X_k(\mathbf i_k)=1$. We call the selected ones the {\em{basic cubes}} of the $k$-th generation. The collection $\mathcal Q_k$, which is indexed by $\mathbf i_k$ with $X_k(\mathbf i_k) = 1$, specifies the cubes $Q(\mathbf i_k)$ that are selected. If $X_k(\mathbf i_k) = 0$, then  so is $X_{\ell}(\mathbf i_{\ell})$ for any $\ell > k$ with $\pi_k(\mathbf i_{\ell}) = \mathbf i_k$, by \eqref{X-and-Y}. Here $\pi_k$ denotes the projection onto the first $k$ vector coordinates in $\mathbb R^d$. Thus, once a cube is discarded at a given step, its descendants are eliminated from consideration for the remainder of the construction. The union of the cubes in $\mathcal Q_k$ therefore gives rise to a decreasing sequence of closed sets.    
\vskip0.1in
\noindent Figure 2 depicts the first two steps of such a construction, with $d=1$. At the first step, only the first and fourth intervals of length $1/N_1$ are chosen, corresponding to the  selection choices: $X_1(i_1) = 1$ if $i_1 = 1, 4$ and $X_1(i_1) = 0$ otherwise. At the second step, we have picked $N_2 = 7$, and the basic intervals of the second generation, highlighted in thickened lines, are chosen based on $X_2(1, i_2) = 1$ if and only if $i_2 = 1, 3,4, 7$; $X_2(4, i_2) = 1$ if and only if $i_2 = 2, 3, 6$. 
\vskip0.1in  
\noindent Given these quantities, we define the successive nested iterates $E_k$ of the construction, and the limiting set $E$:  
\begin{equation}
E_0 := [0,1]^d, \quad E_k := \bigcup_{X_k(\textbf{i}_k) = 1} Q(\textbf{i}_k) = \bigcup_{Q \in \mathcal Q_k} Q , \quad
E := \bigcap_{k=1}^{\infty} E_k \subseteq [0,1]^d. \label{def-E} 
\end{equation}
We observe that $|E_k| = P_k \delta^d_k$; so, if $P_k \delta_k^d \rightarrow 0$, then $E$ is a Lebesgue-null set. Our current hypotheses do not require $P_k \delta_k^d$ to go to zero; so in principle $E$ could have positive Lebesgue measure. However, the convergence to zero will hold almost surely in our random construction in the next section (see the remark following Lemma \ref{PkRk-lemma} on page \pageref{Lebesgue-null-remark}). On the other hand, 
\begin{equation} \label{E-not-empty-nasc} 
E \neq \emptyset \text{ if and only if } P_k \ne 0 \text{ for each } k \geq 1.
\end{equation}

\begin{figure} 
\begin{tikzpicture}[scale=5] 

\draw [|-]  (0,0) node [anchor = south east] {\small $0$}  -- (4/7,0);
\foreach \x in {1/7,2/7,3/7,4/7, 6/7}
\draw (\x,-0.5pt) -- (\x,0.5pt);

\foreach \x in {1,2,3}
\draw (\x/7, -1pt) node[anchor= north] {\tiny $\frac{\x}{N_1}$};

\draw (6/7, -1pt) node[anchor= north] {\tiny 1 - $\frac{1}{N_1}$};

\draw [dotted] (4/7,0) -- (6/7,0);
\draw [-|] (6/7,0) -- (1,0) node [anchor = south west] {\small $1$};

\draw [magenta] (-0.3,0.1) -- (-0.31,0.1) -- (-0.31,-0.0)
node[anchor=east]{\small Step 1} -- (-0.31,-0.1) -- (-0.3,-0.1);

\draw [blue, ultra thick]  (0,0) -- (1/7,0);
\draw [blue, ultra thick]  (3/7,0) -- (4/7,0);


\draw [magenta] (-0.3,-0.2) -- (-0.31,-0.2) -- (-0.31,-0.45)
node[anchor=east]{\small Step 2} -- (-0.31,-0.7) -- (-0.3,-0.7);

\def\endA{(-0.25,-0.28)}

\draw [green, ultra thick]  (1/14,0) ellipse (1/12 and 1/16);
\draw [->] [green, ultra thick] (1/14 - 1/12,0) -- \endA;

\draw [|-|] \endA ++(0,0) -- ++(1,0);

\foreach \x in {1/7,2/7,3/7,4/7,5/7,6/7}
\draw \endA  +(\x,-0.5pt) -- +(\x,0.5pt);

\draw \endA ++(1.5/7,0) node [anchor = north] {\tiny $\frac{1}{N_1 N_2}$};

%

\draw [blue, ultra thick] \endA  -- ++(1/7,0);
\draw [blue, ultra thick] \endA  ++(2/7,0) -- ++(2/7,0);
\draw [blue, ultra thick] \endA  ++(6/7,0) -- ++(1/7,0);


\def\endB{(0.5,-0.5)}

\draw [green, ultra thick]  (7/14,0) ellipse (1/12 and 1/16);
\draw [->] [green, ultra thick] (1/2 - 1/12,0) -- \endB;

\draw [|-|] \endB ++(0,0) -- ++(1,0);

\foreach \x in {1/7,2/7,3/7,4/7,5/7,6/7}
\draw \endB  +(\x,-0.5pt) -- +(\x,0.5pt);

\draw \endB ++(2.5/7,0) node [anchor = north] {\tiny $\frac{1}{N_1 N_2}$};

%

\draw [blue, ultra thick] \endB  ++(1/7,0) -- ++(2/7,0);
\draw [blue, ultra thick] \endB  ++(5/7,0) -- ++(1/7,0);

\end{tikzpicture}
\caption{The first two steps of a general Cantor construction} \label{cantor-picture} 
\end{figure}
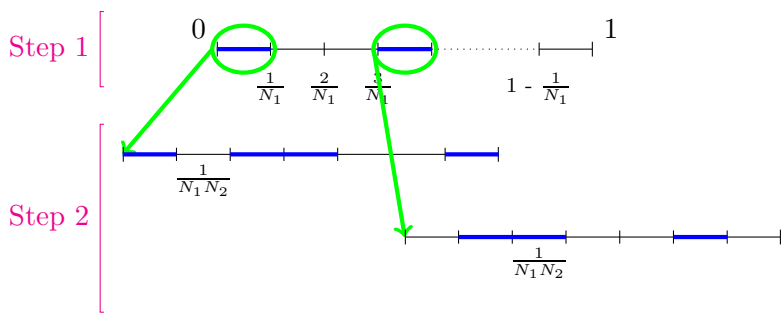

\subsection{A Cantor measure on $E$} \label{Cantor-measure-section} Our next task is to define a probability measure $\mu$ on $E$, which will be used to verify the crucial assumptions \eqref{mu-ball-condition-2} and \eqref{AD-local-3}. 
For each $k \geq 1$, we define the function
\begin{equation} \label{def-muk} 
\mu_k := {1_{E_k}}/{|E_k|} = \frac{1}{P_k \delta_k^d} \sum_{\mathbf i_k} X_k(\mathbf i_k) 1_{Q(\mathbf i_k)},  
\end{equation} 
which is the density function of the normalized Lebesgue measure on $E_k$. The Banach-Alaoglu theorem ensures that any family of probability measures on a locally compact space has limit points in the weak-$\ast$ topology. However, apriori one does not know whether the limit point is unique, or even if it is, whether the limiting measure is a probability measure. The following result provides a sufficient condition for the sequence $\{ \mu_k: k \geq 1\}$ in \eqref{def-muk} to be weak-$\ast$ convergent, and have a probability measure $\mu$ as the weak-$\ast$ limit. 
\begin{lemma} \label{weak-star-limit-lemma} 
Suppose that the distribution of the chosen cubes $\{Q(\mathbf i_k) : X_k(\mathbf i_k) = 1 \}$ within $E_{k-1}$ is approximately uniform in the following sense:
\begin{equation} \label{weak-star-limit-condition} 
\sup_{k': k' \geq k} \sum_{\begin{subarray}{c} \mathbf i_k \\ X_k(\mathbf i_k) = 1 \end{subarray}} \Bigl| \int_{Q_k(\mathbf i_k)} \bigl[ \mu_{k'} - \mu_k \bigr](x) \, dx \Bigr| \longrightarrow 0 \text{ as } k \longrightarrow \infty.
\end{equation} 
Then there exists a probability measure $\mu$ supported on $E$ such that $\mu_k \longrightarrow \mu$ in the weak-$\ast$ topology, i.e., for all continuous functions $f: [0,1]^d \rightarrow \mathbb C$, 
\begin{equation}  \int f \, d\mu_k \longrightarrow \int f \, d\mu \quad \text{ as } k \longrightarrow \infty. \label{conv-cts-fns} \end{equation}  
\end{lemma} 
\begin{proof}  
The main point here is to prove the convergence \eqref{conv-cts-fns}. Once this is proved, the fact that the limit $\mu$ is a probability measure follows easily by setting $f \equiv 1$. In order to show that $\mu$ is supported on $E$, we observe from the construction that $[0,1]^d \setminus E$ is a countable union of open cubes of the form $U$ such that the closure $\overline{U} = Q(\mathbf i)$, for some $\mathbf i \in \mathbb I^{\ast}$. Suppose that $U = \text{int}(Q(\mathbf i))$ for some $\mathbf i \in \mathbb I(k, d)$. Then $X_k(\mathbf i) = 0$, hence \eqref{def-muk} dictates that $\mu_{\ell}(U) = 0$ for every $\ell \geq k$. The weak-$\ast$ convergence \eqref{conv-cts-fns} then ensures that $\mu(U) = 0$, completing the claim.  
\vskip0.1in
\noindent  The proof of  \eqref{conv-cts-fns} in general dimensions is an adaptation of its univariate counterpart in \cite[Lemma 2.2]{LP11}; we recall it here. 
Since \[ \bigl|\int f d\mu_k \bigr| \leq ||f||_{\infty}  \quad \text{ for all $k$}, \]  
the existence of the limiting measure $\mu$ would follow from the Riesz representation theorem, provided the sequence $\{ \int f d\mu_k : k \geq 1\}$ is convergent for every continuous function $f$ on $[0,1]^d$, or equivalently, if it is Cauchy. 
\vskip0.1in
\noindent In order to prove the Cauchy property, let us fix $f \in C[0,1]^d$, $f \not\equiv 0$, and $\epsilon > 0$. Since $f$ is uniformly continuous, there exists $\delta > 0$ such that 
\begin{equation}  |f(x) - f(y)| < \frac{\epsilon}{4} \quad \text{ whenever } \quad |x-y| < \delta. \label{unif-cty} \end{equation} 
In view of \eqref{weak-star-limit-condition}, there also exists an integer $K \geq 1$ such that $\delta_K < \delta$ and 
\begin{equation}   \sup_{k': k' \geq k} \sum_{\begin{subarray}{c} \mathbf i_k \\ X_k(\mathbf i_k) = 1 \end{subarray}} \Bigl| \int_{Q_k(\mathbf i_k)} \bigl[ \mu_{k'} - \mu_k \bigr](x) \, dx \Bigr| < \frac{\epsilon}{2 ||f||_{\infty}} \quad \text{ for all } k \geq K. \label{weak-star-2} \end{equation} 
Then for all $k' \geq k \geq K$, we find that 
\begin{align*}
\Bigl| \int f(x)  \bigl[ \mu_{k'} - \mu_k \bigr](x) \, dx  \Bigr| &= \sum_{\begin{subarray}{c} \mathbf i_k \\ X_k(\mathbf i_k) = 1 \end{subarray}} \Bigl| \int_{Q_k(\mathbf i_k)} f(x)  \bigl[ \mu_{k'} - \mu_k \bigr](x) \, dx  \Bigr| \\
&\leq \sum_{\begin{subarray}{c} \mathbf i_k \\ X_k(\mathbf i_k) = 1 \end{subarray}} \Bigl| \int_{Q_k(\mathbf i_k)} \Bigl[ f(x) - f \bigl( \alpha(\mathbf i_k) \bigr)  \Bigr] \bigl[ \mu_{k'} - \mu_k \bigr](x) \, dx  \Bigr| \\ 
&\hskip1in+  \sum_{\begin{subarray}{c} \mathbf i_k \\ X_k(\mathbf i_k) = 1 \end{subarray}} \Bigl| f(\alpha(\mathbf i_k)) \int_{Q_k(\mathbf i_k)}   \bigl[ \mu_{k'} - \mu_k \bigr](x) \, dx  \Bigr| \\
&\leq \frac{\epsilon}{4} \sum_{\begin{subarray}{c} \mathbf i_k \\ X_k(\mathbf i_k) = 1 \end{subarray}}  \int_{Q_k(\mathbf i_k)} \bigl[ \mu_{k'} + \mu_k \bigr](x) \, dx  \\ &\hskip1in + ||f||_{\infty} \sum_{\begin{subarray}{c} \mathbf i_k \\ X_k(\mathbf i_k) = 1 \end{subarray}} \Bigl| \int_{Q_k(\mathbf i_k)} \bigl[ \mu_{k'} - \mu_k \bigr](x) \, dx \Bigr| \\ 
&\leq \frac{\epsilon}{4} \times 2 + ||f||_{\infty} \frac{\epsilon}{2 ||f||_{\infty}} = \epsilon.
\end{align*}
We have used \eqref{unif-cty} and \eqref{weak-star-2} in the third and fourth steps. This concludes the proof.  
\end{proof} 
\vskip0.1in
\noindent We pause for a moment to discuss the measure $\mu$ in a bit more detail. A general Cantor-type construction as described above permits many measures to be defined on the resulting set, some of which are natural and have applications in geometric measure theory.  The measure $\mu$ given by Lemma \ref{weak-star-limit-lemma}, if it exists, is one of them. Yet another natural measure, say $\widetilde{\mu}$, is provided by the classical ``mass distribution principle" \cite[Proposition 1.7]{F90}. We briefly compare the two procedures. In the construction of $\widetilde{\mu}$, the mass assigned to each selected basic cube is divided equally between its selected children in the next generation. This is different from the construction of $\mu$, where at each stage, the unit total mass is divided equally between {\em{all}} the selected basic cubes of the next generation. This means that if two basic cubes $Q(\mathbf i_k)$ and $Q(\mathbf i'_k)$ of the $k^{\text{th}}$ generation have the same mass but different number of survivors at the next stage, then the mass assigned by $\widetilde{\mu}$ to a child of $Q(\mathbf i_k)$ could be very different from that of $Q(\mathbf i_k')$. This gives rise to an inherent non-uniformity in the mass distribution of $\widetilde{\mu}$ across its support. In contrast, every selected basic child of $Q(\mathbf i_k)$ and $Q(\mathbf i_k')$ at the $(k+1)^{\text{th}}$ level gets the same mass, namely $1/P_{k+1}$, under $\mu$. Thus the two measures $\mu$ and $\widetilde{\mu}$ are in general distinct, though they coincide in the special case when every basic cube in any generation has the same number of selected children in the next.  This would happen, for instance, in the construction of the classical Cantor middle-third set.
\vskip0.1in
\noindent The advantage of $\widetilde{\mu}$ is that, unlike $\mu$, it is guaranteed to exist for every Cantor-like construction, as shown in \cite[Proposition 1.7]{F90}. However, the construction of $\widetilde{\mu}$ is not in general mass-preserving, since a basic cube with no survivors loses its mass. Thus if $E \ne \emptyset$, the measure $\widetilde{\mu}$ supported on $E$ would have positive but possibly arbitrarily small total mass $<1$. More important for us is the following point: even in the case of no die-out so that $\widetilde{\mu}$ is a probability measure, the critical volume growth criteria of the type \eqref{mu-ball-condition-2} and \eqref{AD-local-3} are not in general true for $\widetilde{\mu}$, since the mass assigned to balls could in principle vary widely within its support. For instance, let us consider the following example in $d = 1$:
\begin{align*} &N_k = 2 \text{ for all $k \geq 1$}, \quad Y_k(\mathbf i_k) = \left\{ \begin{aligned} 1 &\text{ if } k =1, \\ 1 &\text{ if } i_1=1 \text{ and } k \geq 2, \\  1 &\text{ if } i_1 = 2 \text{ and } i_2 = \cdots= i_k = 1, k \geq 2, \\  0 &\text{ otherwise. } \end{aligned}  \right\} \text{ so that }  \\ 
&E_0 = E_1 = [0,1], \; E_k = \Bigl[0, \frac{1}{2} \Bigr] \cup \Bigl[\frac{1}{2}, \frac{1}{2} + \frac{1}{2^{k}} \Bigr] \text{ for $k \geq 2$}, \; E = \Bigl[0, \frac{1}{2} \Bigr], \; \widetilde{\mu} = 1_{[0, \frac12]} + \frac{1}{2} \delta_{\frac12}.
\end{align*} 
In other words, $\widetilde{\mu}$ is a probability measure that assigns half its mass uniformly to the interval $[0, \frac12]$ and the other half as a point mass to the point $\frac12$. While \[\alpha = \dim_{\mathbb H}  \bigl[\text{supp}(\widetilde{\mu}) \bigr] = 1,  \]  the ball growth condition \eqref{mu-ball-condition} holds only with $\alpha = 0$ and $\Psi \equiv 1$, due to the presence of the Dirac delta. As a consequence, the measure $\widetilde{\mu}$ under-performs significantly in terms of eigenfunction restriction estimates given the dimension of its support.  Specifically, a sharp estimate of the form \eqref{eigenfunction-growth-nu} holds for $\widetilde{\nu} = \widetilde{\mu} \circ \mathfrak u^{-1}$, with the growth exponent $\vartheta_p$ in that estimate replaced by the much larger exponent $\widetilde{\vartheta}_p = (n-1)/2$, as dictated by the Weyl law. 
\vskip0.1in 
\noindent The measure $\mu$, on the other hand, need not exist without additional assumptions, such as the one in Lemma \ref{weak-star-limit-lemma}.  But when it does, it enjoys the feature of uniform mass distribution across its support that closely emulates the corresponding feature of Euclidean Lebesgue measure, an essential component of the analytical machinery of this article. 
\subsection{Hausdorff dimension} The set $E$ defined in \eqref{def-E} obeys certain dimensionality bounds given in terms of the construction parameters. These bounds are presented in \cite[Lemma 2.1]{LP11} for the case $d=1$. The proof for a general $d$ involves only minor variations, but we include it for completeness. 
\begin{lemma} \label{l:dimension}
Let $\dim_{\mathbb H}(E)$ denote the Hausdorff dimension of the set $E$ constructed in Section \ref{section-basic}. 
\begin{enumerate}[(a)]
\item Then 
\begin{equation} \dim_{\mathbb H}(E) \leq \liminf_{k \rightarrow \infty} \frac{\log(P_k)}{-\log(\delta_k)}. \label{Hdim-upper-bound} \end{equation} 
\item  Suppose further that the sequence $\mu_k$ given by \eqref{def-muk} converges to a probability measure $\mu$ in the weak-$\ast$ topology. Then 
\begin{equation} \dim_{\mathbb H}(E) \geq s_0 := \liminf_{k \rightarrow \infty} \frac{\log(P_k/N^d_k)}{- \log(\delta_{k-1})}. \label{Hdim-lower-bound} 
\end{equation} 
\end{enumerate} 
\end{lemma}
\begin{proof}
The relation \eqref{Hdim-upper-bound} follows immediately from a standard result, namely that the Hausdorff dimension is bounded above by the lower box-counting dimension; see  Proposition 4.1 in \cite{F90}.  The proof of \eqref{Hdim-lower-bound} relies on Frostman's lemma \cite[Theorem 8.8]{M95} or \cite[Section 4.1]{F90}, which has already appeared in the proof of Theorem \ref{mainthm-restriction-arbitrary}. According to this lemma, the claim \eqref{Hdim-lower-bound} would follow if we prove the following: for every $s < s_0$, there exists a probability measure $\tau^{[s]}$ supported on $E$ and a positive, finite constant $C = C_s$ such that the ball condition on the right hand side of \eqref{Frostman} holds. 
Since any ball of radius $r$ can be covered by an axes-parallel cube of sidelength $2r$,  
this amounts to proving that the estimate \begin{equation} \tau^{[s]}(J) \leq C_s r^{s} \label{to-show-dim} \end{equation} holds for all axes-parallel cubes $J$ of sidelength $r$. If $\mu = \lim_{k \rightarrow \infty} \mu_k$ exists in the weak-$\ast$ topology, we will prove that \eqref{to-show-dim} holds with $\tau^{[s]} = \mu$, for all $s < s_0$. 
\vskip0.1in
\noindent To this end, we fix a small number $0 < r \leq 1$ and let $k = k(J)$ denote the unique index such that $\delta_{k+1} \leq r < \delta_k$.  The number of basic cubes in $\mathcal{Q}_{k+1}$ that can intersect $J$ is either: (i) at most $2^d N^d_{k+1}$ as $J$ may intersect at most $2^d$ adjacent cubes in $\mathcal{Q}_k$, or (ii) the natural upper bound of $|J|/\delta_{k+1}^d$, since the cubes in $\mathcal{Q}_{k+1}$ have disjoint interiors. From the definition \eqref{def-muk}, we see that $\mu_{k+1}$ assigns each cube in $\mathcal Q_{k+1}$ a mass of $1/P_{k+1}$. Therefore, 
\begin{align}
\mu_{k+1}(J) &\leq 
P_{k+1}^{-1} \min \left[ 2^d N_{k+1}^d, \frac{|J|}{\delta_{k+1}^d} \right] \nonumber \\
& \leq P_{k+1}^{-1}  \left( 2^d N_{k+1}^d \right)^{1-\theta} \left( \frac{|J|}{\delta_{k+1}^d} \right)^{\theta} \nonumber \\ 
&\leq 2^d \frac{N_{k+1}^d}{P_{k+1} \delta_k^{d \theta}} r^{d \theta}, \text{ hence } \nonumber \\ 
\frac{\mu_{k+1}(J)}{r^{d \theta}} &\leq  2^d \frac{N_{k+1}^d}{P_{k+1} \delta_k^{d \theta}}.  \label{towards-Frostman}
\end{align}
Here $0 \leq \theta \leq 1$ is a constant to be determined shortly. Setting $s = d\theta$, we observe that the quantity ${N_{k+1}^d}/({P_{k+1} \delta_k^{s}})$ on the right hand side is uniformly bounded in $k$ precisely when $s < s_0$. Letting $k \rightarrow \infty$ in \eqref{towards-Frostman} completes the proof of \eqref{to-show-dim}. 
\end{proof}

\section{Random Cantor sets} \label{random-cantor-section} 
\noindent We now delve into the probabilistic construction that generates our desired Cantor-type sets.  The basic procedure is as in Section \ref{section-basic}, with the crucial additional point that the sequence $\mathbf X_k$ is now randomized. Recall the definitions of $M_k, N_k$ and $\mathbb I(k,d)$ from \eqref{def-Ikd} and the discussion preceding it. 
\vskip0.1in
\subsection{Construction of the measure space $(\Omega, \mathcal B, \mathbb P^{\ast})$ in Theorem \ref{mainthm-example}} \label{measure-space-section}  The measure space under consideration is 
\[ \Omega = \prod_{k=1}^{\infty} \Omega_k, \quad \text{ where }  \quad \Omega_k = \prod_{\mathbf i_k \in \mathbb I(k,d)} \{ 0, 1\} = \{0, 1\}^{\mathbb I(k, d)}.  \]
Thus $\Omega_k$ consists of all binary strings of length $M_k^d$. We denote a ``random'' string in $\Omega_k$ by $\mathbf Y_k = \{Y_k(\mathbf i_k) : \mathbf i_k \in \mathbb I(k,d) \}$. The set $\Omega$ is the collection of all infinite binary strings $\omega = (\mathbf Y_1, \mathbf Y_2, \cdots)$, with $\mathbf Y_k \in \Omega_k$. Clearly, $\Omega$ is uncountable, with cardinality same as that of the continuum. As  described in Section \ref{section-basic}, every $\omega \in \Omega$ generates a Cantor-type set $E(\omega)$. At this stage, depending on the choice of $\omega$, some of the sets $E(\omega)$ could be empty. 
\vskip0.1in
\noindent We now assign an auxiliary probability measure $\mathbb P$ to $\Omega$, which will be related, but not identical to, the final probability measure $\mathbb P^{\ast}$ that we put on it. For a sequence of small positive numbers $\{\epsilon_k : k \geq 1\}$ to be specified, set  
\begin{equation} p_k = N_k^{- d \epsilon_k} \quad \text{ with the requirement that } p_k \leq \frac{1}{2}.  \label{def-pk}
\end{equation}  The numbers $p_k$ will serve as the ``selection probability''  at step $k$. For each $k \geq 1$ and $\mathbf i_k \in \mathbb I(k,d)$, the two-point set $\Xi_k(\mathbf i_k) = \{0, 1\}$ is endowed with the probability measure $\lambda_k$, where  
\begin{align*}  \lambda_k (\{1\}) &= \text{ probability of the event } \{Y_k(\mathbf i_k) = 1 \} = p_k, \quad \text{ and } \\ \lambda_k(\{0\}) &= \text{ probability of the event } \{Y_k(\mathbf i_k) = 0 \} = 1-p_k. \end{align*}   
The space $\Omega_k$ is equipped with a product probability measure that is a finite $\# (\mathbb I(k,d))$-fold Cartesian product of the measures $\lambda_k$, with one copy of $\lambda_k$ for each $\mathbf i_k \in \mathbb I(k,d)$. In other words, 
\[\mathbb P_k = \prod_{\mathbf i_k \in \mathbb I(k,d)} \lambda_k.\] This means that for every binary string $\eta = (\eta_1, \cdots, \eta_{M_k^d}) \in \Omega_k$, 
\begin{align*}   
\mathbb P_k(\{\eta\}) &= \text{probability of } \{\mathbf Y_k = \eta \} = p_k^{|\eta|} (1 - p_k)^{M_k^d - |\eta|} \text{ where } \\ |\eta| &= \eta_1 + \cdots + \eta_{M_k^d} = \text{number of 1-s in the string $\eta$}. \end{align*} 
The underlying $\sigma$-algebra for $\mathbb P_k$ is of course $\mathcal B_k$, the power set of $\Omega_k$. Finally, the measure $\mathbb P$ on $\Omega$ is the product probability measure of the measures $\mathbb P_k$. It is defined on the $\sigma$-algebra $\mathcal B$, the product $\sigma$-algebra with components $\mathcal B_k$. 
\vskip0.1in
\noindent In summary, the random vectors $\{ \mathbf Y_k : k \geq 1\}$ across different scales are independent. The scalar entries $\{ Y_k(\mathbf i_k) : \mathbf i_k \in \mathbb I(k,d) \}$ within a single scale $\mathbf Y_k$ are independent as well; in addition, within a single scale $k$, the individual entries $Y_k(\mathbf i_k)$ are identically distributed as Bernoulli random variables with success probability $p_k$. On the other hand, it is important to note that the random variables $\{\mathbf X_k : k \geq 1\}$ defined as in \eqref{X-and-Y} are {\em{not}} independent. 
\vskip0.1in
\noindent Before proceeding further, we need to ensure that the limiting sets $E = E(\omega)$ obtained in this manner are nonempty, with nonzero probability.
\begin{lemma} \label{nonemptyE-lemma}
Assume that the construction parameters $N_k$ and $\epsilon_k$ defining \eqref{def-pk} are chosen so that \begin{equation} \sum_{k=1}^{\infty} (1-p_k)^{N_k^d} < 1. \label{nonemptyE-condition} \end{equation}  
Then for the construction described above, $\mathbb P(E \text{ is nonempty}) > 0$. 
\end{lemma}    
\begin{proof} 
It suffices to show that the probability of the complementary event, namely when $E$ is empty, is bounded from above by the left hand side of \eqref{nonemptyE-condition}. Accordingly, we express this event as a disjoint union:
\begin{align*}
\{E \text{ is empty} \} &= \bigcup_{k=1}^{\infty}\bigl\{ \text{$\exists k$ such that $E_k $ is empty}\bigr\} \\ 
&= \bigcup_{k=1}^{\infty} \bigl\{ P_{k} = 0 \text{ but } P_{k-1} > 0 \bigr\} \\ 
&\subseteq \bigcup_{k=1}^{\infty} \bigl\{\exists \; \mathbf i \in \mathbb I(k-1,d) \text{ such that } Y_k(\mathbf i_k) = 0 \; \forall \; \mathbf i_k = (\mathbf i, i_k) \in \mathbb I(k,d)  \bigr\}. 
\end{align*} 
The $k$-th event in the last union of sets can happen only if there are at least $N_k^d$ independent Bernoulli random variables $Y_k(\mathbf i_k)$ at the $k$-th stage that vanish. Since each $Y_k(\mathbf i_k)$ assumes the value 1 with probability $p_k$, the probability of this $k$-th event is at most $(1-p_k)^{N_k^d}$, completing the proof.  
\end{proof}  
\noindent In view of Lemma \ref{nonemptyE-lemma}, we can define the conditional probability measure $\mathbb P^{\ast}$ as follows: for any measurable set $A$,  
\begin{equation} \label{P-star}
\mathbb P^{\ast}(A) := \frac{\mathbb P(A \cap \{E \neq \emptyset \})}{\mathbb P(E \neq \emptyset)}.   
\end{equation}  
All the probabilistic statements made in Theorem \ref{mainthm-example} and Corollary \ref{maincor-example} are with respect to $\mathbb P^{\ast}$ defined above. We pause for a moment to make the following observations:
\begin{itemize} 
\item The measure $\mathbb P^{\ast}$ is supported on the set $\Omega^{\ast} = \{\omega \in \Omega : E(\omega) \text{ is nonempty} \}$. 
\vskip0.1in
\item Since $p_k \leq 1/2$ for all $k$, any binary string $\omega$ that generates a nonempty set $E(\omega)$ receives zero measure from $\mathbb P$, and hence from $\mathbb P^{\ast}$. As a result, any countable collection of such strings $\omega$ is also $\mathbb P^{\ast}$-null. 
\vskip0.1in
\item Thus $\Omega^{\ast}$ is uncountable. \label{uncountable-sets} 
\end{itemize} 
\subsection{Deterministic estimates on the number of basic cubes} We have seen in Lemmas \ref{l:dimension} and \ref{weak-star-limit-lemma} that the quantity $P_k$ plays an important role in determining the Hausdorff dimension of $E$ and the measure $\mu$ on it.  For the random construction that we have just described in this section,  $P_k$ is a random variable. We now record some quantitative estimates for $P_k$ that will play a crucial role in the sequel. Set 
\begin{equation}  \label{PbarRk}
\overline{P}_k := N_k^{d(1-\epsilon_k)} P_{k-1}, \qquad R_k := \prod_{j=1}^{k} N_j^{d(1-\epsilon_j)}.  
\end{equation}  
Note that while $P_k$ is a random variable given by $\mathbf X_k$, the random variable $\overline{P}_k$ depends only on $\mathbf X_{k-1}$. The quantity $R_k$ on the other hand is purely deterministic. 
\begin{lemma} \label{PkRk-lemma}
Assume that the construction parameters $N_k$ and $\epsilon_k$ obey \eqref{nonemptyE-condition} and \begin{equation}  \label{another-sum} \sum_{k=1}^{\infty} (\log k) \; N_k^{-d(1 - \epsilon_k)/2} < \infty. \end{equation} 
Then for $\mathbb P^{\ast}$-almost every $\omega \in \Omega$, there exist constants $C_1, C_2 \geq 1$ depending on $\omega$ such that for every $k \geq 1$, the following two estimates hold:  
\begin{align}
&|P_k - \overline{P}_k| \leq C_1 \sqrt{\log(k+1)}   \max\bigl(\overline{P}_k, \log (k+1) \bigr)^{\frac{1}{2}}, \label{PkPk-bar}\\ 
&C_2^{-1} R_k \leq P_k \leq C_2 R_k. \label{PkRk}
\end{align}  
\end{lemma} 
\begin{proof} 
Estimates of this type are consequences of large deviation inequalities ubiquitous in the probabilistic literature, and have also appeared in previous work on random construction of sets, see for instance \cite{{LP09}, {LP11}}. For completeness, we include the proof of \eqref{PkPk-bar} and \eqref{PkRk} in Section \ref{PkRk-proof}. 
\end{proof} 
\noindent {\bf{Remark: }} \label{Lebesgue-null-remark}
In view of Lemma \ref{PkRk-lemma} and \eqref{def-pk}, we deduce that for $\mathbb P^{\ast}$-almost every $\omega \in \Omega$, 
\[ |E_k| = P_k \delta_k^d \leq CR_k \delta_k^d \leq C \prod_{j=1}^k p_j = C2^{-k} \rightarrow 0 \quad \text{ as } k \rightarrow \infty. \]
In other words, $E$ is Lebesgue-null almost surely.  In Lemma \ref{dim-count-lemma}, we will obtain more refined information on the size of $E$.    
\subsection{Estimating the number of descendants of a basic cube} \label{descendant-section} 
In addition to certain deterministic estimates on $P_k$, we will also need analogous bounds on the number of basic cubes descended from a given one. For indices $1 \leq r < \ell$, and a fixed $\mathbf i_r \in \mathbb I(r, d)$, let us define 
\begin{equation}  q_{\ell}[\mathbf i_r] := \sum' X_{\ell}(\mathbf i_{\ell}), \label{def-ql}  \end{equation} 
where $\overset{'}{\sum}$ ranges over all multi-indices $\mathbf i_{\ell} \in \mathbb I(\ell, d)$ whose projection onto the first $r$ coordinates yields $\mathbf i_r$.  Thus $q_{\ell}[\mathbf i_r]$ represents the number of basic cubes of the $k$-th generation descended from $Q(\mathbf i_r)$.  
\begin{lemma} \label{descendant-count-lemma}
Suppose that the construction parameters $N_k$ and $\epsilon_k$ obey, in addition to \eqref{nonemptyE-condition} and \eqref{another-sum}, the summability condition 
\begin{equation} 
\label{third-sum}\sum_{k > k'} k |\log \delta_{k'}|  N_k^{-d(1-\epsilon_k)/2} < \infty. 
\end{equation} 
Then for $\mathbb P^{\ast}$-almost every $\omega \in \Omega$, there exists a constant $C = C_{\omega} > 0$ such that for every choice of indices $r < \ell$, 
\begin{equation} \label{ql-est}  \sup \bigl\{ q_{\ell}[\mathbf i_r] : \mathbf i_r \in \mathbb I(r, d)   \bigr\} \leq C \left( \frac{\delta_{r}}{\delta_{\ell}}\right)^d \prod_{m=r+1}^{\ell} p_m.  \end{equation}  
\end{lemma} 
\noindent The proof of the lemma appears in Section \ref{ql-bounds-section}. 
\subsection{The random Cantor measure $\mu$}
Having constructed the Cantor-like set $E$, we now describe the measure $\mu$ on $E$ that will eventually realize the conclusions of Theorem \ref{mainthm-example} and Corollary \ref{maincor-example}. 
\begin{proposition} \label{random-measure-existence-prop} 
Suppose that the construction parameters $N_k$ and $\epsilon_k$ obey, in addition to \eqref{nonemptyE-condition}, \eqref{another-sum} and \eqref{third-sum}, the following conditions: 
\begin{equation} 
\lim_{k \rightarrow \infty} \bigl| \log \delta_k \bigr| R_k^{\frac{1}{2}} \sum_{m=1}^{\infty} (k+m) R_{k+m}^{-\frac{1}{2}} = 0, \label{msum}
\end{equation}  
where $\{R_k : k \geq 1\}$ is the sequence of deterministic constants defined in \eqref{PbarRk}. Then, for $\mathbb P^{\ast}$-almost every $\omega \in \Omega$, the convergence condition \eqref{weak-star-limit-condition} in Lemma  \ref{weak-star-limit-lemma} holds. Hence there exists a probability measure  $\mu$ supported on $E$ such that 
\[ \mu_k = \frac{1_{E_k}}{|E_k|} \longrightarrow \mu \text{ as $k \rightarrow \infty$, in the weak-$\ast$ topology.} \]  
\end{proposition} 
\noindent The proof has been given in Section \ref{limiting-measure-section}. 
\subsection{Choice of construction parameters}  \label{parameters-choice-section}
So far, we have not specified values of $N_k$ and $\epsilon_k$ that are used in the random construction of our Cantor sets. We do so now. Even though a vast majority of our results will continue to hold for very general choices of large $N_k$ and small $\epsilon_k$ obeying the summability conditions \eqref{nonemptyE-condition}, \eqref{another-sum}, \eqref{third-sum} and \eqref{msum}, we set down two specific choices of $(N_k, \epsilon_k)$-pairs that will be used as reference points for the rest of the analysis. They are 
\begin{align}
N_k &:= N^k, \quad \epsilon_k = \frac{\gamma}{k}, \text{ and } \label{dim1} \\ 
N_k &:= N^k, \quad \epsilon_k = \epsilon. \label{dim-epsilon}  
\end{align} 
In \eqref{dim1}, $N \geq 1$ is a fixed large integer and $0 < \gamma < 1$ is a small constant such that $N^{\gamma}$ is large. In \eqref{dim-epsilon}, $N$ is a fixed large integer and $0 < \epsilon < 1$ is an arbitrary constant independent of $N$. For the convenience of the reader, we have included the verification of the summability criteria in the appendix Section \ref{appendix-summability}. 
\vskip0.1in
\noindent {\bf{Remarks: }} A few comments about the choices of parameters \eqref{dim1} and \eqref{dim-epsilon} are in order. 
 \begin{enumerate}[1.]
 \item Both of these choices obey the hypotheses of Lemmas \ref{nonemptyE-lemma}, Lemma \ref{PkRk-lemma} and Proposition \ref{random-measure-existence-prop}. Thus, on one hand, Lemma \ref{nonemptyE-lemma} ensures that the measure space $(\Omega, \mathcal B, \mathbb P^{\ast})$ is well-defined for these choices. On the other, the crucial almost sure estimates on $P_k$ provided by Lemma \ref{PkRk-lemma} apply to the sets $E = E(\omega)$ constructed using these parameters. Almost every such set $E$ supports a natural measure $\mu$, as dictated by Proposition \ref{random-measure-existence-prop}. In Lemma \ref{dim-count-lemma} below, we will see that $N_k$ and $\epsilon_k$ also dictate the dimensionality of $E$. In the remainder of this paper, we will show that with $\mathbb P^{\ast}$-probability 1, the measure $\mu$ obeys (essentially optimum) ball conditions \eqref{ball-condition-upper-bound} and \eqref{ball-condition-lower-bound} given its dimension. These in turn will lead to the sharp eigenfunction restriction estimates on $E$ of the form \eqref{generic-sharp}.
 \vskip0.1in
 \item In view of \eqref{dim1} and \eqref{dim-epsilon}, it is natural to ask whether simpler choices of $N_k$ and $\epsilon_k$ might lead to the same conclusions; in particular, whether the seemingly natural choice, $N_k = N$ for all $k$, will work. The hypotheses of Lemmas \ref{nonemptyE-lemma}, \ref{PkRk-lemma} and Proposition \ref{random-measure-existence-prop} fail for this choice. At the moment, we do not know whether an almost sure statement analogous to Theorem \ref{mainthm-example} holds for a random construction involving this choice of parameters. However, as a point of interest, we refer the reader to \cite[Section 6]{LP09}, where the authors construct a Cantor-type set with $N_k = N$ for all $k$, that obeys the optimal ball condition \eqref{mu-ball-condition}, with $\alpha = \text{dim}(E)$ and $\Psi \equiv 1$. This construction, while randomized, is different from ours; for every $\alpha \in (0, d)$, it yields a {\em{single good set }} $E$ of dimension $\alpha$, not a large family as above.  \label{remark-losses} 
 \end{enumerate} 
\subsection{Almost sure Hausdorff dimension} \label{H-dim-section} 
The relevance of the choices of $\epsilon_k$ in \eqref{dim1} and \eqref{dim-epsilon} is further clarified in the following lemma. 
\begin{lemma} \label{dim-count-lemma}
For the random construction described in Section \ref{random-cantor-section}, and for $\mathbb P^{\ast}$-almost every $\omega \in \Omega$, the corresponding set $E = E(\omega)$ obeys the dimensional bound  
 \begin{equation} \label{dim-count}
 \text{dim}_{\mathbb H}(E) = \begin{cases} d &\text{for $N_k$ and $\epsilon_k$ as in \eqref{dim1}}, \\ d(1-\epsilon) &\text{for $N_k$ and $\epsilon_k$ as in \eqref{dim-epsilon}}. \end{cases} 
 \end{equation}  
\end{lemma}
\begin{proof} 
We appeal to Lemma \ref{l:dimension}, combined with the deterministic estimates on $P_k$ obtained in Lemma \ref{PkRk-lemma}. Note that $M_k = \delta_k^{-1} = N^{k(k+1)/2}$ for both choices \eqref{dim1} and \eqref{dim-epsilon}. In view of \eqref{PbarRk} and \eqref{PkRk}, the quantity $P_k$ is $\mathbb P^{\ast}$-almost surely bounded above and below by constant multiples of 
 \begin{equation} \label{def-Rk} 
 R_k = \prod_{j=1}^{k} N_j^{d(1-\epsilon_j)} = \left\{ \begin{aligned} &\prod_{j=1}^{k} N^{dj(1 - \gamma/j)} = N^{dk(k+1)/2 - d \gamma k} &\text{ for \eqref{dim1}}, \\ &\prod_{j=1}^{k} N^{dj(1 - \epsilon)} = N^{d(1 - \epsilon)k(k+1)/2} &\text{ for \eqref{dim-epsilon}}. \end{aligned} \right\}
 \end{equation} 
 With these choices of $\delta_k$ and $R_k$, it is now easy to check that \[ \liminf_{k \rightarrow \infty} \frac{\log(P_k/N_k^d)}{- \log(\delta_{k-1})} = \liminf_{k \rightarrow \infty} \frac{\log(P_k)}{-\log(\delta_k)} = \lim_{k \rightarrow \infty} \frac{\log(R_k)}{-\log(\delta_k)},  \]
 and that the value of the limit is precisely the quantity in the right hand side of \eqref{dim-count} in the two cases.  The desired conclusion now follows from Lemma \ref{l:dimension}.  
\end{proof}  
\noindent {\bf{Remarks: }} 
\begin{enumerate}[1.]
\item Lemma \ref{dim-count-lemma} says that $d\epsilon$ should be viewed as a marker of ``codimension'' of the set $E$, which therefore is independent of $\gamma$ (for \eqref{dim1}) and $N$ (for both \eqref{dim1} and \eqref{dim-epsilon}). We therefore fix $N$ and $\gamma$ as absolute constants that will not change in the sequel; for instance, $N = 10^{6d}$ and $\gamma = 1/3$ will suffice. The quantity $\epsilon$ will vary. For notational consistency, we will henceforth set $\epsilon = 0$ for the case given by \eqref{dim1}. 
\vskip0.1in
\item In order for Lemma \ref{l:dimension}, and hence Lemma \ref{dim-count}, to be applicable, the parameters $N_k$ cannot grow too fast. For instance, the upper bound in \eqref{Hdim-upper-bound} 
would not match the lower bound in \eqref{Hdim-lower-bound} for a choice of $N_k = N^{N^k}$.  
\vskip0.1in
\item Lemma \ref{dim-count-lemma} also highlights the truly probabilistic nature of the random construction. A binary string $\omega \in \Omega$ that results in $P_k = 1$ for all $k$ corresponds to a set $E$ that is a singleton, and hence of Hausdorff dimension zero. The same statement is true for strings $\omega$ that generate a countable set $E$. This shows that the exceptional set in Lemma \ref{dim-count-lemma} is nonempty, and moreover, of infinite cardinality. 
\end{enumerate}

\section{Ball conditions for random measures} \label{ball-measures}
\noindent In Sections \ref{s:randomcantor} and \ref{random-cantor-section}, we have described a measure space $(\Omega, \mathcal B, \mathbb P^{\ast})$, almost every element of which generates a Cantor set $E$ and an accompanying measure $\mu$ obeying certain properties. We are now ready to use these properties to establish the conclusions of Theorem \ref{mainthm-example}.    
\begin{proof}[{\bf{Proof of Theorem \ref{mainthm-example}}}]
Let the parameters $N_k$ and $\epsilon_k$ be chosen as in Section \ref{parameters-choice-section}. Then the dimensionality statements in part \eqref{parta} of the theorem follows from Lemma \ref{dim-count-lemma}. When $\epsilon = 0$, the remark following Lemma \ref{PkRk-lemma} on page \pageref{Lebesgue-null-remark} ensures that we still have the relation $|E| = \lim_{k \rightarrow \infty} |E_k| = \lim_{k \rightarrow \infty} P_k \delta_k^d = 0$, which proves that $\mu$ is singular with respect to Lebesgue.
\vskip0.1in   
\noindent We proceed to \eqref{mainthm-mainest}. Fix any $v_0 \in E$ and any $0<r<1$. According to Proposition \ref{random-measure-existence-prop}, $\mu$ is by definition the weak-star limit of $\mu_k$. Therefore, \[ \mu \bigl[ B(v_0; r) \bigr] = \lim_{k \rightarrow \infty} \mu_k \bigl[ B(v_0; r)\bigr]. \]
Thus, in order to establish \eqref{ball-condition-upper-bound}, it suffices to prove a similar inequality with $\mu_k$ replacing $\mu$, for all $k$ sufficiently large. It follows from \eqref{def-muk} that 
\begin{align}
\mu_k \bigl[ B(v_0; r)\bigr] &= \frac{1}{P_k \delta_k^d} \sum_{\mathbf i_k} X_k(\mathbf i_k) \int_{B(v_0;r)} 1_{Q(\mathbf i_k)}(x) \, dx \nonumber \\ 
&= \frac{1}{P_k \delta_k^d} \sum_{\mathbf i_k} X_k(\mathbf i_k) \bigl| B(v_0;r) \cap Q(\mathbf i_k) \bigr|.  \label{ball-sum}
\end{align}  
Given $r \in (0,1)$, let us choose the unique scale $\ell^{\ast}$ such that 
\begin{equation} 
\delta_{\ell^{\ast}+1} < r \leq \delta_{\ell^{\ast}}. \label{def-l-star} 
\end{equation}  
The relevance of $\ell^{\ast}$ is that $B(v_0;r)$ can be covered by cubes of the form $\{Q(\mathbf i_{\ell^{\ast}}): \mathbf i_{\ell^{\ast}} \in \mathbb I^{\ast}[v_0;r]\}$, where $\mathbb I^{\ast}[v_0;r] \subseteq \mathbb I(\ell, d)$, and the cardinality of $\mathbb I^{\ast}[v_0;r]$ is at most a constant $C_d$ depending only on $d$.  Thus the cubes $Q(\mathbf i_k)$ that contribute to the sum in \eqref{ball-sum} are those descended from $Q(\mathbf i_{\ell^{\ast}})$ for some $\mathbf i_{\ell^{\ast}} \in \mathbb I^{\ast}[v_0;r]$. In other words, the sum ranges over multi-indices $\mathbf i_k \in \mathbb I(k,d)$ whose projection onto the first $\ell^{\ast}$ coordinates yields  some $\mathbf i_{\ell^{\ast}} \in \mathbb I^{\ast}[v_0;r]$. This leads to the following estimate: 
\begin{align} 
\mu_k \bigl[ B(v_0; r)\bigr] &\leq \frac{1}{P_k \delta_k^d} \sum_{\mathbf i_{\ell^{\ast}} \in \mathbb I^{\ast}[v_0;r]} \sum' X_{k}(\mathbf i_k) \delta_k^d \nonumber \\ &= \frac{1}{P_k} \sum_{\mathbf i_{\ell^{\ast}} \in \mathbb I^{\ast}[v_0;r]} q_k[\mathbf i_{\ell^{\ast}}] \nonumber  \\ 
&\leq \frac{C_d}{P_k} \sup \Bigl\{ q_k[\mathbf i_{\ell^{\ast}}] : \mathbf i_{\ell^{\ast}} \in \mathbb I(\ell^{\ast}, d)\Bigr\}  \nonumber \\ 
&\leq \frac{C}{P_k} \left( \frac{\delta_{\ell^{\ast}}}{\delta_k}\right)^d \prod_{m=\ell^{\ast}+1}^{k} p_m \leq \frac{C}{R_k} \left( \frac{\delta_{\ell^{\ast}}}{\delta_k}\right)^d \prod_{m=\ell^{\ast}+1}^{k} p_m. \label{mu-upper} 
\end{align} 
The sum $\overset{'}{\sum}$ in the first displayed line above ranges over all multi-indices $\mathbf i_k$ which project onto some $\mathbf i_{\ell^{\ast}} \in \mathbb I^{\ast}[v_0;r]$ in the first $\ell^{\ast}$ coordinates. Thus the quantity $q_{k}[\mathbf i_{\ell^{\ast}}]$ that appears in the second line is the same as the one defined in \eqref{def-ql} in Section \ref{descendant-section}, namely the number of basic cubes of the $k$th generation descended from $Q(\mathbf i_{\ell^{\ast}})$.  Lemma \ref{descendant-count-lemma} then leads to the first upper bound in \eqref{mu-upper}. The second inequality in \eqref{mu-upper} is a consequence of \eqref{PkRk}. A simplification of this last expression using \eqref{def-pk}, \eqref{dim1} and \eqref{dim-epsilon} yields 
\[ \frac{1}{R_k} \left( \frac{\delta_{\ell^{\ast}}}{\delta_k}\right)^d \prod_{m=\ell^{\ast}+1}^{k} p_m = \begin{cases} \delta_{\ell^{\ast}}^{d(1-\epsilon)} &\text{ if } \epsilon > 0, \\ \delta_{\ell^{\ast}}^{d} N_{\ell^{\ast}}^{d \gamma} &\text{ if } \epsilon = 0. \end{cases} \] 
In view of \eqref{def-l-star}, both expressions above are dominated by $r^{d(1-\epsilon)} \Phi_{C_1}(1/r)$ for a deterministic constant $C_1$, completing the proof of \eqref{ball-condition-upper-bound}.  
\vskip0.1in
\noindent We turn to proving part \eqref{mainthm-example-partc}, namely the inequality \eqref{ball-condition-lower-bound}. For any $v_0 \in \mathbb R^d$, the ball $B(v_0; r)$ contains a cube $Q = Q(v_0; r)$ centred at $v_0$ and of sidelength $2r/\sqrt{d}$. It follows from \eqref{def-l-star} that there exists a multi-index $\widehat{\mathbf i}_{\ell^{\ast} + 2} \in \mathbb I(\ell^{\ast} + 2, d)$ such that  
\[ v_0 \in \widehat{Q} = Q(\widehat{\mathbf i}_{\ell^{\ast} + 2}) \subset Q, \] 
provided the parameter $N$ in \eqref{dim1} and \eqref{dim-epsilon} is chosen large enough relative to $d$. The relation \eqref{ball-sum} then leads to the following lower bound: for $k \geq \ell^{\ast} + 3$, 
\begin{equation}
\mu_k[B(v_0; r)] \geq \frac{1}{P_k \delta_k^d} \hat{\sum_{\mathbf i_k}} X_k(\mathbf i_k) \delta_k^d = \frac{1}{P_k} q_k[\widehat{\mathbf i}_{\ell^{\ast}+2}],  \label{ball-lower} 
\end{equation} 
where $\hat{\sum}$ indicates summation over all multi-indices $\mathbf i_k$ whose projection onto the first $\ell^{\ast} + 2$ coordinates yields $\widehat{\mathbf i}_{\ell^{\ast} + 2}$. In other words, we only restrict attention to cubes $Q(\mathbf i_k)$ that are descended from $\widehat{Q}$. If additionally we assume that $v_0 \in E$, then  the right hand side of \eqref{ball-lower} is guaranteed to be nonzero; in fact, by Lemmas \ref{PkRk-lemma} and \ref{descendant-count-lemma-2} we can estimate it from below by 
\[ \mu_k[B(v_0; r)] \geq  \frac{1}{P_k} q_k[\widehat{\mathbf i}_{\ell^{\ast}+2}] \geq \frac{1}{C R_k} \left( \frac{\delta_{\ell^{\ast} +2}}{\delta_k}\right)^d \prod_{m = \ell^{\ast} + 3}^{k} p_m.   \]
We leave the reader to verify, along the same lines of the proof of \eqref{mainthm-mainest}, that this last quantity is bounded from below by $r^{d(1-\epsilon)}/\Phi_{C_1}(1/r)$, completing the proof.  
\end{proof} 
\begin{proof} [{\bf{Proof of Corollary \ref{maincor-example}}}]
Part \eqref{mainthm-mainest} of Theorem \ref{mainthm-example} ensures that $\mu = \mu(\omega)$ obeys \eqref{mu-ball-condition} with $\alpha = d(1 - \epsilon)$ and $\Psi \equiv \Phi_{C_1}$. It follows from Theorem \ref{mainthm-restriction-ball} that \eqref{eigenfunction-growth-nu} holds with this choice of $\alpha$ and $\Psi$. Recalling the definitions of $p_0$, $p_{\ast}$, $\theta_p$ and $\vartheta_p$ from \eqref{def-p0}, \eqref{our-exponent} and \eqref{our-exponent-2}, we observe that $p_0 = p_{\ast}$, $\vartheta_p = \theta_p$ and $\overline{\Psi}_p \leq \Phi_{R/p}$ for some large deterministic constant $R$.    
\vskip0.1in
\noindent Part \eqref{mainthm-example-partc} verifies the condition \eqref{AD-local} for the measure $\mu = \mu(\omega)$, with $\alpha = d(1 - \epsilon)$ and $\Xi = \Phi_{C_1}$. Thus the conclusion of Theorem \ref{maincor-sharpness} holds for this measure. We observe that $\Lambda_p \leq \Phi_R$ for large $R$,  and $\varkappa_p = \vartheta_p = \theta_p$ for $p \geq \max(2, p_{\ast}) = \max(2, p_0)$. 
\end{proof}

\section{Appendix: A generalized Young-type inequality} \label{young-section}
\noindent The following proposition is the generalized Young's inequality that is used in the proof of Proposition \ref{gen-Young-prop}. Special cases of this inequality appear in the functional-analytic literature as Schur's lemma; see for instance Theorem 6.18 in \cite{Fol}. However, we were unable to find a textbook version of the result that states and proves it in the level of generality that we need, and have included a proof here for completeness. 
\begin{proposition} \label{l:youngsineq}
Let $\tau$ be a positive Borel measure supported on a set $E \subseteq \mathbb R^d$ that is not necessarily translation-invariant.  Given a measurable function $K(\cdot, \cdot)$, consider the integral operator:
\begin{equation}
Tf(x) := \int K(x, y) f(y) \, d\tau(y).
\end{equation}
Then for any choice of exponents $1 \leq s, q, r \leq \infty$ satisfying \begin{align} 1 + \frac{1}{r} &= \frac{1}{s} + \frac{1}{q}, \label{pqr} \\ A_s := \sup_{x \in E} \left[ \int |K(x, y)|^s \, d\tau(y) \right]^{\frac{1}{s}} < \infty, \quad & B_s := \sup_{y \in E} \left[ \int |K(x, y)|^s \, d\tau(x) \right]^{\frac{1}{s}} < \infty,  \label{AsBs}\end{align} 
the following inequality holds: 
\[ ||Tf||_{L^r(\tau)} \leq A_s^{1 - \frac{s}{r}} B_s^{\frac{s}{r}} ||f||_{L^q(\tau)}.  \]
\end{proposition}
\begin{proof}
We adapt the same method of proof as the classical Young's inequality, where $\tau$ is the Lebesgue measure, with the modifications needed to deal with the lack of translation-invariance. We begin with a pointwise bound for the operator $T$. 
\begin{align*}
\left| Tf(x) \right| &\leq \int |K(x, y)| |f(y)| \, d\tau(y) \\ &\leq \int |K(x, y)|^{1 + \frac{s}{r} - \frac{s}{r}} |f(y)|^{1 + \frac{q}{r} - \frac{q}{r}} \, d\tau(y) \\
& = \int  \left( | K(x, y)|^{\frac{s}{r}} |f(y)|^{\frac{q}{r}} \right) |K(x, y)|^{\frac{r-s}{r}} |f(y)|^{\frac{r-q}{r}} \, d\tau(y) \\
& \leq  \mathfrak T_1(x) \times \mathfrak T_2(x) \times \mathfrak T_3, \text{ where } \\
\mathfrak T_1(x) &= || \left( | K(x, \cdot )|^{s} |f(\cdot)|^{q} \right)^{\frac{1}{r}} ||_{L^r(\tau)} = \Bigl(\int |K(x, y)|^s |f(y)|^q \, d\tau(y) \Bigr)^{\frac{1}{r}}, \\
\mathfrak T_2(x) &= \| |K(x, \cdot)|^{\frac{r-s}{r}} \|_{L^{\frac{sr}{r-s}}(\tau)} = \Bigl(\int |K(x, y)|^s d \tau(y) \Bigr)^{\frac{r-s}{sr}} \leq A_s^{\frac{r-s}{r}} \\
\mathfrak T_3 &=  \| |f|^{\frac{r-q}{r}} \|_{L^{\frac{qr}{r-q}}(\tau)} = ||f||_{L^q(\tau)}^{\frac{r-q}{r}}.
\end{align*}
The last line in the estimation of $Tf(x)$ above uses the generalized H\"older's inequality with the triple of exponents $(r, sr/(r-s), {qr}/(r-q))$, whose reciprocals add up to one, by \eqref{pqr}. We ask the reader to verify that the relation \eqref{pqr} ensures that each of the exponents $sr/(r-s)$ and $qr/(r-q)$ lie in $[1, \infty]$. Using the pointwise bound on $Tf$ obtained above, we proceed to compute its $L^r(\tau)$-norm:
\begin{align*}
\| Tf\|_{L^r(\tau)}^r &\leq \int \left[\mathfrak T_1(x) \times \mathfrak T_2(x) \times \mathfrak T_3 \right]^r \, d\tau(x) \\ &\leq \int  \Bigl(\int |K(x, y)|^{s} |f(y)|^{q} \, d\tau(y) \Bigr) A_s^{r-s} \times ||f||_{L^q(\tau)}^{r-q} \, d\tau(x) \\ 
& \leq A_s^{r-s}  ||f||_{L^q(\tau)}^{r-q}  \iint |K(x, y)|^{s} |f(y)|^{q} \, d\tau(y) \, d\tau(x) \\ 
& \leq A_s^{r-s}  ||f||_{L^q(\tau)}^{r-q}  \times (B_s^s ||f||_{L^q(\tau)}^q) \\ 
& \leq A_s^{r-s} B_s^{s} ||f||_{L^q(\tau)}^r, 
\end{align*}
which completes our proof.
\end{proof}

\section{Appendix: Probabilistic tools} \label{Appendix-section}
\subsection{Large deviation inequalities} 
The following well-known concentration inequality plays a pivotal role in the proofs that occur in Sections \ref{random-cantor-section} and \ref{ball-measures}, in particular Lemmas \ref{PkRk-lemma}, \ref{descendant-count-lemma} and Proposition \ref{random-measure-existence-prop}. 
\begin{theorem} (Bernstein's inequality \cite[p.205]{Usp}, \cite[Section 2.8]{Ver}) \label{t:bernstein}
Let $Z_1, \dots, Z_m$ be independent random variables with \[|Z_r| \leq M, \quad \mathbb{E}Z_r= 0 \quad \text{ and } \quad \mathbb{E}|Z_r|^2 = {\sigma}_r^2. \]  Let $\sum {\sigma}_r^2 \leq \sigma^2$. Then for all $t > 0$,  
\begin{equation} \label{bernstein-inequality}
\mathbb{P} \Bigl( \bigl| \sum_1^m Z_r \bigr| \geq t \Bigr) \leq \exp\left(-\frac{\frac{1}{2}t^2}{\sigma^2 + \frac{1}{3}Mt}\right).
\end{equation}
\end{theorem}


\subsection{Bounds on $P_k$} \label{PkRk-proof} 
\begin{proof}[{\bf{Proof of Lemma \ref{PkRk-lemma}}}]
The inequality \eqref{PkRk} is a consequence of \eqref{PkPk-bar}. We start by proving this implication. Define the auxiliary quantity 
\[ \eta_k := \frac{P_k - \overline{P}_k}{\overline{P}_k}, \quad \text{ so that } \quad P_k = \overline{P}_k (1 + \eta_k). \]
An iteration involving \eqref{PbarRk} then gives 
\begin{align*}
P_k &= \overline{P}_k (1 + \eta_k) = N_k^{d(1-\epsilon_k)} (1 + \eta_k) P_{k-1}= \cdots \\ 
&= \prod_{j=1}^{k} \bigl[ N_j^{d(1 - \epsilon_j)} (1 + \eta_j) \bigr] = R_k \prod_{j=1}^{k} (1 + \eta_j). 
\end{align*} 
Thus, in order to prove \eqref{PkRk} it suffices to establish the $\mathbb P^{\ast}$-almost sure existence of a constant $C_2 = C_2(\omega) > 0$ such that 
\begin{equation} \label{eta-product}
C_2^{-1} \leq \prod_{j=1}^{k} (1 + \eta_j) \leq C_2 \quad \text{ for all } k \geq 1. 
\end{equation} 
This follows from two observations: the first is that for $\mathbb P^{\ast}$-almost every random set $E$, 
\begin{equation} \label{nonvanishing-eta}\eta_k \ne -1, \text{ since } P_k \ne 0 \text{ for all } k \geq 1, \end{equation} so the product in \eqref{eta-product} is strictly positive. The second point to note is that the sequence $\eta_k$ is absolutely summable. To see this, we estimate the sum as follows,  
\begin{align*} \sum_{k=1}^{\infty} |\eta_k| = \frac{|P_k - \overline{P}_k|}{\overline{P}_k} &\leq C_1 \sum_k\sqrt{\log (k+1)} \; \max\bigl(\overline{P}_k, \log (k+1) \bigr)^{\frac{1}{2}} \times (\overline{P}_k)^{-1}   \\ &\leq C_1 \Biggl[\sum_{k: \overline{P}_k \leq \log(k+1)} \frac{\log(k+1)}{\overline{P}_k} + \sum_{k: \overline{P}_k > \log(k+1)} \frac{\sqrt{\log(k+1)}}{\overline{P}_k^{1/2}} \Biggr]\\
&\leq C_1 \sum_k \frac{\log(k+1)}{\overline{P}_k^{1/2}} \\
&\leq C_1 \sum_k \log (k+1) \; \bigl[ N_k^{d(1- \epsilon_k)} P_{k-1}\bigr]^{-\frac{1}{2}} \\ &\leq C_1 \sum_k \log (k+1) \; N_k^{-d(1-\epsilon_k)/2}.  
\end{align*}    
The first inequality above follows from \eqref{PkPk-bar}. In the second step, we have rewritten the sum in two parts, depending on the relative sizes of $\overline{P}_k$ and $\log(k+1)$.  The penultimate inequality makes use of the defining identity of $\overline{P}_k$ in \eqref{PbarRk}. The last inequality is a consequence of the fact that $P_{k-1} \geq 1$ on the support of $\mathbb P^{\ast}$.  Our summability hypothesis \eqref{another-sum} therefore implies that $\mathbb P^{\ast}$-almost surely there exists a large integer $k_0 = k_0(\omega) > 0$ depending on $C_1$ such that 
\[ \frac{1}{2} \leq \prod_{j=k'}^{k} (1+ \eta_j) \leq 2 \quad \text{ for all } k \geq k' \geq k_0. \] Set $C_2 \geq 1$ to be any constant such that $C_2= C_2(\omega) \geq \max \left(2C_3, 1/(2C_4) \right)$, where \[ C_3 =\sup_{1 \leq k \leq k_0}  \prod_{j=1}^{k}(1 + \eta_j)  \quad \text{ and } \quad C_4 =\sup_{1 \leq k \leq k_0}  \prod_{j=1}^{k}(1 + \eta_j)^{-1}. \]    Both $C_3$ and $C_4$ are strictly positive, by \eqref{nonvanishing-eta}. Then \eqref{eta-product} holds with this $C_2$, proving \eqref{PkRk}.  
\vskip0.1in 
\noindent It remains to prove \eqref{PkPk-bar}. For a fixed large $\omega$-independent constant $B$ soon to be specified ($B = 5$ will suffice), we define the event 
\[ \mathtt S_k := \left\{ \omega \in \Omega: |P_k - \overline{P}_k| \leq t_k \right\},   \]
where  
\begin{align} t_k &:= B \sqrt{\log(k+1)}\max \Bigl(\overline{P}_k, \log(k+1) \Bigr)^{\frac{1}{2}} \label{def-t0}\\  
&= \begin{cases} 
B \log (k+1) &\text{ for } \overline{P}_k \leq \log (k+1), \\  B 
\sqrt{\log (k+1)} \; \overline{P}_k^{\frac{1}{2}} &\text{ for } \overline{P}_k > \log (k+1).  \nonumber
\end{cases} \end{align}
We aim to show that 
\begin{equation}  \label{BC1} \sum_{k=1}^{\infty} \mathbb P^{\ast}(\mathtt S_k^c) < \infty. \end{equation}  
Once \eqref{BC1} is established, the Borel-Cantelli Lemma \cite[p 53, Theorem 4.3]{Billingsley} implies that for $\mathbb P^{\ast}$-almost every $\omega$, there exists an integer $k_0(\omega) \geq 1$ such that the event $\mathtt S_k$ occurs for all $k \geq k_0(\omega)$. Since  we have 
\[ \sup_{k \leq k_0(\omega)} \frac{|P_k - \overline{P}_k|}{\sqrt{\overline{P}_k \log(k+1)}} \leq \frac{1}{\sqrt{\log 2}}\sup_{k \leq k_0(\omega)} |P_k - \overline{P}_k| \leq \frac{2\delta_{k_0}^{-d}}{\sqrt{\log 2}} \]
for every such $\omega$, the desired inequality \eqref{PkPk-bar} holds $\mathbb P^{\ast}$-almost surely by setting \[ C_1 = \max(B, 2\delta_{k_0}^{-d}/\sqrt{\log 2}). \]  
\vskip0.1in
\noindent We now turn our attention to proving \eqref{BC1}. The description of the measure space $(\Omega, \mathcal B, \mathbb P^{\ast})$ from Section \ref{measure-space-section} will be helpful in this regard. Using the notation from that section, let $\mathcal F_{k}$ denote the product $\sigma$-algebra of $\mathcal B_1, \cdots, \mathcal B_k$, which is a sub $\sigma$-algebra of $\mathcal B$. Then, using the theory of conditional probability \cite[Chapter 6, Section 34]{Billingsley}, $\mathbb P^{\ast}(\mathtt S_k^c)$ can be written as  
\[ \mathbb P^{\ast}(\mathtt S_k^c) = \frac{\mathbb P(\mathtt T_k)}{\mathbb P(E \ne \emptyset)} =  \frac{\mathbb E(\mathbb P(\mathtt T_k |\mathcal F_{k-1}))}{\mathbb P(E \ne \emptyset)} \quad \text{ with }  \mathtt T_k = \mathtt S_k^c \cap \{E \neq \emptyset \}, \]
where $\mathbb P(\mathtt T_k |\mathcal F_{k-1})$ denotes the conditional probability of $\mathtt T_k$, conditioned relative to $\mathcal F_{k-1}$. The notation $\mathbb E$ denotes expected value with respect to $\mathbb P$. In order to prove \eqref{BC1}, it therefore suffices to show that $\mathbb P(\mathtt T_k | \mathcal F_{k-1})$ is bounded above by a deterministic constant that is summable in $k$. We estimate $\mathbb P(\mathtt T_k | \mathcal F_{k-1})$ using Bernstein's inequality, quoted in Theorem \ref{t:bernstein} below. Conditioning on $\mathcal F_{k-1}$, i.e., holding $\{ X_{k-1}(\mathbf i_{k-1}) : \mathbf i_{k-1} \in \mathbb I(k-1, d) \}$ fixed, we observe that 
\[ P_k - \overline{P}_k = \sum_{\mathbf i_{k-1}} X_{k-1}(\mathbf i_{k-1}) \sum_{i_k = 1}^{N_k^d} \left( Y_k(\mathbf i_k) - N_k^{-d \epsilon_k} \right) \]  
is the sum of $\overline{P}_{k} = P_{k-1} N_k^d$ independent, centred random variables $(Y_k(\mathbf i_k) - N_k^{-d \epsilon_k})$, each of which is bounded above by 1 in absolute value and has variance $\leq p_k = N_k^{-d \epsilon_k}$. 
Thus in the notation of Theorem \ref{t:bernstein}, $m = \overline{P}_k$, $M=1$ and $\sum \sigma_r^2 \leq \sigma^2 = \overline{P}_k$. We apply \eqref{bernstein-inequality} with these values and with $t= t_k$ as in \eqref{def-t0}. This yields  
\begin{align*} \mathbb P(\mathtt T_k | \mathcal F_{k-1}) &= \mathbb P(|P_k - \overline P_k| > t_k | \mathcal F_{k-1}) \leq \exp \left(-\frac{t_k^2}{{\overline{P}}_k + \frac{t_k}{3}}\right) \\ 
&\leq 
\begin{cases} 
\exp \left(-\frac{t_k^2}{4t_k/3} \right) & \text{ for } \overline{P}_k \leq \log (k+1), \\ \exp \left(-\frac{t_k^2}{4 B\overline{P}_k/3} \right) &\text{ for } \overline{P}_k > \log (k+1).  
\end{cases} \\
&\leq (k+1)^{-\frac{3B}{4}}.
\end{align*}       
The final step is obtained by substituting $t_0$ from \eqref{def-t0} into the two cases. The right hand side is summable for any choice of $B > 4$, completing the proof. 
\end{proof} 

\subsection{Bounds on $q_{\ell}[\mathbf i_r]$} \label{ql-bounds-section} 
\begin{proof}[{\bf{Proof of Lemma \ref{descendant-count-lemma}}}]
The proof is very similar to that of Lemma \ref{PkRk-lemma} above, so we only sketch the details. For fixed $r$, we define  a partially averaged version of $q_{\ell}$, which we call $\overline{q}_{\ell}$: 
\begin{equation} \label{def-qbar} \overline{q}_{\ell}[\mathbf i_r] := N_{\ell}^{d} p_{\ell} q_{\ell-1}[\mathbf i_r] = N_{\ell}^{d(1 - \epsilon_{\ell})} q_{\ell-1}[\mathbf i_r], \end{equation}    
with $p_{\ell}$ and $\epsilon_{\ell}$ as in \eqref{def-pk}. 
\vskip0.1in 
\noindent For a fixed large $\omega$-independent constant $B$ to be specified, we set 
\begin{align} R_0(\ell, r) = R_0 &:= (\ell + B)^{1/2} |\log \delta_r|^{{1}/{2}} \max \bigl( \overline{q}_{\ell}[\mathbf i_r], (\ell + B) |\log \delta_r| \bigr)^{{1}/{2}} \nonumber \\
&\,= \left\{ 
\begin{aligned} 
& (\ell + B)^{\frac{1}{2}} |\log \delta_r|^{{1}/{2}} (\overline{q}_{\ell}[\mathbf i_r])^{{1}/{2}} &\text{ if } &\overline{q}_{\ell}[\mathbf i_r] \geq (\ell + B) |\log \delta_r| \\ 
& (\ell + B) |\log \delta_r| &\text{ if } &\overline{q}_{\ell}[\mathbf i_r]\leq (\ell + B) |\log \delta_r|, 
\end{aligned} \right\} 
\label{def-tau}
\end{align} 
and define the event 
\[ \mathtt T_{r} := \bigcap_{\ell=r+1}^{\infty} \mathtt T_{r \ell}, \quad \text{ where } \quad \mathtt T_{r \ell} := \bigl\{\omega \in \Omega: |q_{\ell}[\mathbf i_r] - \overline{q}_{\ell}[\mathbf i_r]| \leq R_0 \text{ for all } \mathbf i_r \in \mathbb I(r, d) \bigr \}.  \] 
We aim to show that \begin{equation} \sum_{r=1}^{\infty} \mathbb P^{\ast}(\mathtt T_r^c) = \sum_{r=1}^{\infty} \mathbb P^{\ast} \Bigl(\bigcup_{\ell = r+1}^{\infty} \mathtt T_{r \ell}^c \Bigr)\leq \sum_{r=1}^{\infty} \sum_{\ell = r+1}^{\infty}  \mathbb P^{\ast}(\mathtt T_{r \ell}^c) < \infty. \label{qt0}  \end{equation} 
The same Borel-Cantelli argument as in Lemma \ref{PkRk-lemma} would then imply that for almost every $\omega \in \Omega$, there is a constant $C > 0$ such that  
\begin{equation}  |q_{\ell}[\mathbf i_r] - \overline{q}_{\ell}[\mathbf i_r]| \leq C R_0 \text{ for all } r < \ell \text{ and all } \mathbf i_r. \label{q-difference}  \end{equation}  We will return to the proof of \eqref{qt0} shortly, but will leave the verification of \eqref{q-difference} from \eqref{qt0} to the reader. Assuming \eqref{q-difference} for the moment, the remainder of the proof is completed as follows. We fix indices $r < \ell$, and a multi-index $\mathbf i_r$, and for simplicity write $q_{\ell} = q_{\ell}[\mathbf i_r]$, $\overline{q}_{\ell} = \overline{q}_{\ell}[\mathbf i_r]$. Define an auxiliary quantity $\zeta_{\ell}$ according to the following relation \begin{equation} q_{\ell}  = \overline{q}_{\ell} (1 + \zeta_{\ell}).  \label{q-qbar}  \end{equation}
The definitions \eqref{def-ql} and \eqref{def-qbar}  of $q_{\ell}$ and $\overline{q}_{\ell}$ imply if $q_{\ell}$ is nonzero, then so is every $q_m$ and $\overline{q}_{m}$ for $r+1 \leq m \leq \ell$. As a result, $\zeta_{\ell}$ is well-defined for nonzero $q_{\ell}$. Further, \eqref{q-difference} implies that in this case,    
\begin{align}  
|\zeta_{\ell}| \leq \frac{C R_0}{{\overline{q}_{\ell}}} &\leq C (\ell + B) |\log \delta_r|  \times  (\overline{q}_{\ell})^{-1/2}  \nonumber \\
&\leq C (\ell + B) |\log \delta_r| N_{\ell}^{-d(1 - \epsilon_{\ell})/2} \label{zetal-estimate}
\end{align} 
The second inequality in the sequence above follows from \eqref{def-tau}. The last inequality is a consequence of \eqref{def-qbar}, which says that if $\overline{q}_{\ell}$ is nonzero, it must be larger than $N_{\ell}^d p_{\ell}$. 
\vskip0.1in
\noindent The quantity $\zeta_{\ell}$ is analogous to $\eta_k$ in the proof of Lemma \ref{PkRk-lemma} and will play a similar role.  Iterating the relation \eqref{q-qbar} and applying \eqref{def-qbar} at every step, we arrive at 
\begin{align} 
q_{\ell} &= N_{\ell}^d p_{\ell} q_{\ell-1} (1 + \zeta_{\ell}) = \cdots  \nonumber \\
&= \Bigl[ \prod_{m=r+1}^{\ell} N_m^d p_m \Bigr] \Bigl[ \prod_{m=r+1}^{\ell} (1 + \zeta_m) \Bigr] q_{r} =  \Bigl( \frac{\delta_r}{\delta_{\ell}}\Bigr)^d  \Bigl[\prod_{m=r+1}^{\ell} p_m  \Bigr]\Bigl[ \prod_{m=r+1}^{\ell} (1 + \zeta_m) \Bigr].  \label{q-and-zeta} 
\end{align} 
Since $q_r = q_r[\mathbf i_r] = 1$, it suffices to show that the second product above is bounded above by a constant independent of $r$ and $\ell$ and depending only on $\omega$. The estimate \eqref{zetal-estimate} shows that $\zeta_m$ is small for $m \geq r$ and large $r$, so in order to establish the desired conclusion it suffices to show that  the sum of $|\zeta_m|$ in the range $r+1 \leq m \leq \ell$ is bounded above by a large $\omega$-dependent constant that is uniform in $r$ and $\ell$. The summability hypothesis \eqref{third-sum} ensures that this is the case.  
\vskip0.1in
\noindent It remains to prove \eqref{qt0}. We do this again with an application of Bernstein's inequality, as we did in the proof of Lemma \ref{PkRk-lemma}. Let us recall for a moment the definition of $q_{\ell}$ and the sum $\overset{'}{\sum}$ from \eqref{def-ql}. Using this notation and conditioning on $\mathcal F_{\ell-1}$, we can write $q_{\ell} - \overline{q}_{\ell}$ as a sum of independent centred random variables:  \[ q_{\ell} - \overline{q}_{\ell} = \sum_{\mathbf i_{\ell -1}}' X_{\ell-1}(\mathbf i_{\ell-1}) \sum_{\overline{i}_{\ell}} \bigl(Y_{\ell}(\mathbf i_{\ell}) - p_{\ell}). \] Setting \[ m = q_{\ell-1} N_\ell^d, \qquad \sigma^2 = N_{\ell}^d q_{\ell-1}p_{\ell} = \overline{q}_{\ell} \qquad \text{ and } \quad t = R_0 \]
in Theorem \ref{t:bernstein}, we deduce that  for each $\mathbf i_r \in \mathbb I(r,d)$, 
\[ \mathbb P^{\ast} \bigl(|q_{\ell} - \overline{q}_{\ell}| > R_0 | \mathcal F_{\ell-1}\bigr) \leq \exp \left( - \frac{R_0^2}{\overline{q}_{\ell} + \frac{R_0}{3}}\right)  \leq \delta_r^{\frac{3}{4}(\ell + B)}. \]  
Since the number of possible choices of multi-indices $\mathbf i_r$ is $\delta_r^{-d}$, summing the above estimate over all $\mathbf i_r$ yields 
\[ \mathbb P^{\ast}(\mathtt T_{r \ell}^c) \leq \delta_{r}^{\frac{3}{4}(\ell + B) -d}. \]
Choosing $B > 4d/3$ ensures that the last quantity is summable in $r$ and $\ell$ for all $\ell \geq r+1$. This completes the proof of \eqref{qt0} and hence the proof of the lemma.  
\end{proof} 
\vskip0.1in
\noindent A careful analysis of the proof of Lemma \ref{descendant-count-lemma} yields the following conclusions as well. 
\begin{lemma} \label{descendant-count-lemma-2} 
Assume the summability condition \eqref{third-sum}. Then for $\mathbb P^{\ast}$-a.e. $\omega \in \Omega$ and every $x \in E = E(\omega)$, the following property holds: 
\vskip0.1in
\noindent There exists a constant $C = C_{x, \omega} > 0$ such that for all indices $r < \ell$ and all multi-indices $\mathbf i_r$ and $\mathbf i_{\ell}$ such that 
\begin{equation}  x \in Q(\mathbf i_{\ell}) \subseteq Q(\mathbf i_r),  \label{where is x}\end{equation} 
we have 
\begin{equation} 
q_{\ell}[\mathbf i_r] \geq C^{-1} \left( \frac{\delta_{r}}{\delta_{\ell}}\right)^d \prod_{m=r+1}^{\ell} p_m. \label{ql-lower-bound} 
\end{equation}     
\end{lemma} 
\begin{proof} 
We proceed exactly as in Lemma \ref{descendant-count-lemma}, leading up to the relation \eqref{q-and-zeta}. The hypothesis  \eqref{where is x} implies that $q_{\ell}$ is nonzero for all the relevant choices of $\mathbf i_r$ and $\mathbf i_{\ell}$; in particular, none of the factors $1 + \zeta_m$ can be zero, for $r+1 \leq m \leq \ell$.  On the other hand, the estimate \eqref{zetal-estimate} and the summability hypothesis \eqref{third-sum}  imply that the tail product of $\prod_m (1 + \zeta_m)$  converges to a nonzero quantity, i.e., there are large absolute constants $R, C > 0$ such that 
\[ \inf \Bigl\{ \prod_{m=r+1}^{\ell} (1 + \zeta_m) : \ell > r \geq R \Bigr\} \geq C^{-1}.  \]
 This leaves at most $R$ factors unaccounted for, but since each factor is nonzero, we can reach \eqref{ql-lower-bound} simply by enlarging $C$ by a constant factor depending only on $R, x$ and $\omega$.  
\end{proof} 
\begin{corollary} \label{q-difference-corollary} 
Assume the summability condition \eqref{third-sum}. Then for $\mathbb P^{\ast}$-a.e. $\omega \in \Omega$ and every $x \in E = E(\omega)$, the following property holds: 
\vskip0.1in
\noindent There exists a constant $C = C(\omega) > 0$ such that for every $1 \leq r < \ell$ and every multi-index $\mathbf i_r$, 
\begin{equation} \label{qlql-bar} 
\Bigl| q_{\ell}[\mathbf i_r] - \overline{q}_{\ell}[\mathbf i_r]\Bigr| \leq C \ell |\log \delta_r| \Bigl( 1 + \overline{q}_{\ell}[\mathbf i_r]\Bigr)^{\frac{1}{2}}.  
\end{equation} 
\end{corollary} 
\begin{proof} 
This follows from \eqref{q-difference} and the fact that $R_0$ is bounded above by the right hand side of \eqref{qlql-bar}.  
\end{proof} 
\subsection{Existence of the limiting measure $\mu$} \label{limiting-measure-section} 
\begin{proof}[{\bf{Proof of Proposition \ref{random-measure-existence-prop}}}]  
Conditions \eqref{nonemptyE-condition}, \eqref{another-sum} and \eqref{third-sum} ensure that the conclusions of Lemmas \ref{nonemptyE-lemma}, \ref{PkRk-lemma} and \ref{descendant-count-lemma} hold. In other words, the measure space $(\Omega, \mathcal B, \mathbb P^{\ast})$ is well-defined, and the estimates \eqref{PkPk-bar}, \eqref{PkRk} and \eqref{ql-est} hold on this measure space almost surely.  We will use these estimates to show that \eqref{weak-star-limit-condition} holds, under the additional assumption \eqref{msum}.  
\vskip0.1in
\noindent Let us define the signed measure $\varrho_k := \mu_{k+1} - \mu_k$, so that 
\[ \mu_{k'} - \mu_k = \sum_{m=0}^{k'-k-1} \varrho_{k+m}, \quad \text{ for all } k' \geq k. \] 
The left hand side of \eqref{weak-star-limit-condition} can therefore be estimated as follows; denoting by $\overset{'}{\sum}$ the summation over all indices $\mathbf i_k \in \mathbb I(k, d)$ such that $X_k(\mathbf i_k) = 1$, we obtain  
\begin{align}
\sup_{k' \geq k} \sum'_{\mathbf i_k} \Bigl| \int_{Q(\mathbf i_k)} \bigl[ \mu_{k'} - \mu_k \bigr](x) \, dx \Bigr| &\leq  \sup_{k' \geq k} \sum'_{ \mathbf i_k} \sum_{m=0}^{k'-k-1}  \Bigl| \int_{Q(\mathbf i_k)} \varrho_{k+m}(x) \, dx \Bigr| \nonumber  \\
&\leq P_k \sup_{\mathbf i_k : X_k(\mathbf i_k) = 1}\sum_{m=0}^{\infty} \Bigl| \int_{Q(\mathbf i_k)} \varrho_{k+m}(x) \, dx \Bigr|. \label{goes-to-zero} 
\end{align}
In order to verify \eqref{weak-star-limit-condition}, it suffices to show that the expression in \eqref{goes-to-zero} approaches zero as $k \rightarrow \infty$, $\mathbb P^{\ast}$-almost surely. 
\vskip0.1in
\noindent To this end, we fix an index $m \geq 0$ and a multi-index $\mathbf i_k \in \mathbb I(k, d)$ with $X_k(\mathbf i_k) = 1$. Recalling the description of $\mu_{k+m}$ from \eqref{def-muk} and substituting it into $\varrho_{k+m}$, we find that  
\begin{align} 
P_k \int_{Q(\mathbf i_k)} \varrho_{k+m}(x) \, dx &= \frac{P_k}{P_{k+m+1}} \sum_{\bar{\mathbf j}} X_{k+m+1} (\mathbf i_k, \bar{\mathbf j}) -  \frac{P_k}{P_{k+m}} \sum_{\mathbf j} X_{k+m}(\mathbf i_k, \mathbf j) \nonumber  \\
&= \mathfrak S_1(m, \mathbf i_k)  + \mathfrak S_2(m, \mathbf i_k), \nonumber \text{ where } \\
\mathfrak S_1(m, \mathbf i_k)  &:=  P_k \Bigl[ \frac{1}{P_{k+m+1}} - \frac{1}{\overline{P}_{k+m+1}}\Bigr] \sum_{\bar{\mathbf j}} X_{k+m+1} (\mathbf i_k, \bar{\mathbf j}), \label{Sfrak-1} \text{ and } \\ 
\mathfrak S_2(m, \mathbf i_k) &:= \frac{P_k}{\overline{P}_{k+m+1}} \sum_{\mathbf j} X_{k+m}(\mathbf i, \mathbf j) \sum_{j_{m+1}} \bigl[ Y_{k+m+1}(\mathbf i, \mathbf j, j_{m+1}) - p_{k+m+1} \bigr]. \label{Sfrak-2}
\end{align} 
In the above sums, $\mathbf j$ ranges over all multi-indices with $m$ entries such that $(\mathbf i_k, \mathbf j) \in \mathbb I(k+m, d)$. The index $\bar{\mathbf j} = (\mathbf j, j_{m+1})$ has a similar description, with $m$ replaced by $(m+1)$. The required convergence to zero of \eqref{goes-to-zero} will follow if we show that 
\begin{equation} \label{S1-S2-goes-to-zero}
\sup_{\mathbf i_k : X_k(\mathbf i_k) = 1} \sum_{m=0}^{\infty} \bigl|\mathfrak S_r(m, \mathbf i_k) \bigr| \longrightarrow 0 \text{ as } k \rightarrow \infty, \quad \text{ for } r=1,2.  
\end{equation} 
\vskip0.1in
\noindent Let us first consider the expression in \eqref{Sfrak-1}. The sum of $X_{k+m+1}(\mathbf i_k, \overline{\mathbf j})$ over $\bar{\mathbf j}$ represents the number of descendants of the basic cube $Q(\mathbf i_k)$ at step $(k+m+1)$. Hence, using the definition \eqref{def-ql}, we first simplify $|\mathfrak S_1|$ as 
\begin{equation} \label{S1-simplified}  \bigl| \mathfrak S_1(m, \mathbf i_k) \bigr|  =  P_k \Biggl[ \frac{\bigl| P_{k+m+1} - \overline{P}_{k+m+1}\bigr|}{P_{k+m+1} \overline{P}_{k+m+1}} \Biggr] q_{k + m+1} [\mathbf i_k]. \end{equation} 
We can then use Lemma \ref{PkRk-lemma} and Lemma \ref{descendant-count-lemma} to estimate $\bigl| P_{k+m+1} - \overline{P}_{k+m+1}\bigr|$ and $q_{k+m+1}[\mathbf i_k]$ respectively. Substituting the estimates \eqref{PkPk-bar}, \eqref{PkRk}  and \eqref{ql-est} from these lemmas into \eqref{S1-simplified} above, and simplifying the resulting expression, we find that for some finite positive constant $C = C(\omega)$,
\begin{align*}
\sup_{\mathbf i_k : X_k(\mathbf i_k) = 1} \sum_{m=0}^{\infty} \bigl|\mathfrak S_r(m, \mathbf i_k) \bigr| &\leq C \sum_{m=0}^{\infty} R_k \Biggl[ \frac{\sqrt{\log(k+m+1)} \;  R_{k+m+1}^{\frac{1}{2}}}{R_{k+m+1}^2} \Biggr] \Biggl[ \Bigl(\frac{\delta_k}{\delta_{k+m+1}} \Bigr)^d \prod_{j=k+1}^{k+m+1} p_j  \Biggr]  \\
&\leq C \sum_{m=0}^{\infty} \sqrt{\log(k+m+1)} \; R_{k+m+1}^{-\frac{1}{2}}. 
\end{align*}
Our hypothesis \eqref{msum} implies that this last quantity is bounded above by $R_k^{-1/2}$, which tends to zero as $k \rightarrow \infty$. This proves \eqref{S1-S2-goes-to-zero} for $r = 1$.  
\vskip0.1in
\noindent We now turn our attention to proving \eqref{S1-S2-goes-to-zero} for $r = 2$. A simplification of \eqref{Sfrak-2} using \eqref{def-ql} and \eqref{def-qbar} yields
\begin{align*} \bigl|  \mathfrak S_2(m, \mathbf i_k) \bigr| &= \frac{P_k}{\overline{P}_{k+m+1}} \times \bigl| q_{k+m+1}[\mathbf i_k] - \overline{q}_{k+m+1}[\mathbf i_k] \bigr|  \\
&\leq \frac{P_k}{\overline{P}_{k+m+1}} \Big[ C (k+m+1) \bigl| \log \delta_k \bigr|  \bigl( \overline{q}_{k+m+1}[\mathbf i_k] \bigr)^{\frac{1}{2}} \Bigr] \\ 
&\leq C (k+m+1) \frac{R_k}{R_{k+m+1}} \; \bigl| \log \delta_k \bigr| \; \Biggl[  \Bigl( \frac{\delta_k}{\delta_{k+m+1}} \Bigr)^d \times \Bigl( \prod_{j=k+1}^{k+m+1} p_j \Bigr) \Biggr] ^{\frac{1}{2}} \\
&\leq C (k+m+1) \; \bigl| \log \delta_k \bigr| \;  \Bigl( \frac{R_k}{R_{k+m+1}} \Bigr)^{\frac{1}{2}}.
\end{align*}  
Here we have invoked \eqref{qlql-bar} from Corollary \ref{q-difference-corollary} at the second step, and the estimates \eqref{PkRk}, \eqref{ql-est} at the third step. The last quantity, after summing in $m$, goes to zero as $k \rightarrow \infty$ by our hypothesis \eqref{msum}, completing the proof. 
\end{proof} 

\section{Appendix: Summability criteria} \label{appendix-summability}
\begin{lemma} 
The choice of parameters \eqref{dim1} and \eqref{dim-epsilon} obey the summability conditions \eqref{nonemptyE-condition}, \eqref{another-sum}, \eqref{third-sum} and \eqref{msum}.  
\end{lemma} 
\begin{proof} 
With $N_k$ and $\epsilon_k$ as in \eqref{dim1} and \eqref{dim-epsilon}, we find that 
\[ p_k = N_k^{-d \epsilon_k} = \begin{cases} N^{-d \gamma} &\text{ for } \eqref{dim1}, \\ N^{-kd \epsilon} &\text{ for } \eqref{dim-epsilon}.  \end{cases} \]  
This leads to the estimate 
\[ (1 - p_k)^{N_k^d} = \exp \left[ N_k^d \log(1 - p_k) \right] \leq e^{-N_k^d p_k} \leq \begin{cases} e^{-N^{d(k-\gamma)}} &\text{ for } \eqref{dim1}, \\ e^{-N^{kd(1 - \epsilon)}} &\text{ for } \eqref{dim-epsilon},   \end{cases}   \] 
which verifies \eqref{nonemptyE-condition} for large $N$. A direct substitution shows that 
\[ N_k^{-d(1 - \epsilon_k)/2}  =   \begin{cases} N^{-\frac{kd}{2}}  N^{\frac{d \gamma}{2}} &\text{ for \eqref{dim1}}, \\   N^{-\frac{kd}{2}(1 - \epsilon)} &\text{ for \eqref{dim-epsilon}}, \end{cases} \]
from which \eqref{another-sum} follows. In either of the two cases \eqref{dim1} and \eqref{dim-epsilon}, $N_k^{-d(1-\epsilon_k)/2}$ is bounded above by $N^{-ckd}$ for some small $c > 0$. This estimate, along with $\delta_k^{-1} = N^{k(k+1)/2}$, also gives 
\[ \sum_{k > k'} k \bigl| \log \delta_{k'}\bigr| N_k^{-d(1 - \epsilon_k)/2} \leq \sum_{k} k^2 \bigl| \log \delta_{k} \bigr| N^{-c kd} \leq \sum_k k^4 (\log N) N^{-c kd} < \infty, \] 
which establishes \eqref{third-sum}. It remains to check \eqref{msum}. For this, we use the expressions for $R_k$ derived in \eqref{def-Rk}, and note that both $R_k$ and $R_{k}/R_{k-1}$  are non-decreasing functions of $k$, with 
\[ \frac{R_{k}}{R_{k-1}} = \begin{cases} N^{dk - d\gamma} &\text{ for } \eqref{dim1}, \\ N^{d(1-\epsilon)k} &\text{ for } \eqref{dim-epsilon}, \end{cases} \; \text{ so } \; \frac{R_{k}}{R_{k-1}} \geq N^{ck} \text{ for some constant $c > 0$ in either case.}   \]  
Substituting this into the sum in \eqref{msum} yields
\begin{align*}  
\sum_{m=1}^{\infty} (k+m) R_{k+m}^{-\frac{1}{2}} &\leq  C k R_{k+1}^{-\frac{1}{2}} \sum_{m=1}^{\infty} m R_{k+1}^{\frac{1}{2}} R_{k+m}^{-\frac{1}{2}} \\
&\leq C k R_{k+1}^{-\frac{1}{2}}  \Bigl[ 1  + \sum_{m=2}^{\infty} m R_{k+m-1}^{\frac{1}{2}} R_{k+m}^{-\frac{1}{2}} \Bigr] \\ 
&\leq Ck  R_{k+1}^{-\frac{1}{2}} \Bigl[ 1 + \sum_m m N^{-c(k+m)}\Bigr] \leq Ck  R_{k+1}^{-\frac{1}{2}}. 
\end{align*}  
Since $k|\log \delta_k| R_k^{\frac{1}{2}} R_{k+1}^{-\frac{1}{2}} \leq C k^3 N^{-c k} \rightarrow 0$, the condition \eqref{msum} is verified. 
\end{proof} 
\section{Acknowledgements} 
\noindent The authors are grateful to Melissa Tacy for her careful reading of an earlier version of the manuscript that led to a simplification of the argument, and two anonymous referees for their substantive feedback that greatly improved the presentation of the main results. SE warmly thanks the Department of Mathematics at McGill University for its hospitality and working environment during a sabbatical period, when this work was partially done. 
A Heilbronn Institute 10th Anniversary Grant funded SE during the initial stages of this project in 2016. MP was supported by a 2017 MSRI research membership at a thematic program in harmonic analysis, a 2018 Wall Scholarship from the Peter Wall Institute of Advanced Study, a 2019 Simons Fellowship and two NSERC Discovery grants.

}
\end{document}